\documentclass{amsart}

\usepackage{amsrefs,amsthm,amsmath,amsfonts,amssymb,graphicx,latexsym,color}
\usepackage{hyperref}

\newtheorem{theorem}{Theorem}

\newtheorem{lemma}[theorem]{Lemma}

\newtheorem{proposition}{Proposition}

\theoremstyle{definition}

\newcommand{\eqdef}{\overset{\mbox{\tiny{def}}}{=}}

\newcommand{\R}{{\mathbb R}}
\newcommand{\ktot}{K(n)}
\newcommand{\NgE}{{K \ge \ktot}}
\newcommand{\ang}[1]{ \left< {#1} \right> }
\newcommand{\ext}[1]{ \underline{ {#1} } }
\newcommand{\ind}{ {\mathbf 1}}

\newcommand{\sph}{{\mathbb S}^{n-1}}
\newcommand{\nsm}{|}
\newcommand{\nl}{\left|}
\newcommand{\threed}{{\mathbb R}^n}
\newcommand{\last}{{n+1}}
\renewcommand{\dim}{n}
\newcommand{\nspace}{N^{s,\gamma}_{\ell,K}}

\newcommand{\spaceh}{H^{K}_{\ell}}
\newcommand{\ksob}{K^*_n}

\newcommand{\nr}{\right|}
\newcommand{\set}[2]{ \left\{ #1 \ \left| \ #2 \right. \right\} }
\newcommand{\ip}[2]{ \left< #1 , #2 \right>}

\begin{document}
\title[Global Solutions of the Boltzmann Equation with Soft Potentials]{Global Classical Solutions of the Boltzmann Equation with Long-Range Interactions and Soft Potentials}

\author[P. T. Gressman]{Philip T. Gressman}
\address{University of Pennsylvania, Department of Mathematics, David Rittenhouse Lab, 209 South 33rd Street, Philadelphia, PA 19104-6395, USA} 
\email{gressman at math.upenn.edu \& strain at math.upenn.edu}
\urladdr{http://www.math.upenn.edu/~gressman/ \& http://www.math.upenn.edu/~strain/}
\thanks{P.T.G. was partially supported by the NSF grant DMS-0850791.}

\author[R. M. Strain]{Robert M. Strain}
\thanks{R.M.S. was partially supported by the NSF grant DMS-0901463.}

\dedicatory{}


\date{February 15, 2010} 

\begin{abstract}
In this work we prove global stability for the Boltzmann equation (1872) with the physical collision kernels derived by Maxwell in 1866 for the full range of inverse power intermolecular potentials, $r^{-(p-1)}$ with $p>2$.
This completes the work which we began in
\cite{gsNonCut1}
(and announced in 
\cite{gsNonCutA}).  
We  more generally cover 
collision kernels with parameters $s\in (0,1)$ and $\gamma$ satisfying 
$\gamma  > -(n-2)-2s$ in arbitrary dimensions $\mathbb{T}^n \times \mathbb{R}^n$ with $n\ge 2$.  
Moreover, we prove rapid convergence as predicted by the Boltzmann H-Theorem.
When $\gamma + 2s \ge 0$, we have  exponential time decay to the Maxwellian equilibrium states.  When $\gamma + 2s < 0$, our solutions decay polynomially fast in time with any rate.  
These results are constructive.  
Additionally, we prove  sharp constructive upper and lower bounds for the linearized collision operator in terms of a geometric fractional Sobolev norm; we thus observe that a spectral gap exists only when $\gamma + 2s \ge 0$, as conjectured in  Mouhot-Strain 
\cite{MR2322149}.
\end{abstract}

\maketitle

\thispagestyle{empty}

\tableofcontents  

\section{Introduction, main results, and overall strategy}
\label{nrstat}

The present article 
considers the Boltzmann equation \cite{MR0158708} from 1872 with the collision kernels derived by Maxwell \cite{Maxwell1867} in 1866 for the intermolecular repulsive potential $\phi(r) = r^{-(p-1)}$ over the full range  $p\in (2,\infty)$. 
It will be observed that
 this  fundamental equation, derived by both Boltzmann  and Maxwell, grants a basic example where a range of geometric fractional derivatives occur in a physical model of the natural world.

This work contributes to the understanding of global in time, close to Maxwellian equilibrium solutions of the Boltzmann equation without angular cut-off, that is, for long-range interactions. This problem has been the subject of intense investigations for some time now. 
We develop  a mathematical framework 
for these solutions  for all of the collision kernels derived from the intermolecular potentials 
 and more generally.  For the hard-sphere and cut-off collision kernels, such a framework has been well established for a long time 
\cite{MR2095473,MR882376,MR0363332,MR2013332,MR2043729,MR2000470}.  The hard-sphere kernel applies in the limit when $p\to\infty$.
A framework is also known for the Landau collision operator \cite{MR1946444}, which can be thought of as  the limiting model for $p=2$. 
The important case of the Boltzmann equation with the collision kernels derived by Maxwell is then the last case for the intermolecular potentials in which this framework has remained open.  Our concurrent article  \cite{gsNonCut1}  established such a framework for the easier ``hard potential'' cases when $p>3$.  The current article provides this framework for all of the intermolecular potentials including some of the most interesting and singular ``soft potential'' cases. 
We will discuss more of the  historical background for the Boltzmann equation for long-range interactions in Section \ref{sec:HD}.

The model which is the focus of this research is  the {\em Boltzmann equation}
  \begin{equation}
  \frac{\partial F}{\partial t} + v \cdot \nabla_x F = {\mathcal Q}(F,F),
  \label{BoltzFULL}
  \end{equation}
where the unknown $F(t,x,v)$ is a nonnegative function.  For each time $t\ge 0$, $F(t, \cdot, \cdot)$ represents the density  of particles in phase space, and is often 
called the empirical measure. The spatial coordinates  are $x\in\mathbb{T}^n$, and the velocities are $v\in\mathbb{R}^n$ with $n\ge 2$.
The  {\em Boltzmann collision operator} ${\mathcal Q}$ 
is a bilinear operator which acts only on the velocity variables $v$ and is local in $(t,x)$ as 
  \begin{equation*}
  {\mathcal Q} (F,F)(v) \eqdef 
  \int_{\mathbb{R}^n}  dv_* 
  \int_{\mathbb{S}^{n-1}}  d\sigma~ 
  B(v-v_*, \sigma) \, 
  \big[ F'_* F' - F_* F \big].
  \end{equation*} 
Here we are using the standard shorthand $F = F(v)$, $F_* = F(v_*)$, $F' = F(v')$, 
$F_*^{\prime} = F(v'_*)$. 
In this expression,  $v$, $v_*$ and $v'$, $v' _*$  are 
the velocities of a pair of particles before and after collision.  They are connected through the formulas
  \begin{equation}
  v' = \frac{v+v_*}{2} + \frac{|v-v_*|}{2} \sigma, \qquad
  v'_* = \frac{v+v_*}{2} - \frac{|v-v_*|}{2} \sigma,
  \qquad \sigma \in \mathbb{S}^{n-1}.
  \label{sigma}
  \end{equation}
This representation of ${\mathcal Q}$ results from making a choice for the parameterization of the set of solutions of the physical law of elastic collisions:
\begin{equation}
\begin{array}{rcl}
v+v_* &=& v^{\prime } + v^{\prime }_*,
\\
|v|^2+|v_*|^2 &=& |v^\prime|^2+|v_*^\prime|^2.
\end{array}
\notag
\end{equation}
This choice is not unique, and we also use
 Carleman-type representations later on.

 The {\em Boltzmann collision kernel} $B(v-v_*, \sigma)$ for a monatomic gas is a non-negative function which 
only depends on the {\em relative velocity} $|v-v_*|$ and on
the {\em deviation angle}  $\theta$ through 
$\cos \theta = \ip{ k }{ \sigma}$ where $k = (v-v_*)/|v-v_*|$ and $\langle \cdot, \cdot \rangle$ is the usual scalar product in $\mathbb{R}^n$. 
Without loss of generality, we may assume that  $B(v-v_*, \sigma)$
is supported on $\ip{ k }{ \sigma} \ge 0$, i.e., $0 \le \theta \le \frac{\pi}{2}$.
Otherwise we can  reduce to this situation 
with the following ``symmetrization'': 
$$
\overline{B}(v-v_*, \sigma)
=
\left[ 
B(v-v_*, \sigma)
+
B(v-v_*, -\sigma)
\right]
{\bf 1}_{\ip{ k }{ \sigma} \ge 0}.
$$ 
Above and generally, ${\bf 1}_{A}$ is the usual indicator function of the set $A$.

\subsection*{The Collision Kernel}

Our assumptions are as follows:  
 \begin{itemize}
 \item 
We suppose that $B(v-v_*, \sigma)$ takes a product form in its arguments as
\begin{equation}
B(v-v_*, \sigma) =\Phi( |v-v_*| ) \, b(\cos \theta).
\notag
\end{equation}
In general, both $b$ and $\Phi$ are non-negative functions. 
 
  \item The angular function $t \mapsto b(t)$ is not locally integrable; for $c_b >0$ it satisfies
\begin{equation}
\frac{c_b}{\theta^{1+2s}} 
\le 
\sin^{n-2} \theta \  b(\cos \theta) 
\le 
\frac{1}{c_b\theta^{1+2s}},
\quad
s \in (0,1),
\quad
   \forall \, \theta \in \left(0,\frac{\pi}{2} \right].
   \label{kernelQ}
\end{equation}
   Some authors use the notation 
$\nu = 2s \in (0,2)$, which is equivalent. 
 
  \item The kinetic factor $z \mapsto \Phi(|z|)$ satisfies for some $C_\Phi >0$
\begin{equation}
\Phi( |v-v_*| ) =  C_\Phi  |v-v_*|^\gamma , \quad \gamma  > -(n-2)-2s.
\label{kernelP}
\end{equation}
 \end{itemize}

Our main physical motivation is derived from  particles interacting according to a spherical intermolecular repulsive potential of the form
$$
\phi(r)=r^{-(p-1)} , \quad p \in (2,+\infty).
$$
For these potentials, Maxwell \cite{Maxwell1867} in 1866 showed that 
the kernel $B$ can be computed.   In dimension $n=3$, $B$ satisfies the conditions above with 
$\gamma = (p-5)/(p-1)$ 
and $s = 1/(p-1)$; 
see for instance \cite{MR1313028,MR1307620,MR1942465}.
Thus the conditions in  \eqref{kernelQ} and \eqref{kernelP} include all of the potentials $p > 2$ in the physical dimension $n=3$.
Note further that the Boltzmann collision operator is not well defined for $p=2$, see \cite{MR1942465}.

We will study the linearization of \eqref{BoltzFULL} around the Maxwellian 
equilibrium state 
\begin{equation}
F(t,x,v) = \mu(v)+\sqrt{\mu(v)} f(t,x,v),
\label{maxLIN}
\end{equation}
where without loss of generality for $n\ge 2$ we have
$$
\mu(v) = (2\pi)^{-n/2}e^{-|v|^2/2}.
$$
We will also suppose without restriction that the mass, momentum, and energy conservation laws for the perturbation $f(t,x,v)$ hold for all $t\ge0$ as
\begin{equation}
\int_{\mathbb{T}^n} ~ dx ~ 
\int_{\mathbb{R}^n} ~ dv ~
 \begin{pmatrix}
      1   \\      v  \\ |v|^2
\end{pmatrix}
~ \sqrt{\mu(v)} ~ f(t,x,v) 
=
0.
\label{conservation}
\end{equation}
This condition should be satisfied initially, and then will continue to be satisfied for a suitable classical solution.  Our main interest is the  global stability of unique classical solutions to the Boltzmann equation \eqref{BoltzFULL} for the long-range collision kernels  \eqref{kernelQ} and \eqref{kernelP}  when the initial data are perturbations of a global Maxwellian  \eqref{maxLIN}.  

Our solution to this problem rests heavily on our introduction into the Boltzmann theory of the 
following weighted geometric fractional  Sobolev space:
$$
N^{s,\gamma}
\eqdef 
\left\{ f \in \mathcal{S}'(\mathbb{R}^n): 
 \nsm f\nsm_{N^{s,\gamma}} <\infty
\right\},
$$
where  we specify the  norm 
by
\begin{equation} 
\nsm f\nsm_{N^{s,\gamma}}^2 
\eqdef 
\nsm f\nsm_{L^2_{\gamma+2s}}^2 + \int_{\mathbb{R}^n} dv ~ \int_{\mathbb{R}^n} dv' ~ 
(\ang{v}\ang{v'})^{\frac{\gamma+2s+1}{2}}
~
 \frac{(f(v') - f(v))^2}{d(v,v')^{n+2s}} 
\ind_{d(v,v') \leq 1}. \label{normdef} 
\end{equation}
This space includes the weighted $L^2_\ell$ space, for $\ell\in\mathbb{R}$, with norm given by
$$
\nsm f\nsm_{L^2_{\ell}}^2 
\eqdef
\int_{\mathbb{R}^n} dv ~ 
\ang{v}^{\ell}
~
|f(v)|^2.
$$
The weight is
$
\ang{v}
= \sqrt{1+|v|^2}.
$
The fractional differentiation effects are measured 
using the following non-isotropic  metric $d(v,v')$ on the ``lifted'' paraboloid:
$$
d(v,v') \eqdef \sqrt{ |v-v'|^2 + \frac{1}{4}\left( |v|^2 -  |v'|^2\right)^2}.
$$
The inclusion of the quadratic difference $|v|^2 - |v'|^2$ is essential; it is not a lower order term.
Heuristically, this metric encodes the non-isotropic changes in the power of the weight, which are 
entangled with the non-local fractional differentiation effects.

The space $N^{s,\gamma}$ is essentially a weighted non-isotropic Sobolev norm; this particular feature was conjectured in \cite{MR2322149}.  
We see precisely that if $\mathbb{R}^n$ is identified with a paraboloid in $\mathbb{R}^{n+1}$ by means of the mapping $v \mapsto (v, \frac{1}{2}|v|^2)$ and $\Delta_{P}$ is defined to be the Laplacian on the paraboloid induced by the Euclidean metric on $\mathbb{R}^{n+1}$ then
\[ 
\nsm f \nsm_{N^{s,\gamma}}^2 \approx \int_{\mathbb{R}^n} dv \ang{v}^{\gamma+2s} \left| (I- \Delta_{P})^{\frac{s}{2}} f (v) \right|^2. 
\]
This holds
in the sense that the sharp geometric Littlewood-Paley characterization of this space is the same as that of $N^{s,\gamma}$.
The rest of our Sobolev spaces are defined in Section \ref{sec:FuncSp} just below. 
We may now to state our main result as follows:

\begin{theorem}
\label{mainGLOBAL}  
(Main Theorem)
Fix $\NgE$, the total number of derivatives, and $\ell\ge 0$ the order of the velocity weight in our Sobolev spaces.
Choose initially $f_0(x,v) \in \spaceh(\mathbb{T}^n \times \mathbb{R}^n)$ in \eqref{maxLIN} which satisfies \eqref{conservation}.  
There is an $\eta_0>0$ such that if $\| f_0 \|_{\spaceh} \le \eta_0$, then there exists a unique global classical solution to the Boltzmann equation \eqref{BoltzFULL}, in the form 
\eqref{maxLIN}, which satisfies
$$
f(t,x,v) \in 
L^\infty_t \spaceh  ( [0,\infty)\times\mathbb{T}^n\times \mathbb{R}^n )
\cap
L^2_t \nspace( (0,\infty)\times\mathbb{T}^n\times \mathbb{R}^n ).
$$
When $\gamma + 2s \ge 0$,
for a fixed $\lambda >0$, 
we have exponential decay as
$$
\| f (t)  \|_{\spaceh (\mathbb{T}^n\times \mathbb{R}^n)}
\lesssim e^{-\lambda t} 
\| f_0 \|_{\spaceh (\mathbb{T}^n\times \mathbb{R}^n)}.
$$
When $\gamma + 2s < 0$, if 
$f_0  \in H^K_{\ell+m}(\mathbb{T}^n \times \mathbb{R}^n)$ for weights $\ell,m \ge 0$ then
 we have 
$$
\| f(t) \|_{\spaceh(\mathbb{T}^n \times \mathbb{R}^n)} \le C_m (1+t)^{-m}
\| f_0 \|_{H^K_{\ell+m}(\mathbb{T}^n \times \mathbb{R}^n)}.
$$
 We also have positivity, i.e. $F= \mu + \sqrt{\mu} f \ge 0$ if $F_0= \mu + \sqrt{\mu} f_0 \ge 0$.
\end{theorem}

For the derivatives used above, 
in dimension $n=3$,  we use $K(3) = 5$.  Generally $\ktot$ is the number of derivatives  required by our use of the Sobolev embedding theorems in $n$ dimensions.  Here $\ktot = 3\ksob - 1$ 
 is sufficient, where $\ksob = \lfloor \frac{n}{2} +1 \rfloor$ is the smallest integer which is strictly greater than $\frac{n}{2}$.
 This regularity is not needed for the hard potentials as the theorem below shows:

\begin{theorem}
\label{SG1mainGLOBAL}  
(Main Theorem from \cite{gsNonCut1})
Fix $K_x \ge 3$, the total number of spatial derivatives, and $K_v\ge 0$ the total number of velocity derivatives.  Suppose $\gamma + 2s \ge 0$.
Choose initially $f_0(x,v) \in H^{K_v}(\mathbb{R}^3; H^{K_x}(\mathbb{T}^3))$ in \eqref{maxLIN} which satisfies \eqref{conservation}.  
There is an $\eta_0>0$ such that if $\| f_0 \|_{H^{K_v}(\mathbb{R}^3; H^{K_x}(\mathbb{T}^3))} \le \eta_0$, then there exists a unique global strong solution to the Boltzmann equation \eqref{BoltzFULL}, in the form 
\eqref{maxLIN}, which satisfies
$$
f(t,x,v) \in 
L^\infty_t( [0,\infty); H^{K_v} H^{K_x} (\mathbb{T}^3\times \mathbb{R}^3 ))
\cap
L^2_t( (0,\infty); N^{s,\gamma}_{\emptyset,K_v}  H^{K_x}(\mathbb{T}^3\times \mathbb{R}^3 )).
$$
Moreover, we have exponential decay 
to equilibrium.  For some fixed $\lambda >0$,
$$
\| f(t) \|_{H^{K_v}(\mathbb{R}^3; H^{K_x}(\mathbb{T}^3))} \lesssim e^{-\lambda t} 
\| f_0 \|_{H^{K_v}(\mathbb{R}^3; H^{K_x}(\mathbb{T}^3))}.
$$
 We also have positivity, i.e. $F= \mu + \sqrt{\mu} f \ge 0$ if $F_0= \mu + \sqrt{\mu} f_0 \ge 0$.
\end{theorem}

   We do not attempt  to optimize the amount of regularity needed in the above theorems, although  
   a  trick used just after Lemma \ref{NonLinEstA} allows us to reduce the standard regularity conditions (as required for instance by Guo \cite{MR1946444,MR2013332} for the cut-off soft potential and the Landau equation) from $K(3) = 8$ down to $K(3) = 5$, etc. 
   
Before we discuss  these results we will now define the relevant Sobolev spaces.

\subsection{Function spaces}\label{sec:FuncSp}
We first define a unified weight function
$$
w(v) 
\eqdef 
\left\{
\begin{array}{ccc}
\ang{v}, &\gamma + 2s \ge 0, & \text{``hard potentials''}\\
\ang{v}^{-\gamma - 2s}, & \gamma + 2s < 0, & \text{``soft potentials''}.\\
\end{array}
\right.
$$
This somewhat non-standard terminology distinguishes exactly when a spectral gap exists for the linearized collision operator.
In both cases  the weight goes to infinity as $|v|$ goes to infinity.  The scaling chosen above will be convenient for the soft potentials in our analysis below.

We will use the multi-indices $\alpha = (\alpha^1, \alpha^2, \ldots, \alpha^n)$ and $\beta = (\beta^1, \beta^2, \ldots, \beta^n)$. 
Inequalities for multi-indices are defined as follows: if each component of $\alpha$ is not greater than that of $\alpha_1$,
we write $\alpha \le \alpha_1$.  Also $\alpha <\alpha_1$ means 
$\alpha \le \alpha_1$ and $|\alpha |<|\alpha_1|$, 
where
$|\alpha | = \alpha^1 + \alpha^2 + \cdots + \alpha^n$ as usual.
We will use the spatial and velocity derivatives
$$
\partial^{\alpha}_{\beta}
=
\partial^{\alpha^1}_{x_1} 
\partial^{\alpha^2}_{x_2} \cdots
\partial^{\alpha^n}_{x_n}
\partial^{\beta^1}_{v_1} 
\partial^{\beta^2}_{v_2} \cdots 
\partial^{\beta^n}_{v_n}. 
$$
For any $\ell \in  \mathbb{R}$, the space
$ H^K_\ell(\mathbb{R}^n) $ with
 $K \ge 0$ velocity derivatives is defined by
 \begin{gather*}
|h|^2_{H^K_\ell}
=
|h|^2_{H^K_{\ell} (\mathbb{R}^n)}
\eqdef 
\sum_{|\beta| \le K}
|w^{{\ell - |\beta|}}\partial_{\beta}h |^2_{L^2 (\mathbb{R}^n)}.
\end{gather*}
The norm notation 
$
|h|^2_{H^K_{\ell,\gamma+2s}}
\eqdef 
\sum_{|\beta| \le K}
|w^{{\ell - |\beta|}}\partial_{\beta}h |^2_{L^2_{\gamma+2s} (\mathbb{R}^n)}
$
will also find utility. 
It will be convenient to define the total $H = H^K_\ell(\mathbb{R}^n \times \mathbb{T}^n)$ space by
 \begin{gather*}
\|h\|^2_{H}
=
\|h\|^2_{H^K_{\ell} (\mathbb{T}^n \times \mathbb{R}^n)}
\eqdef \sum_{|\alpha | + |\beta| \le K}
\|w^{\ell - |\beta|}\partial^{\alpha}_{\beta}h\|^2_{L^2 (\mathbb{T}^n \times \mathbb{R}^n)}.
\end{gather*}
We also consider the general weighted non-isotropic derivative space as in \eqref{normdef} by 
 \begin{gather*}
\nsm h \nsm^2_{{N^{s,\gamma}_{\ell}}}
\eqdef 
\nsm w^{\ell} f\nsm_{L^2_{\gamma+2s}}^2 + \int_{\mathbb{R}^n} dv ~ 
\ang{v}^{\gamma+2s+1} w^{2\ell}(v)
\int_{\mathbb{R}^n} dv' 
~
 \frac{(f(v') - f(v))^2}{d(v,v')^{n+2s}} 
\ind_{d(v,v') \leq 1}.
\end{gather*}
The $L^2(\mathbb{R}^n)$-based norms  exclusively in the velocity variables are always denoted as $| \cdot |$ with the space indicated by the subscript.  The usual $L^2(\mathbb{R}^n)$ inner product is then $\ang{\cdot, \cdot }$, which because of the context should not be confused with our notation for the scalar product.  We will define $L^2(\mathbb{T}^n \times \mathbb{R}^n)$ norms in both space and velocity variables by replacing $| \cdot |$ with $\| \cdot \|$ and keeping the same norm subscript.  We also use $(\cdot, \cdot )$ to be the standard $L^2(\mathbb{T}^n \times \mathbb{R}^n)$ inner product.
In particular,
$
\| h\|_{N^{s,\gamma}}^2
\eqdef
\left\|~ \nsm h\nsm_{N^{s,\gamma}} ~ \right\|_{L^2(\mathbb{T}^n)}^2
$
and
$
\| h \|_{{N^{s,\gamma}_{\ell}}}
\eqdef
\| \nsm h \nsm_{{N^{s,\gamma}_{\ell}}} \|_{L^2(\mathbb{T}^n)}.
$
We will also use the high derivative, non-isotropic space $N=\nspace (\mathbb{T}^n\times\mathbb{R}^n)$, given with $K \ge 0$ by
 \begin{gather*}
\|h\|^2_{N}
=
\|h\|^2_{{N^{s,\gamma}_{\ell, K}}(\mathbb{T}^n\times\mathbb{R}^n)}\eqdef \sum_{|\alpha | + |\beta| \le K}
\|\partial^{\alpha}_{\beta}h\|^2_{{N^{s,\gamma}_{\ell-|\beta|}}(\mathbb{T}^n\times\mathbb{R}^n)}.
\end{gather*}
Furthermore 
$
\nsm h \nsm_{{N^{s,\gamma}_{\ell, K}}}
=
\nsm h \nsm_{{N^{s,\gamma}_{\ell, K}}(\mathbb{R}^n)}
$
is the same norm as above, but only in the velocity variables.
We remark that the decaying weight $w^{-|\beta|}$ on the velocity derivatives is only needed for the soft potentials.  However we may use it for free for the hard potentials and it does not reduce the size of the space.  
Additionally, in Theorem \ref{SG1mainGLOBAL} above
the space $N^{s,\gamma}_{\emptyset,K_v}$ denotes specifically the regular $N^{s,\gamma}$ space with $K_v$ isotropic velocity derivatives and no weights.  
Also in Theorem \ref{SG1mainGLOBAL} 
the space $H^K$ denotes a the usual Sobolev space with $K$ derivatives in either the $x$ variables or the $v$ variables.

\subsection{Historical Discussion}
\label{sec:HD}
For early developments in the Boltzmann equation for long-range interactions, between 1952-1988,
we mention  the work of
Arkeryd, 
Bobylev,
Pao,
Ukai,  
and
Wang Chang-Uhlenbeck-de Boer
in 
\cite{WCUh52,MR1128328,MR0636407,MR630119,MR679196,MR839310}.

Otherwise, 
many mathematical works have required the Grad \cite{MR0156656}  angular cutoff
assumption from 1963--$b(\cos \theta) \in L^\infty(\mathbb{S}^2)$ or sometimes $b(\cos \theta) \in L^1(\mathbb{S}^2)$--which is designed to avoid the non-integrable angular singularity
from \eqref{kernelQ}.  We refer the reader to a brief few breakthrough works of \cite{MR1379589,MR2043729,MR1313028,MR1307620,MR1014927,MR0156656,MR2259206,
MR2095473,MR0475532,MR1284432,MR882376,MR0363332,MR2000470,MR760333,MR2209761,MR2366140,MR2013332,MR2116276,MR1057534}; further references can be found in the review article \cite{MR1942465} and \cite{gsNonCut1}.
Notice that many of these works develop ideas and methods that are fundamental and important even without angular cut-off.  In particular the space-time estimates and general non-linear energy method developed by Guo \cite{MR1946444,MR2013332,MR2000470} and
Strain-Guo \cite{MR2209761,MR2366140} is an important element in the last section of our proof.

Yet with the exception of the hard sphere model, which is derived in the limit when $p\to\infty$, the rest of the inverse
power-law potentials dictate that
 the cross section  $B(v-v_*,\sigma)$ is a non-integrable function in the
angular variable
and the collision operator, $\mathcal{Q}(f,f)$, is then a non-linear diffusive operator.

In particular, in the presence of the physical effects due to the angular singularity, the Boltzmann equation is well-known to experience regularizing effects.
These results  go back to Lions \cite{MR1278244} and Desvillettes \cite{MR1324404}
and have seen substantial developments \cite{MR1407542,
MR1475459,MR1750040,MR1737547,MR1715411,MR2149928,MR2038147}.  
Recently 
Chen-Desvillettes-He \cite{MR2506070}
and also
Alexandre-Morimoto-Ukai-Xu-Yang \cite{MR2462585,arXiv:0909.1229v1} have developed independent  machinery to study these general smoothing effects for kinetic equations.

Contrast this with the case of an angular cut-off, where, as a result of works by Boudin-Desvillettes \cite{MR1798557} in 2000, and additional progress in \cite{MR2435186,MR2476678},
we know that under
the angular cut-off assumption small-data solutions
have the same
Sobolev space regularity as the initial data.  
The results described in these last two paragraphs illustrate that the Boltzmann equation with angular cut-off can be in some respects a very different model from the one without any angular cut-off.

We also mention several works which 
developed and utilized the entropy production (or dissipation)
estimates for the
collision operator
(which grants a non-linear smoothing effect) and also the spatially homogeneous theories.  
We point out 
the work of
 Alexandre, Bouchut, Desvillettes, El Safadi, Fournier, Golse,
Goudon, Gu{\'e}rin, Lions, Mouhot,
Villani and Wennberg
as in
\cite{MR1649477,MR1765272,MR2149928,MR1857879,MR2038147,MR1650006,MR1484062,MR2525118,MR2398952,MR1750040}.
Further references can be found in our concurrent article
\cite{gsNonCut1} and the surveys \cite{MR1942465,krmReview2009}.

Lastly, for the most physically interesting and mathematically challenging case of the
spatially inhomogeneous Boltzmann equation there are much fewer results.
 Here we have two results on local existence  
\cite{MR1851391,arXiv:0909.1229v1}.  We also refer to the paper by Alexandre-Villani \cite{MR1857879}, 
which proves global existence of
 DiPerna-Lions renormalized weak solutions \cite{MR1014927}
 if one includes in the equation a non-negative defect measure.

Furthermore, after the completion of this present work and \cite{gsNonCut1,gsNonCutA},
we were kindly informed by R. Alexandre of 
the recent paper
 \cite{newNonCutAMUXY} by Alexandre-Morimoto-Ukai-Xu-Yang.  They 
prove 
global existence and smoothness for  perturbations of the Maxwellian equilibrium states \eqref{maxLIN} in the whole space $\R^3$ for the  Maxwell molecules collision kernel (meaning that the kinetic factor in \eqref{kernelP} is constant) and moderate angular singularities (meaning that $0< s< 1/2$ in \eqref{kernelQ}); these assumptions  apply to the inverse power intermolecular potentials when $p=5$.
The methods of proof in \cite{newNonCutAMUXY} are substantially different from our own.   

For the Boltzmann collision operator, its essential behavior  has been widely conjectured be the same as a fractional flat diffusion $-(-\Delta_v)^s$.  Precisely 
$$
F \mapsto Q(g,F) \sim -(-\Delta_v)^s F + \mbox{l.o.t.}
$$
Above ``l.o.t.'' indicates that the remaining terms will be lower order.
The original mathematical intuition for this conjecture has been credited to Carlo Cercignani \cite{MR0255199} in 1969, now more than forty years ago (see for instance  Villani \cite[p.91]{MR1942465}).  This conjecture is certainly correct in terms of the smoothing effects, as has been proven for instance in \cite{MR1278244,MR1324404,MR1407542,
MR1475459,MR1750040,MR1737547,MR1715411,MR2506070,MR2149928,MR2462585,arXiv:0909.1229v1}.

Even so, our research proves at the linearized level that the essential behavior of the collision operator is not exactly a flat fractional diffusion.  Instead, the fractional diffusive effects are fundamentally intertwined with non-isotropic velocity weights which go to infinity at infinity.  We prove that this interconnection can be characterized geometrically;  the sharp linear behavior is in fact that of a fundamentally non-isotropic fractional geometric Laplacian, the geometry being given by that of a ``lifted'' paraboloid in $\mathbb{R}^{n+1}$.  
The intuition for this behavior is derived from the original physics representations for the collision operator in terms of $\delta$-functions.

Our next estimate below further establishes that this fundamental non-isotropy, with the same geometry, carries over to the fully non-linear situation in many cases.

\subsection{Entropy Production}
Among Boltzmann's most important contributions to statistical physics was his celebrated $H$-Theorem. 
We define the H-functional by
$$
{H}(t)=-\int_{\mathbb{T}^n}dx~ \int_{\R^n}dv~ F\log F.
$$
Then the Boltzmann H-Theorem predicts that the entropy is increasing over time
$$
\frac{d{H}(t)}{dt} =  \int_{\mathbb{T}^n}dx~ D(F) \ge 0,
$$
which is a manifestation of the second law of thermodynamics.  Here the entropy production functional, which is non-negative, is defined by
\begin{multline}
\notag
D(F)= -\int_{\R^n} ~ dv ~ Q(F,F)\log F
\\
=
\frac{1}{4}
\int_{\R^n}dv \int_{\R^n} dv_* \int_{\sph} d\sigma B(v-v_*, \sigma) \left(F' F'_* - F F_* \right)\log \frac{F' F'_*}{F F_*}.
\end{multline}
Moreover, the entropy production functional is zero if and only if it is operating on a Maxwellian equilibrium.
These formulas formally demonstrate that Boltzmann's equation defines an irreversible dynamics and predicts convergence to Maxwellian in large time.  Of course these predictions are usually non-rigorous  because the regularity  required to perform the above formal calculations is unknown at the moment to be propagated by solutions of the Boltzmann equation in general.

In particular, even though there are many important breakthroughs in this direction, none, so far as we are aware, can be said to completely and rigorously justify the H-Theorem for the inverse power-law intermolecular potentials with 
$p\in(2,\infty)$.
We refer to \cite{MR0158708,MR1014927,MR2116276,MR1942465,MR2209761,MR1857879,MR2366140}, and the references therein, in this regard, to mention only a brief few works.
Our results in this paper
prove rapid convergence to equilibrium for all of the inverse power collision kernels derived by Maxwell when one considers
 initial data which is close to some global Maxwellian. 
 This convergence is 
the essential prediction of Boltzmann's $H$-Theorem.

Many works study entropy production estimates in the non cut-off regime, as in \cite{MR1649477,MR1765272}. 
These estimates have found widespread utility.  
Our results and non-isotropic Sobolev space have the following implications in the fully non-linear context.  Combining estimates in \cite{MR1715411,gsNonCut1} we can see the following lower bound
$$
D(F) \gtrsim
\int_{\mathbb{R}^n} dv ~ \int_{\mathbb{R}^n} dv' ~ 
(\ang{v}\ang{v'})^{\frac{\gamma+2s+1}{2}}
~
 \frac{(\sqrt{F'} - \sqrt{F})^2}{d(v,v')^{n+2s}} 
\ind_{d(v,v') \leq 1} 
- \mbox{l.o.t.}
$$
Precisely, this estimate follows straightforwardly from
our new estimates
\cite[Section 6]{gsNonCut1}
when combined for instance with the decomposition \cite[(29)]{MR1715411}   in the particular case when the unknown functions in the term $A(v,v')$ from \cite[(29)]{MR1715411}
satisfy $F(v) \ge \varepsilon ~ e^{-c |v|^2}$ for any given $\varepsilon, c >0$.  
Remarkably, this is the same semi-norm as in the linearized context, and it is a stronger non-isotropic and non-local version of the local smoothing estimate from \cite{MR1765272} (stronger, that is, in terms of the weight power multiplied on the order of differentiation). 
We have not as of now  attempted to prove this estimate away from the regime where $F$ is bounded below by a Maxwellian.
This estimate was derived as a result of our effort to find explicit equivalence between the non-isotropic norm coming out of the linearized collision operator and the non-isotropic Sobolev space $N^{s,\gamma}$.

Our methods introduce a new tool into the study of the Boltzmann equation with non-integrable angular cross-sections \eqref{kernelQ}
which allows to us to obtain sharp control over geometric fractional differentiation effects with respect to both  the order of differentiation and  the non-isotropic geometrically dependent
weight. We expect this approach to be useful for future studies in the Boltzmann theory.

\subsection*{Acknowledgments} 
Recently Carlo Cercignani has passed away 
in January of 2010,
we wish to take this opportunity to express our deep gratitude and respect for all that he has done to advance the field of kinetic theory.

\subsection{Reformulation}
Throughout this sub-section we will define the relevant notation.  We also reformulate the problem in terms of the equation \eqref{Boltz} for the perturbation \eqref{maxLIN}.  In the next sub-section, we will state and explain our main estimates.

Throughout this paper, the notation $A \lesssim B$ will mean that a positive constant $C$ exists such that $A \leq C B$ holds uniformly over the range of parameters which are present in the inequality (and that the precise magnitude of the constant is irrelevant).  In particular, whenever either $A$ or $B$ involves a function space norm, it will be implicit that the constant is uniform over all the relevant elements of the space unless explicitly stated otherwise.  The notation $B \gtrsim A$ is equivalent to $A \lesssim B$, and $A \approx B$ means that both $A \lesssim B$ and $B \lesssim A$.
Furthermore, for any $R>0$, we will use $B_{R}\subset \mathbb{R}^n$ to be the ball centered at the origin of radius $R$.

We linearize the Boltzmann equation \eqref{BoltzFULL} around the perturbation \eqref{maxLIN}.  This grants an equation for the perturbation $f(t,x,v)$ as
\begin{gather}
 \partial_t f + v\cdot \nabla_x f + L (f)
=
\Gamma (f,f),
\quad
f(0, x, v) = f_0(x,v),
\label{Boltz}
\end{gather}
where the {\it linearized Boltzmann operator} $L$ is given by
\begin{align*}
 L(g)
 \eqdef & 
- \mu^{-1/2}\mathcal{Q}(\mu ,\sqrt{\mu} g)- \mu^{-1/2}\mathcal{Q}(\sqrt{\mu} g,\mu) \\
  = &
  \int_{\mathbb{R}^n}dv_{*}
  \int_{\sph} d\sigma~
  B(|v-v_*|,\cos \theta) \, 
   \left[g_{*} M + g M_{*}- g^{\prime}_{*} M^{\prime} - g^{\prime} M^{\prime}_{*}  \right] M_{*},
\end{align*}
and the bilinear operator $\Gamma$ is given by
\begin{gather}
\Gamma (g,h)
\eqdef
 \mu^{-1/2}\mathcal{Q}(\sqrt{\mu} g,\sqrt{\mu} h)
 =
 \int_{\mathbb{R}^n} dv_* \int_{\sph} d \sigma B M_* (g_*' h' - g_* h). 
\label{gamma0}
\end{gather}
In both definitions, we take
$$
M(v) \eqdef \sqrt{\mu(v)} = (2 \pi)^{-n/4} e^{- |v|^2/4}.
$$
Finally, we note that 
\begin{equation}
L(g) \eqdef - \Gamma(M,g) - \Gamma(g, M).
\label{LinGam} 
\end{equation}
This formulation shows that it is crucial to obtain sharp estimates for  $\Gamma$.

We expand the main term of the linearized Boltzmann operator as
  \begin{equation*}
\Gamma(M,g)
  =
  \int_{\mathbb{R}^n}dv_{*}
  \int_{\sph} d\sigma~
  B(|v-v_*|,\cos \theta) \, 
   \left[M^{\prime}_{*} g^{\prime}  -  M_{*} g \right] M_{*}.
  \end{equation*}
 We will now split this in parts whilst preserving the cancellations as follows:
   \begin{gather*}
\Gamma(M,g)  =
 \int_{\mathbb{R}^n}dv_{*}
  \int_{\sph} d\sigma~
  B \, 
   \left(g^{\prime}- g \right) M^{\prime}_{*}M_{*}
   -
 \tilde{\nu}(v)  ~ g(v),
  \end{gather*}
where
  $$
 \tilde{\nu}(v) 
   =
  \int_{\mathbb{R}^n}dv_{*}
  \int_{\sph} d\sigma~
  B(|v-v_*|,\cos \theta) \, 
 ( M_{*} - M^{\prime}_{*} ) M_{*}.
 $$ 
 The first piece above contains a crucial Hilbert space structure.  This
 can be seen from the pre-post collisional change of variables as
   \begin{multline*}
 - \int_{\mathbb{R}^n}dv
  \int_{\mathbb{R}^n}dv_{*}
  \int_{\sph} d\sigma~
  B(|v-v_*|,\cos \theta) \, 
 (g^{\prime}-g) h M^{\prime}_{*}  M_{*}
 \\
     =
     -
     \frac{1}{2}
  \int_{\mathbb{R}^n}dv
  \int_{\mathbb{R}^n}dv_{*}
  \int_{\sph} d\sigma~
  B \,  
 (g^{\prime}-g) h M^{\prime}_{*}  M_{*}
 \\
 -
      \frac{1}{2}
  \int_{\mathbb{R}^n}dv
  \int_{\mathbb{R}^n}dv_{*}
  \int_{\sph} d\sigma~
  B \, 
 (g-g^{\prime}) h^{\prime} M_{*}  M^{\prime}_{*}
  \\
     =
     \frac{1}{2}
  \int_{\mathbb{R}^n}dv
  \int_{\mathbb{R}^n}dv_{*}
  \int_{\sph} d\sigma~
  B \, 
 (g^{\prime}-g) (h^{\prime}-h) M^{\prime}_{*}  M_{*}.
  \end{multline*}
For the weight, we will use Pao's splitting as
$$
 \tilde{\nu}(v)  =  \nu(v) + \nu_K(v),
$$
where under only   \eqref{kernelQ} and \eqref{kernelP} the following asymptotics are known:
\begin{equation}
 \nu (v)\approx
\ang{v}^{\gamma+2s}
\quad
\text{and}
\quad
\left|  \nu_K(v) \right|
\lesssim
\ang{v}^{\gamma}.
\notag
\end{equation}
These  estimates
were established by
 Pao in \cite[p.568 eq. (65), (66)]{MR0636407} by reducing to the known asymptotic behavior of confluent hypergeometric functions.  

We further decompose  $L=N  + K $.  Here $N$ is the ``norm part'' and $K$ will be seen as the ``compact part.''
The norm part is then written as
\begin{gather}
   \label{normpiece}
 Ng
   \eqdef
   - \Gamma(M,g) - \nu_K(v) g
   =
  -\int_{\mathbb{R}^n}dv_{*}
  \int_{\sph} d\sigma~
  B 
 (g^{\prime}-g) M^{\prime}_{*}  M_{*}
   + \nu(v) g(v).
\end{gather}
From previous calculations, this norm piece satisfies the following identity: 
\begin{gather*}
  \ang{Ng,g} =  \frac{1}{2} \int_{\mathbb{R}^n} dv \int_{\mathbb{R}^n} dv_* \int_{\sph} d \sigma B (g'-g)^2 M_*' M_* 
  + 
  \int_{\mathbb{R}^n} dv ~ \nu(v) ~ |g(v)|^2. 
\end{gather*}
As a result, 
in the following we will use the non-isotropic fractional semi-norm
 \begin{equation}
| g |_{B_\ell}^2 \eqdef 
\frac{1}{2} \int_{\mathbb{R}^n} dv~ w^{\ell}(v)\int_{\mathbb{R}^n} dv_*~ \int_{\sph} d \sigma~ B~ (g'-g)^2 M_*' M_*,
\quad 
\ell \in \mathbb{R}. 
\label{normexpr}
\end{equation}
For the second part of $\ang{Ng,g}$ we recall the norm 
$
\nsm f\nsm_{L^2_{\gamma+2s}}
$
defined below equation \eqref{normdef}. 
These two quantities define our norm, which characterizes precisely
 the linearized Boltzmann H-Theorem.
 
We also record here the definition of the ``compact piece'' $K$:
 \begin{equation}
  K g \eqdef  \nu_K(v) g 
- \Gamma(g, M)
=
  \nu_K(v) g 
-
 \int_{\mathbb{R}^n} dv_* \int_{\sph} d \sigma B M_* (g_*' M' - g_* M).
 \label{compactpiece}
\end{equation}
This is our main splitting of the linearized operator.  We will now state our main estimates for each of these decomposed operators.

\subsection{Main estimates}\label{mainESTsec}
In this sub-section we will 
state most of the crucial long-range estimates to be used in our main results.  
In the next sub-section, we discuss how these estimates in particular resolve a conjecture from Mouhot-Strain \cite{MR2322149}.

We will prove all of our estimates for functions in the Schwartz space, $\mathcal{S}(\mathbb{R}^n)$, which is the well-known space of real valued $C^{\infty}(\mathbb{R}^n)$ functions all of whose derivatives decay at infinity faster than the reciprocal of any polynomial.   
Note that the Schwartz functions are dense in the non-isotropic space $\nspace $, and the proof of this fact is easily reduced to the analogous one for Euclidean Sobolev spaces by means of the 
partition of unity as constructed, for example, in our first paper \cite{gsNonCut1}.
Moreover, in all of our estimates, none of the constants that come up will depend on the 
regularity of the functions that we are estimating.  Thus using routine density arguments, our estimates will apply to any function in $\nspace$ or whatever the appropriate function space happens to be for a particular estimate.

 The first important inequality that we explain is for the linear operator:

\begin{lemma}
\label{sharpLINEAR}
(Linear estimates)
Consider the linearized operator $L = N + K$ where
$N$ is defined in \eqref{normpiece} and $K$ is defined in \eqref{compactpiece}.
We have the uniform inequalities
\begin{align}
\left| \ang{w^{2\ell} N g, g} \right| & \lesssim |g|_{N_{\ell}^{s,\gamma}}^2, \label{normupper} \\
\left| \ang{w^{2\ell} K g, g} \right| & \le \eta |w^\ell g|_{L^2_{\gamma + 2s}}^2 + C_\eta 
|g|_{L^2(B_{C_\eta})}^2, \label{compactupper}
\end{align}
where $\ell \in \mathbb{R}$, $\eta>0$ is any small number, and $C_\eta>0$.
\end{lemma}

In these estimates there are several things to observe.  
First of all, there are no derivatives in the ``compact estimate'' from \eqref{compactupper}, which can be contrasted with the corresponding estimate in the Landau case \cite[Lemma 5]{MR1946444}, in which the upper bound  requires the inclusion of derivatives. Further  \eqref{normupper}  tells us that  the coercive inequality for the ``norm piece'' of the linear term in  Lemma \ref{estNORM3} is essentially sharp.

\begin{lemma}
\label{estNORM3}
(Main coercive inequality)
For the sharp space defined in \eqref{normdef} with \eqref{normpiece}, \eqref{kernelQ} and \eqref{kernelP} we have
  the uniform coercive lower bound estimate:
\begin{equation*}
\ang{ w^{2\ell} N f, f}
 \gtrsim
 \nsm f \nsm_{N^{s,\gamma}_\ell}^2
 -
  C \nsm f \nsm_{L^2(B_C)}^2,
\quad
\exists C \ge 0.
\end{equation*}
This holds   for any $\ell \in\mathbb{R}$; 
if $\ell = 0$ we may take $C = 0$.
\end{lemma}

We turn to our first non-linear estimate:

\begin{lemma}
\label{NonLinEst}
(Non-linear estimate)
Consider the non-linear term \eqref{gamma0}.
For any $|\alpha | + |\beta |\le K$, with $\NgE$ and $\ell \ge 0$,  we have
the following  estimate
\begin{equation*}
 \left| \left(w^{2\ell-2|\beta|}\partial^{\alpha}_{\beta}\Gamma(g, h),\partial^{\alpha}_{\beta}f\right) \right| 
 \lesssim 
 \left(
  \| g\|_{N}    \| h\|_{H}
  +
  \| g\|_{H} \| h\|_{N} \right)
  \| \partial^{\alpha}_{\beta}f\|_{N^{s,\gamma}_{\ell - |\beta|}}.
\end{equation*}
\end{lemma}

Notice that this estimate
  does not require a velocity weight which goes to infinity at infinity.
This feature is made possible by our non-isotropic Littlewood-Paley adapted to the paraboloid, 
which characterizes exactly the geometric fractional differentiation effects that are induced by the linearized Boltzmann collision operator. 

Lastly, we have two coercive interpolation inequalities for the linearized operator:

\begin{lemma}
\label{DerCoerIneq}
(Coercive interpolation inequalities)
For any multi-indices $\alpha,\beta$, any $\ell \ge 0$, and any small $\eta >0$ there is a positive constant $C_\eta$,  such that we have the following coercive lower bound for the linearized collision operator
\begin{equation}
\label{coerc1ineq}
\langle w^{2\ell-2|\beta|} \partial^{\alpha}_{\beta} Lg, \partial^{\alpha}_{\beta} g \rangle
\gtrsim
\nsm \partial^{\alpha}_{\beta}g\nsm_{N^{s,\gamma}_{\ell - |\beta|}}^2     
-
\eta
\sum_{|\beta_1| \le |\beta|}\nsm \partial^{\alpha}_{\beta_1} g\nsm_{N^{s,\gamma}_{\ell - |\beta_1|}}^2  
-
C_\eta     \nsm \partial^\alpha g\nsm_{L^2(B_{C_\eta})}^2.
\end{equation}
Furthermore, when no derivatives are present, for some $C>0$ we have 
\begin{equation}
\label{coerc2ineq}
\langle w^{2\ell}Lg,  g \rangle
\gtrsim
\nsm g\nsm_{N^{s,\gamma}_\ell}^2     
-
C     \nsm  g\nsm_{L^2(B_C)}^2.
\end{equation}
\end{lemma}

This concludes our statements of the main estimates that will be used in  Section \ref{sec:deBEest} to establish our main Theorem \ref{mainGLOBAL}.   Next, we deduce some consequences of these estimates for the spectral properties of the linearized Boltzmann collision operator.

\subsection{Conjecture from Mouhot-Strain \cite{MR2322149}}
We will now discuss  sharp constructive coercivity estimates of the linearized collision operator, $L$, away from its null space.  
More generally, from the H-Theorem  $L$ is non-negative and for every fixed $(t,x)$ the null
space of $L$ is given by the $(n+2)$-dimensional space 
\begin{equation}
{\mathcal N}\eqdef {\rm span}\left\{
\sqrt{\mu},
v_1 \sqrt{\mu}, \ldots, v_n \sqrt{\mu}, 
|v|^2 \sqrt{\mu}\right\}.
 \label{null}
\end{equation}
We define the orthogonal projection from $L^2(\mathbb{R}^n)$ onto the null space ${\mathcal N}$ by ${\bf P}$. 
Further expand ${\bf P} h$ as a linear combination of the basis in \eqref{null}: 
\begin{equation}
{\bf P} h
\eqdef
 \left\{a^h(t,x)+\sum_{j=1}^n b_j^h(t,x)v_j+c^h(t,x)|v|^2\right\}\sqrt{\mu}.
\label{hydro}
\end{equation}
It is well-known for the linearized collision operator \eqref{LinGam} that 
$Lh = 0$ if and only if $h = {\bf P} h$; see the discussion in \cite{gsNonCut1} and the references therein.  
By combining the constructive upper bound estimates in Lemma \ref{sharpLINEAR} with the constructive coercivity estimate from Theorem \ref{lowerN} in Section \ref{sec:deBEest}, we obtain
$$
\frac{1}{C} | \{{\bf I- P} \} g |_{N^{s,\gamma}}^2 \le \langle Lg, g \rangle \le C | \{{\bf I- P} \} g |_{N^{s,\gamma}}^2,
$$
with a constant $C>0$ that can be tracked from the proof.
Thus a spectral gap exists if and only 
if $\gamma + 2 s \ge 0$, as it was 
 conjectured in  \cite{MR2322149}.
 The main tools in our proof of these statements are the new constructive estimates from Lemma \ref{sharpLINEAR} and Lemma \ref{estNORM3} combined with the constructive but non-sharp coercive lower bound from \cite{MR2254617}  for the non-derivative part of the norm.
This may be of independent interest. 
 
Notice that for the linearized Landau collision operator, this statement already had been shown several years earlier in \cite{MR2254617,MR2322149, MR1946444} for any $\gamma \ge -3$.  For Landau, there is a spectral gap if and only if $\gamma +2\ge 0$; the Landau operator can be thought of as the limit case when $s=1$ and  regular (rather than fractional) derivatives are present.

\subsection{The proof of Theorem \ref{SG1mainGLOBAL}.}  We will now explain  how to extend the results of \cite{gsNonCut1} to include the more general assumptions stated in Theorem \ref{SG1mainGLOBAL}.
The main theorem of \cite{gsNonCut1} gives our Theorem \ref{SG1mainGLOBAL} for
$
\gamma > -\min \{ 2s, 3/2\},
$
 $s \in (0,1)$,
and $K_v = 0$
in three dimensions.  Notice that all of the estimates in \cite{gsNonCut1} hold for any $\gamma + 2s \ge 0$ save for the second cancellation estimate in \cite[Proposition 6]{gsNonCut1}.  However, following the remark in the proof of this proposition, we can easily use \eqref{lpsobolev0} with $\ell = K = 0$ to extend this lone proposition to all $\gamma + 2s \ge 0$.  The only additional obstruction is that the proof of the coercive estimate in \cite[Lemma 6]{gsNonCut1} does not directly generalize to $\gamma + 2s =0$.  However notice that Theorem \ref{lowerN} herein is both constructive and holds in full generality of 
$
\gamma + 2s \ge -(n-2),
$
we can thus use this new estimate
in place of \cite[Lemma 6]{gsNonCut1}.  Lastly, notice that one can further include velocity derivatives $K_v \ge 0$ in the proof of the main theorem of \cite{gsNonCut1} and our Theorem \ref{SG1mainGLOBAL} using
\eqref{DerivEst}. All of the proofs in  \cite{gsNonCut1} remain exactly the same when velocity derivatives are included, except for the cancellation estimates.  However the cancellation estimates from \cite{gsNonCut1} can be directly extended to cases with derivatives using the decompositions at the beginning of the proofs of Proposition \ref{cancelFprop} and Proposition \ref{cancelHprop} herein.  The rest of the proof of the the cancellation estimates in \cite{gsNonCut1} with velocity derivatives  are exactly the same.

\subsection{Outline of the article, and overview of our proof}
The methods from our first paper \cite{gsNonCut1} 
are the starting point for our proof.  
In particular, several important concepts, such as the space $N^{s,\gamma}$, the dual formulation for $\ang{\Gamma(g,h),f}$, the coercive inequality for the semi-norm \eqref{normexpr}, and the non-isotropic Littlewood-Paley theory adapted to the paraboloid, will be crucial to our analysis.  However, the full range of kinetic singularities from \eqref{kernelP} dealt with in this paper present several additional difficulties not present in \cite{gsNonCut1} and the resolution of these difficulties 
requires numerous widespread essential changes.
For example, the $L^2(\mathbb{R}^3_v)$ approach from \cite{gsNonCut1}  fails because the kinetic factors from \eqref{kernelP} are too singular.  
Instead,  in this work we include velocity derivatives to control the strongly singular kinetic factors via the Sobolev embedding theorems. 
      The addition of velocity derivatives acting on the Boltzmann operator requires several new developments, particularly in our cancellation estimates, where we provide a new five-term decomposition of the the relevant differences.

Another new theme, in contrast to \cite{gsNonCut1}, involves the inclusion of velocity weights.  Positive velocity weights are used to obtain the rapid decay as in \cite{MR2209761}.  Negative  weights are used to close our estimates, when  $\gamma + 2s <0$,  as in \cite{MR2013332}.  Both of these uses were known in the cut-off regime; however, 
the inclusion of these weights require several new developments 
when angular singularities are present.  A crucial theme in this respect 
will be to extract a sufficiently large amount of velocity decay from the Boltzmann operator, which should be embedded in the functional spaces that form the upper bounds in all of our  estimates.
This is, of course, related to the fact that the norm $N^{s,\gamma}$ from \eqref{normdef}  contains velocity weights which go to zero at infinity when $\gamma + 2s <0$.

The rest of the paper is organized as follows.
In Section \ref{physicalDECrel}, we decompose the angular singularity and prove the main estimates for each of the decomposed pieces; these include size and support estimates, cancellation estimates, and also a useful ``compact estimate'' in the spirit of Grad which is new to the non cut-off theory.  This compact estimate is the key new element which 
 yields our resolution of the conjecture from \cite{MR2322149}.
Then in Section \ref{sec:aniLP} we develop the non-isotropic Littlewood-Paley theory adapted to the paraboloid in arbitrary dimensions $n\ge 2$.  
As before, we have drawn inspiration from the earlier works of Stein \cite{MR0252961} and Klainerman-Rodnianski \cite{MR2221254}, though we are able to use a more concrete formulation of our Littlewood-Paley projections than is possible in either of these cases because of the explicit nature of the underlying geometry.
A key point here, new in comparison to our first paper \cite{gsNonCut1}, is to 
develop Littlewood-Paley projections that have higher regularity than was needed in \cite{gsNonCut1}, which requires a more explicit analysis of the necessary moment conditions imposed on the kernels. 
This is crucial for our high-order isotropic velocity derivative estimates of the non-isotropic Littlewood-Paley projections.
After that, in Section \ref{sec:upTRI} we use the estimates from the previous two sections to sum all of the decomposed pieces of the Boltzmann collision operator from the previous subsections.  We sum over the both decomposed singularity and the unknown functions expanded via the geometric Littlewood-Paley.  In particular, we finish the proof of all of the estimates which are stated in Section \ref{mainESTsec}.
Finally, in Section \ref{sec:deBEest}, we use methods from the non-linear energy method as in 
Guo 
\cite{MR2000470}
and Strain-Guo
\cite{MR2209761} to establish global existence and rapid decay.

\section{Physical decomposition and cancellation estimates}
\label{physicalDECrel}

In this section, we introduce the first major decomposition of the collision operator and prove several estimates which will play a central role in establishing many of the main inequalities stated in Section \ref{mainESTsec}.  This  is a decomposition of the singularity of the collision kernel.  For several reasons, it turns out to be useful to decompose $b ( \cos \theta)$ from \eqref{kernelQ} to regions where $\theta \approx 2^{-j} |v-v_*|^{-1}$, rather than a simpler dyadic decomposition not involving $|v-v_*|$.  The primary benefit for doing so is that this extra factor makes it easier to prove estimates on the space $L^2_{\gamma+2s}(\mathbb{R}^n)$  because the weight $\Phi(|v-v_*|)$ from  \eqref{kernelP}  is already present in the kernel and the extra weight $|v-v_*|^{2s}$ falls out automatically from our decomposition.

The estimates to be proved fall into two main categories:  the first are various $L^2$- and weighted $L^2$-inequalities which will follow from the size and support conditions on our decomposed pieces.   The second type of estimate will assume some sort of smoothness and obtain better estimates than the earlier estimates by exploiting the cancellation structure of the non-linear term $\Gamma$ from \eqref{gamma0}.  It is already worth stating at this point that the particular smoothness assumptions we make are dictated by the problem and will  specifically be somewhat unusual; in particular, they will not correspond to the usual Euclidean Sobolev spaces on $\mathbb{R}^n$.  

In the rest of the paper, we will carefully explain various arguments the first time they are used and subsequently use them again and again with less detail, but with references to their earlier appearances. 
This will save having rewrite similar arguments numerous times. 

\subsection{Dyadic decomposition of the singularity}
Let $\{ \chi_k \}_{k=-\infty}^\infty$ be a partition of unity on $(0,\infty)$ such that $\nsm \chi_k\nsm_{L^\infty} \leq 1$ and 
$\mbox{supp}\left( \chi_k \right)\subset [2^{-k-1},2^{-k}]$.  
For each $k$,
\[
B_k = B_k(v-v_*,\sigma) \eqdef \Phi(|v-v_*|) b \left( \left< \frac{v-v_*}{|v-v_*|}, \sigma \right> \right) \chi_k (|v - v'|). \]
Notice that we 
have the expansion
\[ 
|v-v'|^2 = \frac{|v-v_*|^2}{2} \left( 1 - \left< \frac{v-v_*}{|v-v_*|}, \sigma \right>  \right)
= |v-v_*|^2 \sin^2 \frac{\theta}{2}.
\]
With this partition, 
for any $\ell \in \mathbb{R}$ and any multi-index $\beta$ we define
\begin{equation}
\begin{split}
T^{k,\ell}_{+}(g,h,f)  & \eqdef \int_{\mathbb{R}^n} dv \int_{\mathbb{R}^n} dv_* \int_{\sph} d \sigma  ~B_k(v-v_*, \sigma) ~ g_* h ~ M_\beta(v_*') ~ f'  ~ w^{2\ell}(v'),  \\ 
T^{k,\ell}_{-}(g,h,f)  & \eqdef \int_{\mathbb{R}^n} dv \int_{\mathbb{R}^n} dv_* \int_{\sph} d \sigma ~B_k(v-v_*, \sigma) ~ g_* h ~ M_\beta(v_*) ~ f ~ w^{2\ell}(v). 
\end{split}
\label{defTKL}
\end{equation}
We use the notation
$
M_\beta (v)
=
\partial_{\beta} M = p_\beta(v) M(v)$ where $p_\beta(v)$ is the appropriate polynomial of degree $|\beta|$.
The dependence of these operators $T^{k,\ell}_{\pm}$ on $\beta$ is left implicit for notational simplicity; in particular, all of the estimates below hold for any multi-index $\beta$. 

It will also be necessary to express the collision operator \eqref{gamma0} using a ``dual formulation'' which is derived via a Carleman-type representation in our first paper \cite[Appendix]{gsNonCut1}.  
We record here the following alternative representation for $T^{k,\ell}_{+}$ as well 
as the definition of a third trilinear operator $T^{k,\ell}_{*}$
 as follows:
\begin{gather}
\label{defTKLcarl}
T^{k,\ell}_{+} 
= 
\int_{\mathbb{R}^n} dv' w^{2\ell}(v')  \int_{\mathbb{R}^n} dv_* \int_{E_{v_*}^{v'}} d \pi_{v}  
 \frac{B_k }{2^{1-n}} 
\frac{g_*  f' M_\beta(v_*') h}{  |v'-v_*| ~ |v-v_*|^{n-2}}, 
\\ 
T^{k,\ell}_{*}
=
 \int_{\mathbb{R}^n} dv' w^{2\ell}(v')  \int_{\mathbb{R}^n} dv_* \int_{E_{v_*}^{v'}} d \pi_{v} 
 \frac{B_k }{2^{1-n}} 
 \frac{ \Phi(v'-v_*) |v'-v_*|^{n-1} }{\Phi(v-v_*) |v-v_*|^{2n-2}}   g_* f'  M_\beta(v_*) h'.
\notag
\end{gather}
Above 
$T^{k,\ell}_{*}= T^{k,\ell}_{*}(g,h,f) $, $T^{k,\ell}_{+} = T^{k,\ell}_{+}(g,h,f) $,
and we use the notation 
$$
B_k
=
B_k(v-v_*, 2v' - v- v_*) 
=
\Phi(|v-v_*|)b \left( \ang{\frac{v-v_*}{|v-v_*|}, \frac{2v' - v - v_*}{|2v' - v - v_*|} } \right)\chi_k (|v - v'|).
$$
In these integrals $d\pi_{v}$ is Lebesgue measure on the $(n-1)$-dimensional plane $E_{v_*}^{v'}$ passing through $v'$ with normal $v' - v_*$, and $v$ is the variable of integration.  Moreover, above $v_*' = v+v_* - v'$ and 
$
M_*' 
=
\frac{M M_*}{M'}
$
on $E_{v_*}^{v'}$.  Notice that the insertion of the partition of unity guarantees that the kernel is locally integrable for any $k$.

Recall the bilinear collision operator \eqref{gamma0} and the collisional variables \eqref{sigma}.  With the change of variables $u = v_* - v$, and $u^\pm = (u \pm |u| \sigma)/2$, we may write \eqref{gamma0} as
\begin{gather*}
\Gamma (g,h)
 =
 \int_{\mathbb{R}^n} du \int_{\sph} d \sigma  B(|u|, \sigma)  M(u+v)  \left\{g(v+u^+) h(v+u^-) - g(v+u) h(v)\right\}. 
\end{gather*}
Differentiating this formula and  applying the inverse coordinate change allows us to express derivatives of the bilinear collision operator $\Gamma$ as
\begin{gather}
\label{DerivEst}
\partial^{\alpha}_{\beta}\Gamma (g,h)
 =
 \sum_{\beta_1 + \beta_2  = \beta} 
\sum_{ \alpha_1 \le \alpha} ~ C^{\beta, \beta_1, \beta_2}_{\alpha, \alpha_1} ~ 
\Gamma_{\beta_2} (\partial^{\alpha - \alpha_1}_{\beta - \beta_1}g,\partial^{\alpha_1}_{\beta_1}h).
\end{gather}
Here $C^{\beta, \beta_1, \beta_2}_{\alpha, \alpha_1}$ is a non-negative constant which is derived from the Leibniz rule.
Also, $\Gamma_{\beta}$ is the bilinear operator with derivatives on the Maxwellian $M$ given by
$$
\Gamma_{\beta} (g,h)
=
 \int_{\mathbb{R}^n} dv_* \int_{\sph} d \sigma ~ B(|v-v_*|,\sigma) M_\beta (v_*) (g_*' h' - g_* h).
$$
Then
for $f, g, h \in \mathcal{S}(\mathbb{R}^n)$, the pre-post collisional change of variables, the dual representation \cite[Appendix]{gsNonCut1}, and the previous calculations guarantee that 
\begin{align*}
 \left< w^{2\ell} \Gamma_\beta (g,h), f \right> & = 
\sum_{k=-\infty}^\infty  \left\{ T^{k,\ell}_{+}(g,h,f) - T^{k,\ell}_{-}(g,h,f) \right\}
\\
& = \sum_{k=-\infty}^\infty  \left\{ T^{k,\ell}_{+}(g,h,f) - T^{k,\ell}_{*}(g,h,f) \right\} . 
\end{align*}
These will be the general quantities that we estimate in the following sections.
The first step is to estimate each of  $T^{k,\ell}_{+}$, $T^{k,\ell}_{-}$, and $T^{k,\ell}_{*}$ using only the constraints on the size and support of $B_k$ in Section \ref{sec:SSE}.  Then in Section \ref{sec:SSC} we estimate the collision operator again, this time also exploiting the cancellation properties of $\Gamma_\beta$.
Finally, in Section \ref{sec:SSCE} we prove the ``compact estimates'' which are important in proving the constructive lower bound for the linearized collision operator 
in Theorem \ref{lowerN}.

\subsection{Size and support estimates of the decomposed pieces}
\label{sec:SSE}
We will now prove several size and support estimates for the decomposed pieces of the Boltzmann collision operator.
It will be useful to let $\phi(v)$ denote an arbitrary smooth function which satisfies for some positive constants $C_{\phi}$ and $c$ that
\begin{equation}
\left| 
\phi(v)
\right|
\le C_\phi e^{- c |v|^2}.
\label{rapidDECAYfcn}
\end{equation}
We use generic functions satisfying  \eqref{rapidDECAYfcn} often in what follows.  

\begin{proposition}\label{prop11}
For any integer $k$, any $m \ge 0$ and $\ell \in \mathbb{R}$, we have
\begin{gather}
 \left| T^{k,\ell}_{-}(g,h,f) \right|   \lesssim 2^{2sk} \nsm g\nsm_{H^{\ksob}_{-m}} \nsm w^\ell h\nsm_{L^2_{\gamma + 2s}} 
 \nsm w^\ell f\nsm_{L^2_{\gamma + 2s}},
 \label{tminusg} 
 \\
 \left| T^{k,\ell}_{-}(g,h,f) \right|   \lesssim 2^{2sk} 
 \nsm g\nsm_{L^2_{-m}}  
| h |_{H^{\ksob}_{\ell,\gamma + 2s}} 
| w^\ell f |_{L^2_{\gamma + 2s}}.
 \label{tminush}
\end{gather}
Furthermore, for $\phi$ 
defined as in \eqref{rapidDECAYfcn}, we have
\begin{align}
 \left| T^{k,\ell}_{-}(g,\phi,f) \right| 
 +
  \left| T^{k,\ell}_{-}(g,f,\phi) \right| 
 & 
 \lesssim 
  C_\phi ~ 
 2^{2sk} ~ 
 \nsm g\nsm_{L^2_{-m}}  \nsm f\nsm_{L^2_{-m}}.
 \label{tminushRAP}
\end{align}
These estimates hold uniformly.  
\end{proposition}

We record here the following 
the Sobolev embedding theorem
\begin{equation}
|w^{-m-\ksob} f|_{L^\infty}  \lesssim |f|_{H^{\ksob}_{-m}},
\label{sobolevE}
\end{equation}
which holds for any $m\in \mathbb{R}$.
We will also use the following immediate implications
\begin{equation}
|w^{-m} f|_{L^\infty} + |w^\ell f|_{L^2_{\gamma+2s}} \lesssim |f|_{H^{\ksob}_{\ell,\gamma+2s}},
\quad
|w^{-m} f|_{L^\infty} + |w^\ell f|_{L^2} \lesssim |f|_{H^{\ksob}_{\ell}},
\notag
\end{equation}
where $m$ is sufficiently large, depending on $\ell\in\mathbb{R}$, $\gamma+2s$, and $n\ge 2$.  
Furthermore, we are using the abbreviated notation $L^\infty = L^\infty(\mathbb{R}^n)$.

\begin{proof}
Given the size estimates for $b(\cos \theta)$ in \eqref{kernelQ} and the support of $\chi_k$, clearly
\begin{equation}
\int_{\sph} d \sigma ~ B_k \lesssim
\Phi(|v-v_*|)
\int_{2^{-k-1} |v - v_*|^{-1}}^{2^{-k} |v - v_*|^{-1}}  d \theta ~ \theta^{-1-2s} 
\lesssim 2^{2sk} |v-v_*|^{\gamma+2s}.
\label{bjEST}
\end{equation}
This also uses \eqref{kernelP}.
 Thus
\[ 
\left| T^{k,\ell}_{-}(g,h,f) \right| \lesssim 2^{2sk} \int_{\mathbb{R}^n} dv \int_{\mathbb{R}^n} dv_* ~ \sqrt{M_*} |v-v_*|^{\gamma+2s} |g_* h f|~ w^{2\ell}(v).
 \]
 To obtain 
 \eqref{tminusg}, we
first take the $L^\infty$ norm of $g_* M_*^{1/4}$ and use \eqref{sobolevE}, then apply the elementary inequality
 $
\int_{\mathbb{R}^n} dv_* ~ M_*^{1/4} ~  |v-v_*|^{\gamma+2s}
\lesssim
\langle v \rangle^{\gamma + 2s}
$, and lastly use Cauchy-Schwartz 
putting $h$ in one term, $f$ in the other, and half of the rest in both.  

For \eqref{tminushRAP} notice that if either of the second two functions in $T^{k,\ell}_{-}$  has rapid decay such as \eqref{rapidDECAYfcn}, then we have rapid decay in both $v$ and $v_*$ simultaneously.  Thus we can use Cauchy-Schwartz, putting $g$ in one term, $f$ in the other term, and half of all the rest in both terms.  Thus \eqref{tminushRAP} follows with the elementary inequality that we just used in the previous estimate.

To obtain \eqref{tminush}, on the other hand, we will bound above the r.h.s. of the last displayed inequality by 
splitting into the regions 
$
|v-v_*| \le 1
$
and
$
|v-v_*| \ge 1.
$
On the former region 
$
M_*^{1/4} \lesssim \ang{v_*}^{-m} \ang{v}^{-2m-\ksob},
$
$
\forall m >0$
so we may take the $L^\infty$ norm of $|h| \ang{v}^{-m-\ksob}$, use \eqref{sobolevE}, and apply Cauchy-Schwartz as above to obtain
 \[ 
\lesssim 2^{2sk} | h |_{H^{\ksob}_{-m}} | g |_{L^2_{-m}} | f |_{L^2_{-m}},
\quad
\forall m >0.
 \]
On the latter region,
$
|v-v_*| \ge 1,
$
we have $M_*^{1/4} |v-v_*|^{\gamma+2s} \lesssim M_*^{1/8} \ang{v_*}^{-m} \ang{v}^{\gamma+2s}$.  
Applying Cauchy-Schwartz as above in both the $v$ and $v_*$ integrals gives
 \[ 
\lesssim 2^{2sk} | w^\ell h |_{L^2_{\gamma + 2s}} | g |_{L^2_{-m}} | w^\ell f |_{L^2_{\gamma + 2s}},
\quad
\forall m >0.
 \]
Summing these two inequalities  gives \eqref{tminush}.
\end{proof}

\begin{proposition}\label{starPROP}
For all $\ell\in\mathbb{R}$, $m\ge 0$ and integers $k$, the inequalities are uniform:
\begin{gather}
  \left| T^{k,\ell}_{*}(g,h,f) \right|   \lesssim 
  2^{2sk}    \nsm g\nsm_{H^{\ksob}_{-m}} 
  \nsm w^\ell h\nsm_{L^2_{\gamma + 2s}} 
 \nsm w^\ell f\nsm_{L^2_{\gamma + 2s}},
 \label{tstarg} 
 \\
  \left| T^{k,\ell}_{*}(g,h,f) \right|   \lesssim 2^{2sk} 
 \nsm g\nsm_{L^2_{-m}} 
| w^\ell h |_{H^{\ksob}_{\ell,\gamma + 2s}} 
| w^\ell f |_{L^2_{\gamma + 2s}}.
 \label{tstarh}
\end{gather}
Moreover,
\begin{align}
 \left| T^{k,\ell}_{*}(g,\phi,f) \right| 
 +
  \left| T^{k,\ell}_{*}(g,f,\phi) \right| 
 &  \lesssim 2^{2sk} 
 \nsm g\nsm_{L^2_{-m}}  \nsm f\nsm_{L^2_{-m}},
 \label{tstarhRAP}
\end{align}
where as usual $\phi$ is defined as in \eqref{rapidDECAYfcn}.  
\end{proposition}

\begin{proof}
As in the previous Proposition \ref{prop11}, the key to these inequalities is the symmetry between $h$ and $f$ coupled with two applications of Cauchy-Schwartz.  
The difference is that, this time, the Carleman representation \eqref{defTKLcarl} will be used and the main integrals will be over $v_*$ and $v'$.  In this case, the quantity of interest is
\[ 
\int_{E_{v_*}^{v'}}  d \pi_{v} ~ b \left( \frac{|v'-v_*|^2 - |v - v'|^2}{|v' - v_*|^2 + |v - v'|^2} \right) \frac{|v'-v_*|^{n-1}}{|v-v_*|^{2n-2}} \chi_k(|v-v'|). 
\]
The support condition yields $|v-v'| \approx 2^{-k}$.  Moreover, since $b(\cos \theta)$ vanishes for $\theta\in [\pi/2,\pi]$, we have
$|v' - v_*| \ge |v' - v|$. 
Consequently, the condition \eqref{kernelQ} gives 
\begin{equation}
b \left( \frac{|v'-v_*|^2 - |v - v'|^2}{|v' - v_*|^2 + |v - v'|^2} \right)
\lesssim
\left(\frac{ |v - v'|^2}{|v' - v_*|^2} \right)^{-\frac{n-1}{2}-s}.
\notag
\end{equation}
Thus, the integral is bounded by a uniform constant times
\[ 
\int_{E_{v_*}^{v'}}  d \pi_{v} ~\frac{|v'-v_*|^{n-1+2s}}{|v-v'|^{n-1+2s}} |v'-v_*|^{-n+1} \chi_k(|v-v'|) \lesssim 2^{2sk} |v'-v_*|^{2s}. 
\]
As a result,
\[ 
\left| T^{k,\ell}_*(g,h,f) \right| \lesssim 2^{2sk} \int_{\mathbb{R}^n} dv' \int_{\mathbb{R}^n} dv_*  ~ \sqrt{M_*}  ~ |v' - v_*|^{\gamma+2s} |g_* h' f'| ~ w^{2\ell}(v'). 
\]
Now the relevant estimates can be established as in Proposition \ref{prop11}.
\end{proof}

We will now estimate the operator $T^{k,\ell}_{+}$, which is more difficult and more technical because it contains the post-collisional velocities \eqref{sigma}.
A key problem here is to be able to distribute negative decaying velocity weights among the functions $g$, $h$, and $f$.  In the previous propositions this distribution could be accomplished more easily. 
Because the weight is in the $v$ variable and the unknown functions are of the variables $v'$ and $v'_*$ we have to work harder to distribute the negative weights.
These decaying weights are used to close our estimates for the soft potentials.    In the following we generalize methods from \cite{MR2013332,MR1946444} to overcome this particular difficulty.

\begin{proposition}
\label{referLATERprop}
Fix an integer $k$ and $\ell^+, \ell^- \ge 0$,
with $\ell = \ell^+ - \ell^-$.  For any  $0\le \ell' \le \ell^-$,
with $\ell + \ell' = \ell^+ - (\ell^- - \ell')$,
  we   have the uniform estimates 
\begin{equation} 
\begin{split}
\left|  T^{k,\ell}_{+}(g,h,f)  \right| \lesssim   & ~
2^{2sk} 
 \nsm  g\nsm_{H^{\ksob}_{{\ell^+} - \ell'}} 
\nsm w^{\ell + \ell'} h\nsm_{L^2_{\gamma + 2s}} 
\nsm w^{\ell} f\nsm_{L^2_{\gamma + 2s}}
\\ & ~
  +
  2^{2sk}\nsm w^{\ell^+ - \ell'} g\nsm_{L^2_{\gamma + 2s}} \nsm w^{\ell + \ell'} h\nsm_{L^2} 
  \nsm w^{\ell} f\nsm_{L^2_{\gamma + 2s}}.
\end{split}
 \label{tplussmall}
\end{equation}
We also have a similar 
estimate with the roles of $g$ and $h$ reversed
\begin{equation} 
\begin{split}
\left|  T^{k,\ell}_{+}(g,h,f)  \right| \lesssim   & ~
2^{2sk} 
\nsm w^{\ell^+ - \ell'} g\nsm_{L^2} 
 \nsm  h\nsm_{H^{\ksob}_{\ell + \ell', \gamma + 2s} }
\nsm w^{\ell} f\nsm_{L^2_{\gamma + 2s}}
\\ & ~
  +
  2^{2sk}\nsm w^{\ell^+ - \ell'} g\nsm_{L^2_{\gamma + 2s}} \nsm w^{\ell + \ell'} h\nsm_{L^2} 
  \nsm w^{\ell} f\nsm_{L^2_{\gamma + 2s}}.
\end{split}
 \label{tplussmall2}
\end{equation}
\end{proposition}

\begin{proof}
We plan to estimate  $T^{k,\ell}_{+}(g,h,f)$ in \eqref{defTKL} from above as
$$
\left| T^{k,\ell}_{+}(g,h,f)  \right| \lesssim \int_{\mathbb{R}^n} dv \int_{\mathbb{R}^n} dv_* \int_{\sph} d \sigma  ~B_k(v-v_*, \sigma) ~ w^{2\ell}(v) ~  \left| g_*' h' \right|     \left| f \right| \sqrt{M_*}.
$$
We will utilize a number of splittings of this upper bound.  
Firstly, when 
$
|v_*| \ge \frac{1}{2}|v|
$
we have
$
w^{2\ell}(v) M_*^{1/2} \lesssim  ( M_* M )^{\delta}
$
for some $\delta >0$.
Then $\left|  T^{k,\ell}_{+}(g,h,f) \right|$ on this region is bounded above by a uniform constant times
\begin{gather*} 
\int_{\mathbb{R}^n} dv \int_{\mathbb{R}^n} dv_* \int_{\sph} d \sigma ~ B_k(v-v_*, \sigma) ~ |g_*'| |h'|  |f|  ~ 
M_*^\delta M^\delta.
\end{gather*}
Then, with Cauchy-Schwartz, this quantity is again bounded by a uniform constant times
\begin{multline*} 
\left( \int_{\mathbb{R}^n} dv \int_{\mathbb{R}^n} dv_* \int_{\sph} d \sigma ~ B_k(v-v_*, \sigma) ~ |g_*'|^2 |h'|^2 
\left(M_*' M' \right)^{\delta/2}    \right)^{1/2}
\\
\times \left(\int_{\mathbb{R}^n} dv \int_{\mathbb{R}^n} dv_* \int_{\sph} d \sigma ~ B_k(v-v_*, \sigma) ~ \left(M_*M \right)^{\delta/2}    |f|^2 \right)^{1/2}.
\end{multline*}
For the second term, with \eqref{bjEST}, we clearly have for any $m\ge 0$
$$
\int_{\mathbb{R}^n} dv \int_{\mathbb{R}^n} dv_* \int_{\sph} d \sigma ~ B_k(v-v_*, \sigma) ~ \left(M_*M \right)^{\delta/2}  |f|^2
\le 
2^{2sk} \nsm f\nsm_{L^2_{-m}}.
$$
For the first term, do another pre-post and take the $L^\infty$ norm
of $M_*^{\delta/4} g_*$ to obtain that $\left|  T^{k,\ell}_{+}(g,h,f) \right| $, when restricted to this region, with \eqref{sobolevE} is bounded by
\begin{equation*} 
2^{2sk} 
\nsm g\nsm_{H^{\ksob}_{-m}} \nsm h\nsm_{L^2_{-m}} 
\nsm f\nsm_{L^2_{-m}},
\quad \forall m >0
\end{equation*}
(times some uniform constant for any fixed $m$).  If we instead took the $L^\infty$ norm
of $M^{\delta/4} h$ and applied the same procedure, we would have the same inequality except with the roles of $g$ and $h$ reversed.

Our next splitting is onto the region
$
|v_*| \le \frac{1}{2} |v|.
$
Notice that if $|v-v_*| < \frac{1}{2}|v|$ then 
$
|v_*| \ge |v| - |v-v_*| > \frac{1}{2}|v|.
$
Thus $|v-v_*| \ge \frac{1}{2}|v|$ when 
$
|v_*| \le \frac{1}{2} |v|.
$
Further suppose that $|v| \ge 1$. 
On this region,  from \eqref{sigma}, we have, in particular,
$$
|v -v_*| \ge \frac{1}{2} |v| \gtrsim (1+  |v_*|^2 + |v|^2)^{1/2} \gtrsim (1+  |v_*'|^2 + |v'|^2)^{1/2}.
$$ 
This lower bound will be used several times just below.
Then
$\left|  T^{k,\ell}_{+}(g,h,f) \right|$ restricted to this domain is bounded 
 with Cauchy-Schwartz by a uniform constant times
\begin{equation} 
\begin{split} 
\left( \int_{\mathbb{R}^n} dv \int_{\mathbb{R}^n} dv_* \int_{\sph} d \sigma ~ B_k(v-v_*, \sigma) ~ 
w^{2\ell}(v)~
|g_*'|^2 |h'|^2 \sqrt{M_*}  \right)^{1/2}
\\
\times \left(\int_{\mathbb{R}^n} dv \int_{\mathbb{R}^n} dv_* \int_{\sph} d \sigma ~ B_k(v-v_*, \sigma) ~    |f|^2 ~
w^{2\ell}(v)\sqrt{M_*}\right)^{1/2}
\\
 \lesssim 
  2^{2sk} \nsm w^{\ell^+ - \ell'}  g\nsm_{L^2_{\gamma + 2s}} 
  \nsm w^{\ell + \ell'}  h\nsm_{L^2_{\gamma + 2s}} 
  \nsm w^\ell f\nsm_{L^2_{\gamma + 2s}}.
  \end{split} 
  \label{useLaterKest}
\end{equation}
The estimate for the second term involving $f$ follows as in \eqref{bjEST}.
The estimate for the first term involving $g_*' h'$ requires more explanation.  
First of all, on the region of interest using also the collisional geometry \eqref{sigma},
$
w^{2\ell}
=
w^{2\ell^+} w^{-2\ell^-}
$
which is
$$
w^{2\ell}(v)
 \lesssim 
w^{2\ell^+} (v')
  w^{2\ell^+} (v'_*)
\left( 1+ |v'|+|v'_*|\right)^{-2\ell^-}
 \lesssim 
 w^{2(\ell^+-\ell')} (v'_*)
  w^{2(\ell + \ell' )}(v'). 
$$
With that estimate, we take a pre-post collisional change of variables, estimate $M_*'  \lesssim  1$
and 
$
|v -v_*|^{\gamma + 2s} \le \ang{v_*}^{\gamma + 2s}
$
when
 $\gamma + 2s \le 0$
 on the region of interest.
This combined with 
\eqref{bjEST}  yields the estimate above.

Finally when $|v| \le 1$ and 
$
|v_*| \le \frac{1}{2} |v|
$
we have that
$
|v_*|^2 + |v|^2 \le 2.
$
This region is invariant under the pre-post collisional change of variables.  We can thus generate (artificial) Maxwellian type decay in both variables $v_*$ and $v$ since all the variables are trivially bounded.  In this last region, with these preparations, the estimates follow exactly as in the first splitting which was
$
|v_*| > \frac{1}{2} |v|.
$
Thus when $\gamma + 2s \leq 0$, we obtain stronger results than \eqref{tplussmall} and \eqref{tplussmall2} (and may, in fact, omit the second line of terms on the right-hand side for both inequalities).

In the case $\gamma + 2s \ge 0$, we use a particular sequence of Cauchy-Schwartz inequalities which has already found utility in the cut-off context by Guo \cite{MR2000470}:
\begin{multline*}
\int_{\mathbb{R}^n} dv 
\left(\int_{\mathbb{R}^n} dv_* \int_{\sph} d \sigma  ~B_k(v-v_*, \sigma) ~   \left| g_*' h' \right| \sqrt{ M_* }
\right)
   w^{2\ell}(v) ~  \left| f \right|
     \\
\lesssim
     \int dv 
\left(\int dv_* d \sigma  ~B_k^2 \sqrt{ M_* } \right)^{1/2}  
\left(\int dv_*  d \sigma  ~  \left| g_*' h' \right|^2 \sqrt{ M_* } \right)^{1/2}  
  w^{2\ell}(v)  ~ \left| f \right|
        \\
\lesssim
    2^{2sk} \int dv 
    \ang{v}^{\gamma+2s}
\left(\int dv_*  d \sigma  ~  \left| g_*' h' \right|^2 \sqrt{ M_* } \right)^{1/2}  
   w^{2\ell}(v)  ~ \left| f \right|
           \\
\lesssim
    2^{2sk} |w^{\ell}f|_{L^2_{\gamma+2s}} 
    \left(
    \int ~ dv dv_*  d \sigma ~ 
    \ang{v}^{\gamma+2s}
    w^{2(\ell^+-\ell')} (v'_*)
  w^{2(\ell + \ell' )} (v')
    \left| g_*' h' \right|^2 
    \right)^{1/2}. 
\end{multline*}
We have also used
$
w^{2\ell}(v) \sqrt{ M_* }
\lesssim
    w^{2(\ell^+-\ell')} (v'_*)
  w^{2(\ell + \ell' )} (v'),
$
which follows from the  splitting as in the previous case.
Since 
$
\ang{v}^{\gamma+2s}
\lesssim
 \ang{v'}^{\gamma+2s}
  +
  \ang{v'_*}^{\gamma+2s},
$
we have
established the estimates \eqref{tplussmall} and \eqref{tplussmall2} for $\gamma + 2s \ge 0$ (the additional terms on the right-hand sides of \eqref{tplussmall} and \eqref{tplussmall2} arising from the inequality just observed).
\end{proof}

\begin{proposition}
\label{referLATERpropK}
Suppose that $k\ge 0$.  Also consider
  $\ell^+, \ell^- \ge 0$,
with $\ell = \ell^+ - \ell^-$ and $0\le \ell' \le \ell^-$
so that $\ell + \ell' = \ell^+ - (\ell^- - \ell')$.
We   have the uniform estimates 
\begin{equation} 
\begin{split}
\left|  T^{k,\ell}_{+}(g,h,f)  \right| \lesssim   & ~
2^{2sk} 
 \nsm  g\nsm_{H^{\ksob}_{{\ell^+} - \ell'}} 
\nsm w^{\ell + \ell'} h\nsm_{L^2_{\gamma + 2s}} 
\nsm w^{\ell} f\nsm_{L^2_{\gamma + 2s}}.
\end{split}
 \label{tplussmallK}
\end{equation}
We also have
\begin{equation} 
\begin{split}
\left|  T^{k,\ell}_{+} (g,h,f) \right| \lesssim & ~
2^{2sk} 
\nsm w^{\ell^+ - \ell'} g\nsm_{L^2} 
 \nsm  h\nsm_{H^{\ksob}_{\ell + \ell', \gamma + 2s} }
\nsm w^{\ell} f\nsm_{L^2_{\gamma + 2s}}.
\end{split}
 \label{tplussmall2K}
\end{equation}
\end{proposition}

\begin{proof}
Notice that both inequalities are proved exactly in the proof of Proposition \ref{referLATERprop} in the case when $\gamma + 2s \le 0$.  For $\gamma + 2s \ge 0$, the only estimate which requires modification in the 
proof of Proposition \ref{referLATERprop} is the one in the splitting surrounding the estimate \eqref{useLaterKest}.  This estimate can be modified as follows.
When  $k\ge 0$, we have $|v_* - v_*'| \le 1$ which implies that for some $\delta >0$ there is exponential decay as 
$M_* \lesssim (M_* M_*')^{\delta}$.  This decay further implies 
$$
\ang{v - v_*}^{\gamma + 2s} (M_* M_*')^{\delta/2}
\lesssim
\ang{v}^{\gamma + 2s} (M_* M_*')^{\delta/4}
\lesssim
\ang{v'}^{\gamma + 2s}. 
$$
Now in the estimate for \eqref{useLaterKest}, when $\gamma + 2s \ge 0$, this polynomial growth 
 can be put entirely on $h'$, which finishes the argument.  
\end{proof}

This concludes our  size and support estimates.  In the next sub-section we will prove estimates which incorporate the essential cancellation properties of the Boltzmann collision operator in the appropriate geometric framework.

\subsection{Cancellations}
\label{sec:SSC}
We now establish estimates for the differences $T^{k,\ell}_{+}  - T^{k,\ell}_{-}$ and 
$T^{k,\ell}_{+} - T^{k,\ell}_{*}$ from \eqref{defTKL} and \eqref{defTKLcarl}. We wish the estimates to have good dependence on $k$ (the estimate should contain a negative power of $2^k$), but this improved norm will be paid for by assuming differentiability of one of the functions $h$ or $f$.  The key obstacle to overcome in making these estimates is that the magnitude of the gradients of $h$ and $f$ must be measured in some non-isotropic way; this is a fundamental point, as the scaling is imposed upon us by the structure of the ``norm piece'' $\ang{N f,f}$.

The scaling dictated by the problem is that of the paraboloid: namely, that the function $f(v)$ should be thought of as the restriction of some ``lifted'' function $F$ of $n+1$ variables to the paraboloid $(v, \frac{1}{2} |v|^2)$.  Consequently, the correct metric to use in measuring the length of vectors in $\mathbb{R}^n$ will be the metric on the paraboloid in $\mathbb{R}^{n+1}$ induced by the $(n+1)$-dimensional Euclidean metric.  To simplify calculations, we  work directly with the function $F$ rather than $f$ and take its $(n+1)$-dimensional derivatives in the usual Euclidean metric.  This will be sufficient since our Littlewood-Paley-type decomposition in Section \ref{sec:aniLP} will grant a natural way to extend the projections $Q_j f$ into $n+1$ dimensions while preserving the relevant differentiability properties of the $n$-dimensional restriction to the paraboloid.

To begin, it is necessary to find a suitable formula relating differences of $F$ at nearby points on the paraboloid to the various derivatives of $F$ as a function of $n+1$ variables.  To this end, fix any two $v,v' \in \mathbb{R}^n$, and consider $\gamma : [0,1] \rightarrow \mathbb{R}^n$ and $\ext{\gamma} : [0,1] \rightarrow \mathbb{R}^{n+1}$ given by
\[ 
\gamma(\theta) \eqdef \theta v' + (1-\theta) v,
\quad 
\mbox{ and } 
\quad 
\ext{\gamma}(\theta) \eqdef  \left(\theta v' + (1-\theta)v, \frac{1}{2} \left| \theta v' + (1-\theta) v \right|^2 \right). \]
Note that $\ext{\gamma}$ lies in the paraboloid $\set{(v_1,\ldots,v_{n+1}) \in \mathbb{R}^{n+1}}{ v_{n+1} = \frac{1}{2} (v_1^2 + \cdots v_n^2)}$, and that $\gamma(0) = v$ and $\gamma(1) = v'$.  
Also consider the starred analogs defined by
\[ 
\gamma_*(\theta) \eqdef \theta v'_* + (1-\theta) v_*,
\quad 
\mbox{ and } 
\quad 
\ext{\gamma_*}(\theta) \eqdef  \left(\theta v'_* + (1-\theta)v_*, \frac{1}{2} \left| \theta v'_* + (1-\theta) v_* \right|^2 \right). 
\]
Elementary calculations show that $\gamma(\theta) + \gamma_*(\theta) = v + v_*$ and 
\begin{align*}
\frac{d \ext{\gamma}}{d \theta} = \left( v' - v, \ang{\gamma(\theta), v' - v} \right) 
\quad 
& \mbox{ and }
\quad 
 \frac{d^2 \ext{\gamma}}{d \theta^2} = (0, |v'-v|^2), \\
\frac{d \ext{\gamma_*}}{d \theta} = - \left( v' - v, \ang{\gamma_*(\theta), v' - v} \right)
\quad 
& \mbox{ and }
\quad 
 \frac{d^2 \ext{\gamma_*}}{d \theta^2} = (0, |v'-v|^2).
 \end{align*}
Now we use the standard trick of writing the difference of $F$ at two different points in terms of an integral of a derivative (in this case the integral is along the path $\gamma$):
\begin{align} 
F\left(v',\frac{|v'|^2}{2} \right) - F\left(v, \frac{|v|^2}{2} \right) & = \int_0^1 d \theta ~ \frac{d}{d \theta} F(\ext{\gamma}(\theta)) \nonumber \\
& = \int_0^1 d \theta  \left( \frac{d \ext{\gamma}}{d \theta} \cdot  (\tilde{\nabla} F) (\ext{\gamma}(\theta)) \right), 
\label{paraboladiff}
\end{align}
where the dot product on the right-hand side is the usual Euclidean inner-product on $\mathbb{R}^{n+1}$ and $\tilde{\nabla}$ is the $(n+1)$-dimensional gradient of $F$. 
For convenience we define
\[ 
|\tilde{\nabla}|^i F(v_1,\ldots,v_{n+1}) \eqdef 
\max_{0\le j \leq i}\sup_{|\xi| \leq 1} \left| \left(\xi \cdot \tilde{\nabla} \right)^j F(v_1,\ldots,v_{n+1}) \right|, 
\quad
  i=0,1,2,
\]
where $\xi \in \mathbb{R}^{n+1}$ and $|\xi|$ is the usual Euclidean length. In particular, note that we have defined $|\tilde{\nabla}|^0 F = |F|$.

If $v$ and $v'$ are related by the collision geometry \eqref{sigma}, then $\ang{v-v',v'-v_*} = 0$, which yields that 
\begin{align*}
\ang{\gamma(\theta),v'-v} & = \ang{v_*,v'-v} - (1-\theta) |v-v'|^2, \\
\ang{\gamma_*(\theta),v'-v} & = \ang{v_*,v'-v} + \theta |v-v'|^2.
\end{align*}
Thus, whenever $|v-v'| \leq 1$, which holds near the singularity, we have
$$
\left| \frac{d \ext{\gamma_*}}{d \theta} \right|
+
\left| \frac{d \ext{\gamma}}{d \theta} \right| \lesssim |v-v'| \ang{v_*}. 
$$
 Throughout this section we suppose that $|v-v'| \leq 1$ since this is the situation where our cancellation inequalities will be used.
 In particular, we have the following inequality for differences related by the collisional geometry:
\begin{align}
 \left| F\left(v',\frac{|v'|^2}{2}\right) - F\left(v, \frac{|v|^2}{2}\right)\right| 
 &
  \lesssim 
  \ang{v_*} |v-v'|  \int_0^1  d \theta ~ |\tilde{\nabla}| F (\ext{\gamma}(\theta)). 
 \label{paradiff1} 
\end{align}
Furthermore,
by subtracting the linear term from both sides of \eqref{paraboladiff} and using the integration trick iteratively on the integrand of the integral on the right-hand side of \eqref{paraboladiff}, we obtain
\begin{multline}
\left| F\left(v',\frac{|v'|^2}{2} \right) -  F\left(v, \frac{|v|^2}{2}\right) -  \frac{d \ext{\gamma}}{d \theta}(0) \cdot \tilde{\nabla} F(v) \right| 
\\
 \lesssim
\ang{v_*}^2 |v-v'|^2 \int_0^1 d \theta  ~ | \tilde{\nabla}|^2 F (\ext{\gamma}(\theta)). 
\label{paradiff2}
\end{multline}
We note that, by symmetry, the same result holds when the roles of $v$ and $v'$ are reversed (which only changes the curve $\ext{\gamma}$ by reversing the parametrization: $\ext{\gamma}(\theta)$ becomes $\ext{\gamma}(1-\theta)$).  
It is also trivially true that the corresponding starred version of \eqref{paradiff2} holds as well.
We will use these two basic cancellation inequalities to prove the  cancellation estimates for the trilinear form in the following propositions.

\begin{proposition}
\label{cancelFprop}
Suppose $f$ is a Schwartz function on $\mathbb{R}^n$ given by the restriction of some Schwartz function $F$ 
on $\mathbb{R}^{n+1}$ to the paraboloid $(v, \frac{1}{2} |v|^2)$.  Let $|\tilde{\nabla}|^2 f$ be the restriction of $|\tilde{\nabla}|^2 F$ to the same paraboloid.  Then, for any $k \geq 0$,
\begin{gather}
\label{cancelf21}
 \left| (T^{k,\ell}_{+}- T^{k,\ell}_{-})(g,h,f) \right|   
 \lesssim 2^{(2s-2)k} 
\nsm g\nsm_{L^2_{-m}}
\nsm  h\nsm_{H^{\ksob}_{\ell,\gamma+2s}}
\nl w^\ell |\tilde{\nabla}|^2 f \nr_{L^2_{\gamma+2s}},
\\
 \left| (T^{k,\ell}_{+}- T^{k,\ell}_{-})(g,h,f) \right|   
 \lesssim 2^{(2s-2)k} 
\nsm  g\nsm_{H^{\ksob}_{-m}}
  \nsm w^\ell h\nsm_{L^2_{\gamma+2s}}
 \nl w^\ell |\tilde{\nabla}|^2 f \nr_{L^2_{\gamma+2s}}. 
\label{cancelf2}
\end{gather}
Each of these inequalities hold for any $m \ge 0$ and any $\ell \in \mathbb{R}$.
\end{proposition}

\begin{proof}
We write out the relevant difference into several terms
\begin{multline}
M_\beta(v_*') w^{2\ell}(v')  f' -  M_\beta(v_*) w^{2\ell}(v) f 
=
w^{2\ell}(v) ~ f
~ \left( \frac{d \ext{\gamma_*}}{d \theta}(0) \cdot \tilde{\nabla} M_\beta(v_*)  \right)
\\
+
\left( w^{2\ell}(v') ~ f' - w^{2\ell}(v) ~  f   \right)
~ \left( \frac{d \ext{\gamma_*}}{d \theta}(0) \cdot \tilde{\nabla} M_\beta(v_*)  \right)
\\
+
w^{2\ell}(v') ~ f' ~ 
 \left( M_\beta(v_*')  - M_\beta(v_*) -  \frac{d \ext{\gamma_*}}{d \theta}(0) \cdot \tilde{\nabla} M_\beta(v_*)    \right)
 \\
+
M_\beta(v_*) 
\left( w^{2\ell}(v') ~ f' - w^{2\ell}(v) ~  f  -  \frac{d \ext{\gamma}}{d \theta}(0) \cdot \tilde{\nabla} (w^{2\ell}  f )(v) \right)
 \\
 +
M_\beta(v_*) ~ \left(  \frac{d \ext{\gamma}}{d \theta}(0) \cdot \tilde{\nabla} (w^{2\ell}  f )(v)   \right)
 = \mathbb{I} + \mathbb{II}
 +
  \mathbb{III} + \mathbb{IV}+ \mathbb{V}.
\notag
\end{multline}
We further split 
$
(T^{k,\ell}_{+}- T^{k,\ell}_{-})(g,h,f)
=
T^{\mathbb{I}}+ T^{\mathbb{II}}
+
T^{\mathbb{III}}+ T^{\mathbb{IV}}+ T^{\mathbb{V}},
$
where $T^{\mathbb{I}}$ corresponds to the first term in the splitting above, etc.

We begin by considering the first and last terms, that is
$$
T^{\mathbb{I}}
=
\int_{\mathbb{R}^n} dv \int_{\mathbb{R}^n} dv_* \int_{\sph} d \sigma ~B_k(v-v_*, \sigma) ~ g_* h 
~ \left( \frac{d \ext{\gamma_*}}{d \theta}(0) \cdot \tilde{\nabla} M_\beta(v_*)  \right)
~ w^{2\ell}(v) ~ f,
$$
and
$$
T^{\mathbb{V}}
=
\int_{\mathbb{R}^n} dv \int_{\mathbb{R}^n} dv_* \int_{\sph} d \sigma ~B_k(v-v_*, \sigma) ~ g_* h 
~ M_\beta(v_*) ~ \left(  \frac{d \ext{\gamma}}{d \theta}(0) \cdot \tilde{\nabla} (w^{2\ell}  f )(v)   \right).
$$
We may estimate both of these terms in exactly the same way.  

First we define the extension.  If $f$ extends to $\mathbb{R}^{n+1}$, then 
$$
(w^{2\ell} f) (v_1,\ldots,v_{n+1}) = w^{2\ell}(v) f(v_1,\ldots,v_{n+1}).
$$  
Here we think of $w^{2\ell}(v)$ as being constant in the $(n+1)$-st coordinate
With this extension it follows that for $i=0,1,2$ we have
\begin{equation}
\label{nablaD}
|\tilde{\nabla}|^i (w^{2\ell}  f) 
\lesssim 
 w^{2\ell} |\tilde{\nabla}|^i f.
\end{equation}
We will use this basic estimate several times below.

For $T^{\mathbb{V}}$, notice that $\frac{d \ext{\gamma}}{d \theta}(0)$
is linear in $v'-v$ and has no other dependence on $v'$.  Thus the symmetry of $B_k$ with respect to $\sigma$ around the direction $\frac{v-v_*}{|v-v_*|}$ forces all components of $v'-v$ to vanish except the component in the symmetry direction.
Thus, one may replace $v-v'$ with $\frac{v-v_*}{|v-v_*|} \ang{v-v', \frac{v-v_*}{|v-v_*|}}$ in the expression for 
$\frac{d \ext{\gamma}}{d \theta}(0)$.  Since $\ang{v-v',v'-v_*} = 0$, the vector further reduces to $\frac{v-v_*}{|v-v_*|}\frac{|v-v'|^2}{|v-v_*|}$.  Hence
\[ 
\left| \frac{v-v_*}{|v-v_*|}\frac{|v-v'|^2}{|v-v_*|} \right| \leq 2^{-2k} |v-v_*|^{-1}.
\]
The last coordinate direction of $\frac{d \ext{\gamma}}{d \theta}(0)$ is given by $\ang{v,v'-v}$ which
 reduces to 
\[ 
\left| \ang{v, \frac{v-v_*}{|v-v_*|}\frac{|v-v'|^2}{|v-v_*|}} \right| \lesssim 
\left( 1 + |v-v_*|^{-1} \right)
|v-v'|^2  \ang{v_*}. 
\]
These bounds allow us to employ the same methods as used previously.  
To be precise, we must control the following integral
\begin{equation} 
2^{-2k} \int_{\mathbb{R}^n} dv \int_{\mathbb{R}^n} dv_* \int_{S^2} d \sigma ~  B_k ~ M_*^{1-\epsilon} ~ w^{2\ell}(v)~ |g_*| |h| ( |\tilde{\nabla}| f)
\left( 1 + |v-v_*|^{-1} \right).
\notag
\end{equation}
Here we absorb any powers of $\ang{v_*}$  by $M_*^{-\epsilon}$ for any small $\epsilon >0$.
The estimates 
\eqref{cancelf21}
and
\eqref{cancelf2} for $T^{\mathbb{V}}$ then follow as in the proofs of
\eqref{tminusg} 
and
 \eqref{tminush},
the only difference being the extra factor $|v-v_*|^{-1}$; however, since $\gamma+2s-1 \ge -(n-1)$ one may replace $\gamma$ with $\gamma-1$ in the earlier proofs without fear of destroying the local integrability of any of the singularities arising therein.

The estimation of $T^{\mathbb{I}}$ can be handled in exactly the same way because $\frac{d \ext{\gamma_*}}{d \theta}(0)$ also depends linearly on $v-v'$ and has no other $v'$ dependence.  
Here we have exactly the same estimates for $\frac{d \ext{\gamma_*}}{d \theta}(0)$ as just previously obtained for $\frac{d \ext{\gamma}}{d \theta}(0)$ by symmetry.  Then 
\eqref{cancelf21}
and
\eqref{cancelf2} for $T^{\mathbb{I}}$ follow again as in the proof of
\eqref{tminusg} 
and
 \eqref{tminush}.  As in \eqref{nablaD}, we use
 \begin{equation}
\notag
|\tilde{\nabla}|^i M_\beta
\lesssim 
\sqrt{M},
\quad
i = 0,1,2,
\end{equation}
where the extension is defined as
$
M_\beta(v) = (2 \pi)^{-n/2}~p_\beta (v_1, \ldots, v_n) ~ e^{-v_{n+1}/2}.
$

We now turn to the estimation of the term $T^{\mathbb{III}}$, which can be written as
\begin{multline}
T^{\mathbb{III}}
=
\int_{\mathbb{R}^n} dv \int_{\mathbb{R}^n} dv_* \int_{\sph} d \sigma ~B_k(v-v_*, \sigma) ~ g_* h
~ w^{2\ell}(v') ~ f' 
\\
\times
 \left( M_\beta(v_*')  - M_\beta(v_*) -  \frac{d \ext{\gamma_*}}{d \theta}(0) \cdot \tilde{\nabla} M_\beta(v_*)    \right). 
 \notag
\end{multline}
We apply the starred analog of \eqref{paradiff2} to obtain
\begin{equation}
\left| T^{\mathbb{III}}  \right|
\lesssim
2^{-2k}
\int_{\mathbb{R}^n} dv \int_{\mathbb{R}^n} dv_* \int_{\sph} d \sigma ~ B_k  \left| g_* h \right|
 w^{2\ell}(v')  \left| f'  \right|
(M_* M'_*)^{\epsilon},
 \notag
\end{equation}
where we have just used the estimate
$
\sqrt{M(\ext{\gamma_*}(\theta))} \le (M_* M_*')^{\epsilon},
$
(valid for all sufficiently small $\epsilon > 0$) which follows directly from $\ang{v_*} \lesssim \ang{\gamma_*(\theta)}\lesssim \ang{v_*}$
since
$
v_* = \gamma_*(\theta) + \theta(v_*-v_*')
$
and $|v_* - v_*'| \le 1$.
At this point, after applying  Cauchy-Schwartz, it suffices to prove two estimates. The first is 
that for any $m \ge 0$ 
\begin{multline}
 \label{easyterm}
 \left( 
 \int_{\mathbb{R}^n} dv \int_{\mathbb{R}^n} dv_* \int_{\sph} d \sigma B_k M_*^{\epsilon} 
  ~
w^{2\ell}(v)
 ~
 |g_*|^2 |h|^2 \right)^{\frac{1}{2}} 
 \\
 \lesssim 
 2^{sk} \left( \min\left\{ \nsm g\nsm_{L^2_{-m}} \nsm h\nsm_{H^{\ksob}_{-m}}, \nsm g\nsm_{H^{\ksob}_{-m}} \nsm  h\nsm_{L^2_{-m}} \right\}
 +
 \nsm g\nsm_{L^2_{-m}} \nsm w^\ell h\nsm_{L^2_{\gamma + 2s}}
 \right), 
\end{multline}
 where $w^{2\ell}(v)$ replaces $w^{2\ell}(v')$ because $\ang{v} \approx \ang{v'}$ when $|v-v'| \leq 1$. The second is
\begin{equation}
\notag
 \left( \int_{\mathbb{R}^n} dv \int_{\mathbb{R}^n} dv_* \int_{\sph} d \sigma  B_k (M_* M'_*)^{\epsilon} 
w^{2\ell}(v')
 | f'|^2 \right)^\frac{1}{2}  
 \lesssim 
 2^{sk} \nsm  w^{\ell} f\nsm_{L^2_{\gamma+2s}}. 
\end{equation}
Both of these inequalities follow exactly as in the proof of \eqref{tminusg} and \eqref{tminush}.

It remains to only prove \eqref{cancelf21}
and
\eqref{cancelf2} for $T^{\mathbb{II}}$ and $T^{\mathbb{IV}}$, both of which involve differences in $w^{2\ell} f$.  We use the difference estimates
\eqref{paradiff1} for $T^{\mathbb{II}}$
and
\eqref{paradiff2} for $T^{\mathbb{IV}}$ to obtain that
 for any fixed $\epsilon > 0$ both  $|T^{\mathbb{II}}|$ and $|T^{\mathbb{IV}}|$
 are controlled by
$$
  \lesssim 2^{-2k} \int_0^1 d \theta \int_{\mathbb{R}^n} dv \int_{\mathbb{R}^n} dv_* \int_{\sph} d \sigma ~ B_k ~ M_*^{1-\epsilon} ~ 
  w^{2\ell}(v)~ |g_* h| 
  ~ |\tilde{\nabla}|^2 f (\gamma(\theta) ).
$$
 The loss of $\epsilon$ comes from the factor $\ang{v_*}$ in \eqref{paradiff1} and  \eqref{paradiff2}
 in addition to the bound  $| \tilde{\nabla} | M_\beta(v_*)  \lesssim M_*^{1-\epsilon/2}$ and the estimate for 
 $
 \left| \frac{d \ext{\gamma_*}}{d \theta} \right|.
 $
These also account for the $2^{-2k}$. Note, though, that the factor $2^{-2k}$ comes directly from \eqref{paradiff2}, but \eqref{paradiff1} only furnishes a factor of $2^{-k}$. In this case, there is an additional factor of $2^{-k}$ available in the estimate for $T^{\mathbb{II}}$ arising exactly from the derivative of $\ext{\gamma_*}$.  
Also,
$\ang{v} \approx \ang{\gamma(\theta)}$ (which accounts for the replacement of $w^{2 \ell}(\gamma(\theta))$ by $w^{2 \ell}(v)$).

With that last estimate and Cauchy-Schwartz, it suffices to use \eqref{easyterm}
and show
\begin{multline}
 \left( \int_0^1 d \theta \int_{\mathbb{R}^n} dv \int_{\mathbb{R}^n} dv_* \int_{\sph} d \sigma  B_k M_*^{1-\epsilon} 
 w^{2\ell}(\gamma(\theta))
 \left|  |\tilde{\nabla}|^2 f(\gamma(\theta))\right|^2 \right)^\frac{1}{2}  
 \\
 \lesssim 
 2^{sk} \nsm w^\ell  |\tilde{\nabla}|^2 f \nsm_{L^2_{\gamma+2s}}. 
 \label{covterm}
\end{multline}
This uniform bound
 follows from the  change of variables $u = \theta v' + (1-\theta)v$, which is a transformation from  $v$ to $u$.  In view of the collisional variables \eqref{sigma}, we see  (with $\delta_{ij}$  the usual Kronecker delta) that 
$$
\frac{d u_i}{ dv_j} = (1-\theta) \delta_{ij} + \theta \frac{d v'_i}{ dv_j}
 = \left(1-\frac{\theta}{2}\right)  \delta_{ij} + \frac{\theta}{2} k_{j} \sigma_{i},
$$
with the unit vector $k = (v-v_*)/|v-v_*|$.  Thus the Jacobian is 
$$
\left| \frac{d u_i}{ dv_j}  \right| 
= 
\left(1-\frac{\theta}{2}\right)^2\left\{
\left(1-\frac{\theta}{2}\right) + 
\frac{\theta}{2} \ang{k, \sigma}
\right\}.
$$
Since  $b(\ang{k, \sigma}) = 0$ when $\ang{k ,\sigma} \leq 0$ from \eqref{kernelQ}, and  $\theta \in [0,1]$, 
it follows that 
the Jacobian  is bounded from below on the support of the integral \eqref{covterm}. 
But after this change of variable the old pole $k=(v-v_*)/|v-v_*|$ moves with the angle $\sigma$.  However  when one takes $\tilde k = (u-v_*)/|u-v_*|$, 
then
$
1- \ang{ k,\sigma } \approx  1 - \left< \right. \! \tilde{k}, \sigma \! \left. \right>, 
$
meaning that the angle to the pole is comparable to the angle to $\tilde{k}$ (which does not vary with $\sigma$).
Thus the estimate analogous to \eqref{bjEST} will continue to hold after the change of variables, giving precisely the estimate in \eqref{covterm}.
\end{proof}

This completes our proof of the cancellations for the $\sigma$ representation as in \eqref{defTKL}.  In the following, we estimate the cancellations on $h$ instead of putting them on $f$, and for this we use the Carleman representation as in \eqref{defTKLcarl}.

\begin{proposition}
\label{cancelHprop}
As in the prior proposition, suppose $h$ is a Schwartz function on $\mathbb{R}^n$ which is given by the restriction of some Schwartz function in $\mathbb{R}^{n+1}$ to the paraboloid $(v, \frac{1}{2} |v|^2)$ and define $|\tilde{\nabla}|^i h$ analogously for $i=0,1,2$. 
We have 
\begin{gather}
\label{cancelh2g1}
 \left| (T^{k,\ell}_{+}- T^{k,\ell}_{*})(g,h,f) \right| 
   \lesssim 2^{(2s-2)k} 
    \nsm g\nsm_{L^2_{-m}} 
\nl  |\tilde{\nabla}|^2 h \nr_{H^{\ksob}_{\ell,\gamma+2s}} 
  \nsm w^\ell f\nsm_{L^2_{\gamma+2s}},
\\
 \left| (T^{k,\ell}_{+}- T^{k,\ell}_{*})(g,h,f) \right| 
   \lesssim 2^{(2s-2)k} 
\nsm  g\nsm_{H^{\ksob}_{-m}}
  \nl w^\ell |\tilde{\nabla}|^2 h \nr_{L^2_{\gamma+2s}}
  \nsm w^\ell f\nsm_{L^2_{\gamma+2s}}.
  \label{cancelh2g2}
\end{gather}
The above inequalities hold uniformly in $k \geq 0$,  $m\ge 0$, and $\ell \in \mathbb{R}$.
\end{proposition}

\begin{proof}
This proof follows the pattern that is now established.  The new feature in \eqref{cancelh2g1}  and \eqref{cancelh2g2} is that, from \eqref{defTKLcarl}, the pointwise differences to examine are
\begin{multline}
M_\beta(v_*') h \Phi(v-v_*) |v-v_*|^n -  M_\beta(v_*) h' \Phi(v'-v_*) |v' - v_*|^n  
\\
=
 \left( \frac{d \ext{\gamma_*}}{d \theta}(0) \cdot \tilde{\nabla} M_\beta(v_*)    \right)
 h' \Phi(v'-v_*) |v'-v_*|^n
\\
+
 \left( \frac{d \ext{\gamma_*}}{d \theta}(0) \cdot \tilde{\nabla} M_\beta(v_*)    \right)
 \left( h \Phi(v-v_*) |v-v_*|^n -   h' \Phi(v'-v_*) |v' - v_*|^n  
 \right)
\\
+
 \left( M_\beta(v_*')  - M_\beta(v_*) -  \frac{d \ext{\gamma_*}}{d \theta}(0) \cdot \tilde{\nabla} M_\beta(v_*)    \right)
 h \Phi(v-v_*) |v-v_*|^n
 \\
 +
 M_\beta(v_*) 
 \left( h \Phi(v-v_*) |v-v_*|^n -   h' \Phi(v'-v_*) |v' - v_*|^n  
 -  \frac{d \ext{\gamma}}{d \theta}(1) \cdot \tilde{\nabla} F(\ext{\gamma}(1))
 \right)
 \\
  +
 M_\beta(v_*) 
 \left(  \frac{d \ext{\gamma}}{d \theta}(1) \cdot \tilde{\nabla} F(\ext{\gamma}(1))
 \right)
 = \mathbb{I} + \mathbb{II}
 +
  \mathbb{III} + \mathbb{IV}+ \mathbb{V}.
\notag
\end{multline}
Here
since
$
\frac{d}{d \theta} |\gamma(\theta) - v_*|^2 = 2 \ang{v' - v, \gamma(\theta) - v_*} = 0, 
$
the cancellation 
term is
\[  
\frac{d \ext{\gamma}}{d \theta}(1) \cdot \tilde{\nabla} F(\ext{\gamma}(1)) = \Phi(v' - v_*) |v'-v_*|^n \frac{d \ext{\gamma}}{d \theta}(1) \cdot \tilde{\nabla} h'.
\]
We again 
split 
$
(T^{k,\ell}_{+}- T^{k,\ell}_{*})(g,h,f)
=
T^{\mathbb{I}}_*+ T^{\mathbb{II}}_*
+
T^{\mathbb{III}}_*+ T^{\mathbb{IV}}_* + T^{\mathbb{V}}_*,
$
where $T^{\mathbb{I}}_*$ corresponds to the first term in the splitting above, etc.

For the last term $T^{\mathbb{V}}_*$, we have
\begin{gather*}
T^{\mathbb{V}}_*
\eqdef
 \int_{\mathbb{R}^n} dv' ~ w^{2\ell}(v') ~
 \int_{\mathbb{R}^n} dv_*   \int_{E_{v_*}^{v'}} d \pi_v    
 \frac{B_k }{2^{1-n}}  
 \frac{ M_\beta(v_*)  g_* f' }{|v'-v_*| | v-v_*|^{n-2}}
  \frac{\Phi(v'-v_*)}{\Phi(v-v_*) } 
  \\
   \times  \frac{|v'-v_*|^n}{ |v - v_*|^n} 
  \frac{d \ext{\gamma}}{d \theta}(1) \cdot \tilde{\nabla} h' = 0.
\end{gather*}
In this integral as $v$ varies on circles of constant distance to $v'$, the entire integrand is constant except for $\frac{d \ext{\gamma}}{d \theta}(1)$.  If we write $\frac{d \ext{\gamma}}{d \theta}(1)$ as a sum of two vectors, one lying in the span of the first $n$ directions and the second pointing in the last direction, it follows that we may replace the former vector by its projection onto the direction determined by $v' - v_*$.  But since the original vector points in the direction $v-v'$, the projection vanishes.  Since the last direction of $\frac{d \ext{\gamma}}{d \theta}(1)$  is exactly $\ang{v',v'-v}$, the corresponding integral of this over $v$ also vanishes by symmetry.  

Similarly, we consider
the first term $T^{\mathbb{I}}_*$, we have
\begin{gather*}
T^{\mathbb{I}}_*
\eqdef
 \int_{\mathbb{R}^n} dv' ~ w^{2\ell}(v') ~ \int_{\mathbb{R}^n} dv_*   \int_{E_{v_*}^{v'}} d \pi_v    
 \frac{B_k }{2^{1-n}}  
 \frac{ M_\beta(v_*)  g_* f' }{|v'-v_*| | v-v_*|^{n-2}}
  \frac{\Phi(v'-v_*)}{\Phi(v-v_*) } 
  \\
   \times  \frac{|v'-v_*|^n}{ |v - v_*|^n}  h' 
    \left( \frac{d \ext{\gamma_*}}{d \theta}(0) \cdot \tilde{\nabla} M_\beta(v_*)    \right)
   = 0.
\end{gather*}
Here the explanation is the same as in the previous case.

The remaining terms incorporate cancellations.  
To estimate $T^{\mathbb{II}}_*$ and $T^{\mathbb{IV}}_*$ we study the function of $\theta$ which measures differences given by
\[ 
 h(\ext{\gamma}(\theta)) \Phi(\gamma(\theta) - v_*) |\gamma(\theta)-v_*|^n.
\]
Notice that both terms $\mathbb{II}$ and $\mathbb{IV}$ include the difference in the values of this function at $\theta = 0$ and $\theta = 1$ (we also extend $h$ to $\mathbb{R}^{n+1}$ and make the de facto extension of  $\Phi(v-v_*) |v-v_*|^n$ to $\mathbb{R}^{n+1}$ by taking the new function to be constant in the last variable).
In terms of $\mathbb{II}$ and $\mathbb{IV}$, our operator  from \eqref{defTKLcarl} takes the form
\[ 
T^{\mathbb{II}}_*
+
T^{\mathbb{IV}}_* = \int_{\mathbb{R}^n} dv' ~ w^{2\ell}(v') ~\int_{\mathbb{R}^n} dv_* \int_{E_{v_*}^{v'}} d \pi_v  
\frac{B_k ~ g_*  f'   ~\left( \mathbb{II}~ + \mathbb{IV} \right)}{|v'-v_*| |v-v_*|^{2n-2} ~ \Phi(v-v_*)} . 
\]
We use \eqref{paradiff1} and \eqref{paradiff2} to estimate $\mathbb{II}$ and $\mathbb{IV}$.  The additional estimate needed is that
\[ 
\left|\tilde{\nabla}^i (\Phi(v-v_*) |v-v_*|^n) \right| \lesssim (1+|v-v_*|^{-i}) (\Phi(v-v_*) |v-v_*|^n), 
\]
which simply comes from differentiating with respect to $v$ (since the extension is taken to be constant in the last direction, the gradient $\tilde{\nabla}$ reduces to the usual $n$-dimensional gradient).
Applying \eqref{paradiff1} (with the estimates for $\left| \frac{d \ext{\gamma_*}}{d \theta} \right|$) and \eqref{paradiff2} gives the estimate
\begin{gather*}
 \frac{| \mathbb{II} | + | \mathbb{IV} |}{\Phi(v-v_*)|v-v_*|^n}  \lesssim   \frac{|v-v'|^2}{\min\{|v-v_*|^2, 1\}}
 M_*^{1/2} \int_0^1 d \theta ~  |\tilde{\nabla}|^2 h(\ext{\gamma}(\theta)).  
\end{gather*}
Again,  $|v-v'| \lesssim 2^{-k}$ and also 
$
|v-v_*| \approx |\gamma(\theta)-v_*|.
$
With all of that we may  estimate the terms $\left| T^{\mathbb{II}}_* \right|$
and
$\left| T^{\mathbb{IV}}_* \right|$
 above by the following single term:
\begin{multline*}
2^{-2k}\int_0^1 d \theta \int_{\mathbb{R}^n} dv' \int_{\mathbb{R}^n} dv_* \int_{E_{v_*}^{v'}}  d \pi_{v} ~  \frac{B_k(v-v_*, 2v' - v- v_*) }{ |v'-v_*| |v-v_*|^{n-2} \min\{ |v-v_*|^2 , 1\}}
\\
  \times   w^{2\ell}(v') ~
  M_*^{1/2} |g_* f'| ~ |\tilde{\nabla}|^2 h(\gamma(\theta)).
\end{multline*}
 The estimates required now for the term above are completely analogous to a corresponding estimate from a previous proposition.
In particular, we split into the regions $|v-v_*| \le 1$ and $|v-v_*| \ge 1$.  When $|v-v_*| \ge 1$ we apply Cauchy-Schwartz to the quantities above carrying $|g_* f'| $ into one of the integrals and 
$ |\tilde{\nabla}|^2 h(\gamma(\theta))$ into the other integral.  The rest of the terms above are carried half into each integral.
The integrals involving $ |\tilde{\nabla}|^2 h(\gamma(\theta))$ are estimated as in \eqref{covterm} after changing back from the Carleman representation to the $\sigma$ representation; see for instance \cite[Proposition 10]{gsNonCut1}.  
The integrals involving $|g_* f'| $ are estimated  as in the proof of \eqref{tminush}, \eqref{bjEST}, and Proposition \ref{starPROP}.  

When instead $|v-v_*| \le 1$, since also $|v-v'| \le 1$, we have strong exponential decay in all variables.  After taking the $L^\infty$ norm of either $\ang{\gamma(\theta)}^{-m} |\tilde{\nabla}|^2 h(\gamma(\theta))$
or $\ang{v_*}^{-m} |g_*|$, and using \eqref{sobolevE}, we can use the Schur test for integral operators to finish off our estimates, or simply Cauchy-Schwartz as in Proposition \ref{prop11}
and Proposition \ref{starPROP}.
The additional singularity does not cause a problem because $\gamma + 2s -2 > -n$ which implies local integrability.

Regarding the estimate for $T^{\mathbb{III}}_* $, we notice that the estimate of \eqref{paradiff2}
holds if $v, v'$ are replaced by $v_*, v_*'$ and we replace $\ext{\gamma}$ with $\ext{\gamma_*}$.
Using \eqref{paradiff2} in this case, we have
\begin{multline*}
\left| T^{\mathbb{III}}_* \right|
\lesssim
2^{-2k}\int_0^1 d \theta \int_{\mathbb{R}^n} dv' \int_{\mathbb{R}^n} dv_* \int_{E_{v_*}^{v'}}  d \pi_{v} ~  \frac{B_k(v-v_*, 2v' - v- v_*) }{ |v'-v_*| |v-v_*|^{n-2} }
\\
  \times   w^{2\ell}(v') ~ M_*^{1/2} |g_* f'| ~ | h |.
\end{multline*}
Now this estimate can be handled as in the previous case for $T^{\mathbb{II}}_* $. This case is easier because there is no $\gamma(\theta)$ here, which required another change of variable.
\end{proof}

This concludes our cancellation estimates for the differences involving three arbitrary smooth functions.  We will also need cancellation estimates when have a more specific smooth function satisfying
the following estimate
\begin{equation}
\label{derivESTa}
|\tilde{\nabla}|^2 \phi \le C_\phi e^{-c|v|^2},
\quad
C_\phi \ge 0, \quad 
c> 0.
\end{equation}
Above $\phi$ is any  Schwartz function on $\mathbb{R}^n$ which is given by the restriction of some Schwartz function in $\mathbb{R}^{n+1}$ to the paraboloid $(v, \frac{1}{2} |v|^2)$ and $|\tilde{\nabla}|^2 \phi$ is defined analogously as usual. 
With this in mind, we have the next estimates:

\begin{proposition}  
\label{compactCANCELe}
As in the prior propositions, 
with \eqref{derivESTa},
for any $k \geq 0$ we have
\begin{gather}
\notag
 \left| (T^{k,\ell}_{+}- T^{k,\ell}_{-} )(g,h,\phi) \right| 
 +
 \left| (T^{k,\ell}_{+}- T^{k,\ell}_{*})(g,\phi,h) \right| 
   \lesssim 
   C_\phi ~ 2^{(2s-2)k} 
    \nsm g\nsm_{L^2_{-m}}     \nsm h\nsm_{L^2_{-m}} .
\end{gather}
The above inequalities hold uniformly for any $m\ge 0$ and $\ell \in \mathbb{R}$.
\end{proposition}

Notice that the proof of this Proposition \ref{compactCANCELe} follows exactly the proofs of the previous two Propositions \ref{cancelFprop} and \ref{cancelHprop}.  The main difference now is that when going through the last two proofs above, as a result of \eqref{derivESTa}, we will always in every estimate have strong exponential decay in both variables $v$ and $v_*$.  This strong decay allows us to easily obtain Proposition \ref{compactCANCELe}
using exactly the techniques developed in this section and the last one.  We omit repeating these details again.

\subsection{Compact estimates}  
\label{sec:SSCE}
In this sub-section we prove several useful estimates for the ``compact part'' of the linearized collision operator \eqref{compactpiece}.  Here we use the integer index $j$ instead of  $k$ to contrast with  the kernel $\kappa$ below.
Our first step is to notice that
the Carleman representation \eqref{defTKLcarl}
 of $T^{j,\ell}_{+}(g,M_{\beta_1},f)$ grants
$$
T^{j,\ell}_{+}(g,M_{\beta_1},f) = 
\int_{\mathbb{R}^n} dv' ~ w^{2\ell}(v')~ f(v') ~ 
\int_{\mathbb{R}^n} dv_*  ~ g(v_*) ~ \kappa_j^{\gamma+2s}(v',v_*),
$$
where
for any multi-indices $\beta$ and $\beta_1$ we have the kernel
$$
\kappa_j^{\gamma+2s}(v',v_*)
\eqdef
2^{n-1} 
\int_{E_{v_*}^{v'}} d \pi_{v}  ~
\frac{B_j(v-v_*, 2v' - v- v_*)}{|v'-v_*| ~ |v-v_*|^{n-2}} ~   M_{\beta}(v_*') M_{\beta_1}(v). 
$$
Recall the domain
  $
  E^{v'}_{v_*} = \left\{ v'\in \mathbb{R}^n : \ang{v_*-v', v - v'} =0 \right\},
  $
  and that $d\pi_{v} $ denotes the Lebesgue measure on this hyperplane.  More generally,  
  suppose that 
  $$
\kappa_{j}^{\psi, \phi}(v',v_*)
\eqdef
2^{n-1} 
\int_{E_{v_*}^{v'}} d \pi_{v}  ~
\frac{B_j(v-v_*, 2v' - v- v_*)}{|v'-v_*| ~ |v-v_*|^{n-2}} ~   \psi(v_*') \phi(v),
$$
where 
$\phi$ satisfies 
$\left| \phi(v) \right| \le C_\phi e^{-c |v|^2}$, and similarly
$\left| \psi(v) \right| \le C_\psi e^{-c |v|^2}$
 for any positive constants
$C_\phi$, $C_\psi$, and $c$ as in \eqref{rapidDECAYfcn}.
Then we have the main compact estimate:

 \begin{lemma}\label{prop:Grad}
For any $\ell \in \R$,  $\kappa_{j}^{\psi, \phi}$ satisfies the following uniform in $j \le 0$ estimate
 $$
  \forall \, v_* \in \R^n, \quad \int_{\R^n} dv'  ~ \left| \kappa_{j}^{\psi, \phi}(v',v_*) \right| \, 
  w^\ell(v')
   \lesssim
       C_\psi C_\phi ~  2^{2sj}\ang{v_*}^{\gamma + 2s-(n-1)}
         w^\ell(v_*).
 $$
Furthermore, 
the same estimate holds if the variables are reversed
 $$
  \forall \, v' \in \R^n, \quad \int_{\R^n} dv_*  ~ \left| \kappa_{j}^{\psi, \phi}(v',v_*) \right| \, w^\ell(v_*)
   \lesssim
       C_\psi C_\phi ~  2^{2sj}\ang{v'}^{\gamma + 2s-(n-1)}
       w^\ell(v').
 $$
 \end{lemma}

The crucial gain of the weight $-(n-1)$ has been known in the cut-off regime;
see Grad \cite{MR0156656}, and also \cite{MR2322149}.  The essential new feature here is to gain this weighted estimate without angular cut-off using our decomposition of the singularity and various estimates involving \eqref{goodKJestimate} below.  

\begin{proof}[Proof of Lemma~\ref{prop:Grad}]
The second upper bound follows easily from the first after noticing that the roles of $v'$ and $v_*$ can generally be reversed in the first upper bound for 
$ \left| \kappa_{j}^{\psi, \phi}(v',v_*) \right|$ given below in \eqref{symmetricUPPER}.

We focus here on the first bound from above. We perform the change of variables $z = v - v'$ in $v$, with as usual $v_*' = v + v_* - v'$,  then the kernel can be written
$$
\kappa_{j}^{\psi, \phi}(v',v_*)
=
2^{n-1} 
\int_{\omega^\bot} d\pi_{z} 
\frac{B_j(v'-v_* + z, v'-v_* - z)}{|v'-v_*| ~ |v'-v_* + z|^{n-2}} ~   \psi(z+v_*) \phi(z+v'). 
$$
Above $\omega = \frac{v'-v_*}{|v' - v_*|}$ and $\omega^\bot$ denotes integration over the $n-1$ dimensional hyperplane orthogonal to $\omega$ with Lebesgue measure $d\pi_{z}$.  We expand the exponents as
  \begin{multline*}
 |z+v_*|^2 + |z+v'|^2 =  |z+v_*|^2 + |z+v_*+(v'-v_*)|^2  = \frac12 |v'-v_* + 2(v_*+z)|^2 + \frac12 |v'-v_*|^2 \\
     = \frac12 |v'-v_* + 2 \ang{v_*, \omega} \omega|^2
    + 2 |z + \left(v_* - \ang{v_*, \omega} \omega \right)|^2 + \frac12 |v'-v_*|^2.
 \end{multline*}
Furthermore, the kernel can be expanded as
$$
B_j(v'-v_* + z, v'-v_* - z)
=
b \left(\ang{ \frac{v'-v_* + z}{|v'-v_* + z|}, \frac{v'-v_* - z}{|v'-v_* - z|}} \right)  \Phi( |v'-v_* + z|) \chi_j (|z|). 
$$
The orthogonality condition implies
$
\ang{\frac{v'-v_* + z}{|v'-v_* + z|}, \frac{v'-v_* - z}{|v'-v_* - z|}}
=
\frac{|v'-v_*|^2 - |z|^2}{|v'-v_* + z|^2}.
$
Then, 
since $b(t) \lesssim (1-t)^{-\frac{n-1}{2}-s}$ from \eqref{kernelQ}, we have
$$
b \left(\ang{ \frac{v'-v_* + z}{|v'-v_* + z|}, \frac{v'-v_* - z}{|v'-v_* - z|}} \right)
\lesssim
\left(\frac{ |z|^2}{|v'-v_* + z|^2} \right)^{-\frac{n-1}{2}-s}.
$$
We thus conclude that 
   \begin{multline}
\left| \kappa_{j}^{\psi, \phi}(v',v_*) \right|
   \lesssim
    \frac{C_\psi C_\phi}{|v'-v_*| } \,
             \exp\left\{ -\frac{c}{2} |v'-v_*|^2 -\frac{c}{2} |v'-v_* + 2 \ang{v_*, \omega} \omega|^2 \right\} \\ \times
                 \left( \int_{\omega^\bot} ~ d\pi_{z} ~   | v'-v_* + z|^{\gamma +2s + 1}                 \frac{ \chi_j (|z|) }{|z|^{n-1+2s}}
                 e^{ - 2 c |z + \left(v_* - \ang{v_*, \omega} \omega \right)|^2 }  \right).
                 \label{symmetricUPPER}
   \end{multline}
We will estimate this upper bound. Without loss of generality, in the following for notational simplicity we set the unimportant constant to $c = 1/4$.

Suppose $\gamma + 2s  \ge -(n-2)$. 
For the integral over $z \in \omega^\bot$, with $j\le 0$, we have
  \begin{multline*}
\int_{\omega^\bot} ~ d\pi_{z} ~   | v'-v_* + z|^{\gamma +2s + 1}     ~            \frac{ \chi_j (|z|) }{|z|^{n-1+2s}}  ~
                 e^{ - \frac{|z + \left(v_* - \ang{v_*, \omega} \omega \right)|^2}2 } 
                 \\
                 =
      \int_{\omega^\bot} ~ d\pi_{z} ~  \left( | v'-v_*|^2 + |z|^2\right)^{(\gamma +2s + 1)/2}     ~            \frac{ \chi_j (|z|) }{|z|^{n-1+2s}}  ~
                 e^{ - \frac{|z + \left(v_* - \ang{v_*, \omega} \omega \right)|^2}2 }      
                                  \\
                 \lesssim
      \int_{\omega^\bot} ~ d\pi_{z} ~  \ang{  v'-v_* + z}^{\gamma +2s + 1}     ~            \frac{ \chi_j (|z|) }{|z|^{n-1+2s}}  ~
                 e^{ - \frac{|z + \left(v_* - \ang{v_*, \omega} \omega \right)|^2}2 }.            
 \end{multline*}
We have used that $| z| \ge 1$ when $j\le 0$.
The upper bound just above is
  \begin{multline*}
 = \int_{\omega^\bot}~ d\pi_{\bar{z}}~ \ang{v'-v_* + \bar z - \left(v_* - \ang{v_*, \omega} \omega \right)}^{\gamma + 2s +1}
~ e^{ -|\bar z|^2/2 } 
 \\
 \times
 \frac{ \chi_j (|\bar z - \left(v_* - \ang{v_*, \omega} \omega \right)|) }{|\bar z - \left(v_* - \ang{v_*, \omega} \omega \right)|^{n-1+2s}}.
 \end{multline*}
 Now we have the uniform, pointwise inequality (valid for any $\rho \in \R$)
 \[
\ang{v'-v_* + \bar z - \left(v_* - \ang{v_*, \omega} \omega \right)}^{\rho} \lesssim \ang{v_* - \ang{v_*, \omega} \omega}^{\rho} \ang{v'-v_*}^{|\rho|} \ang{\bar{z}}^{|\rho|}.
\]
With that, we observe that the previous displayed quantity is bounded above by a uniform constant times
 \begin{multline*}
\ang{v_* - \ang{v_*, \omega} \omega }^{\gamma + 2s +1} \, \ang{v'-v_*}^{\left| \gamma + 2s +1\right|}
   \\
   \times
 \int_{\omega^\bot}~ d\pi_{\bar{z}}
 ~ \frac{ \chi_j (|\bar z - \left(v_* - \ang{v_*, \omega} \omega \right)|) }{|\bar z - \left(v_* - \ang{v_*, \omega} \omega \right)|^{n-1+2s}}  
~  e^{ -|\bar z|^2/4 } .
 \end{multline*}
The essential new {\it claim} is the following  uniform estimate:
 \begin{multline}
 \int_{\omega^\bot}~d\pi_{\bar{z}} ~
  \frac{ \chi_j (|\bar z - \left(v_* - \ang{v_*, \omega} \omega \right)|)   }{|\bar z - \left(v_* - \ang{v_*, \omega} \omega \right)|^{n-1+2s}}  e^{ -|\bar z|^2/4 } 
\\
   \lesssim 2^{2sj} \ang{v_* - \ang{v_*, \omega} \omega }^{-(n-1)}.
\label{goodKJestimate}
 \end{multline}
Modulo \eqref{goodKJestimate}, the combination of the estimates above shows that
   \begin{multline*}
\left| \kappa_{j}^{\psi, \phi}(v',v_*) \right|
   \lesssim
    \frac{2^{2sj}C_\psi C_\phi}{|v'-v_*| } \,
             \exp\left\{ -\frac{|v'-v_*|^2}8 -\frac{|v'-v_* + 2 \ang{v_*, \omega} \omega|^2}8 \right\} 
             \\ 
             \times
\ang{v_* - \ang{v_*, \omega} \omega}^{\gamma + 2s-(n-2)} \ang{v'-v_*}^{\left| \gamma + 2s +1 \right|}.
   \end{multline*}
 We thus deduce that
 \begin{multline*}
 \int_{\R^n} dv' ~\left| \kappa_{j}^{\psi, \phi}(v',v_*) \right| w^{\ell}(v') dv' 
 \\
\lesssim
C_\psi C_\phi ~
2^{2sj} 
     \int_{\R^n} ~ dv' ~ \ang{v_* - \ang{v_*, \omega} \omega}^{\gamma + 2s -(n-2)} ~|v'-v_*|^{-1} \, w^{\ell}(v')~  
     \ang{v'-v_*}^{\left| \gamma + 2s +1 \right|}
     \\ 
            \times \exp\left\{ -\frac{|v'-v_*|^2}8 -\frac{|v'-v_* + 2 (v_* \cdot \omega) \omega|^2}8 \right\}.
 \end{multline*}
This leads directly to
 \[  
 \int_{\R^n} \left| \kappa_{j}^{\psi, \phi}(v',v_*) \right|  w^{\ell}(v') dv' \lesssim 2^{2sj} w^{\ell}(v_*) J',
 \]
with  
 \begin{multline*}
 J' \eqdef
\int_{\R^n} dv'
\ang{v_* - \ang{v_*, \omega} \omega }^{\gamma + 2s -(n-2)} \, \ang{v'-v_*}^{\left| \gamma + 2s +1 \right|}
\frac{w^{| \ell |}(v'-v_*)}{ |v'-v_*|} 
      \\
            \times  \exp\left\{ -\frac{|v'-v_*|^2}8 -\frac{|v'-v_* + 2 \ang{v_*, \omega} \omega|^2}8 \right\} \, .
 \end{multline*}
Notice  that up to a change in parameters this is the same $J'$ which appeared in Mouhot-Strain \cite[p.525]{MR2322149}.
Therein the following basic pointwise estimate is shown: $J' \lesssim \ang{v_*}^{\gamma + 2s -(n-1)}$.   
This completes the estimate, once we prove \eqref{goodKJestimate}. 

We now establish the key estimate \eqref{goodKJestimate}.
Since $j\le 0$ and $\chi_j$ is supported on $[2^{-j-1}, 2^{-j}]$, we see that 
$ 
|\bar z - \left(v_* - \ang{v_*, \omega} \omega \right)| \approx 2^{-j}  \gtrsim 1.
$  
Thus in particular, since also $| \chi_j |_{L^\infty} \lesssim 1$,  we have the following estimate
 \begin{multline}
 \int_{\omega^\bot}~d\pi_{\bar{z}} ~
  \frac{ \chi_j (|\bar z - \left(v_* - \ang{v_*, \omega} \omega \right)|)   }{|\bar z - \left(v_* - \ang{v_*, \omega} \omega \right)|^{n-1+2s}}  e^{ -|\bar z|^2/4 } 
  \\
   \lesssim 2^{2sj}  \int_{\omega^\bot}
  \frac{ d\pi_{\bar{z}} ~e^{ -|\bar z|^2/4 }  }{\ang{\bar z - \left(v_* - \ang{v_*, \omega} \omega \right)}^{n-1}}  
     \lesssim 2^{2sj}  \ang{v_* - \ang{v_*, \omega} \omega }^{-(n-1)}.
\notag
 \end{multline}
 This concludes our proof of the new  {\it claim} in \eqref{goodKJestimate}, and the proof.  
\end{proof}

With the compact estimate for the kernel from Lemma \ref{prop:Grad} in hand, we can automatically prove the main estimate for the compact term \eqref{tminushRAPplus} below.  Using similar methods, but without Lemma \ref{prop:Grad}, we will also prove 
\eqref{tminushRAPplus2} and \eqref{posUPPer3norm}.  

 \begin{proposition}\label{prop:GradRap}
 Consider $\phi$ satisfying \eqref{rapidDECAYfcn}.
We have the uniform estimates
\begin{equation}
 \left| T^{k,\ell}_{+}(g,\phi,f) \right| 
 \lesssim 
  C_\phi ~ 
 2^{2sk} ~ 
 \nsm w^\ell g\nsm_{L^2_{\gamma+2s-(n-1)}}  \nsm w^\ell f\nsm_{L^2_{\gamma+2s-(n-1)}}.
 \label{tminushRAPplus}
\end{equation}
These hold for any $k \leq 0$, $\ell \in \mathbb{R}$.  
Additionally for any $m\ge 0$ and any $k$ we obtain
\begin{align}
  \left| T^{k,\ell}_{+}(g,f,\phi) \right| 
 \lesssim 
  C_\phi ~ 
 2^{2sk} ~ 
 \nsm g\nsm_{L^2_{-m}}  \nsm f\nsm_{L^2_{-m}}.
 \label{tminushRAPplus2}
\end{align}
Furthermore 
\begin{equation} 
\left|  T^{k,\ell}_{+} (\phi,h,f)\right| 
\lesssim  
C_\phi ~   2^{2sk} 
 \nsm w^{\ell} h\nsm_{L^2_{\gamma + 2s}}   \nsm w^{\ell} f\nsm_{L^2_{\gamma + 2s}}.
 \label{posUPPer3norm}
\end{equation}
Each of these estimates hold for all parameters in \eqref{kernelQ} and \eqref{kernelP}. 
 \end{proposition}

\begin{proof}
For \eqref{tminushRAPplus}, 
using the formula in \eqref{defTKLcarl} and
Cauchy-Schwartz we have directly
\begin{multline*}
 \left| T^{k,\ell}_{+}(g,\phi,f) \right| 
 \lesssim
 \int_{\mathbb{R}^n} dv' ~ w^{2\ell}(v')~ \left| f(v') \right| ~ 
\int_{\mathbb{R}^n} dv_*  ~ \left| g(v_*) \right| ~ \left| \kappa_{k}^{\psi, \phi}(v',v_*) \right|
\\
 \lesssim
  \left(
 \int_{\mathbb{R}^n} dv' ~ w^{2\ell}(v')~ |f(v')|^2 ~ 
\int_{\mathbb{R}^n} dv_*  ~ \left|  \kappa_{k}^{\psi, \phi}(v',v_*) \right|
\right)^{1/2}
\\
\times
 \left(\int_{\mathbb{R}^n} dv_* ~ |g(v_*)|^2 ~ 
\int_{\mathbb{R}^n} dv'  ~ w^{2\ell}(v')
~ \left|  \kappa_{k}^{\psi, \phi}(v',v_*) \right|
\right)^{1/2}.
\end{multline*}
Above we consider the kernel $\kappa_{k}^{\psi, \phi}(v',v_*)$ with $\psi = M_{\beta}$ and $\phi$ as in \eqref{rapidDECAYfcn}.  Now we observe that 
Lemma \ref{prop:Grad} immediately implies \eqref{tminushRAPplus}.

To prove
\eqref{tminushRAPplus2}, notice that the bound \eqref{rapidDECAYfcn} implies there is strong exponential decay in both variables $v$ and $v_*$.  The estimate \eqref{tminushRAPplus2} then follows from the Schur test for integral operators or simply using Cauchy-Schwartz as above.

The proof of \eqref{posUPPer3norm} follows from a similar application of Cauchy-Schwartz.    The main difference is that this time we have strong exponential decay in both $v'_*$ and $v_*$ which displays the appropriate symmetry under the needed pre-post collisional change of variable.
\end{proof}

This completes our basic compact estimates.  In the next section, we develop the geometric Littlewood-Paley theory adapted to the paraboloid in $n+1$ dimensions.

\section{The $n$-dimensional non-isotropic Littlewood-Paley decomposition}
\label{sec:aniLP}

In this section, we  expand the development of the non-isotropic Littlewood-Paley projections that we initiated in \cite{gsNonCut1}.  The principal new result that is needed in this direction is an understanding of the interaction (specifically, the commutator) between the geometric Littlewood-Paley projections and the {\it isotropic} velocity derivatives $\partial_{\beta}$.  These additional estimates are necessary because of the higher derivatives present in the norm
$\nspace$, which are ultimately necessary because the singularity of the kinetic factor is strong enough that many of the $L^2$ estimates in \cite{gsNonCut1} must be replaced with $L^\infty$ estimates.  These $L^\infty$ estimates, in turn, are related back to $L^2$ based spaces via Sobolev embedding.

Throughout this section, we will use the variables $v$ and $v^\prime$ to refer to independent points in $\R^n$, meaning that we will not assume in this section that they are related by the collisional geometry.  The reason we choose to use these variable names is that they give a hint about where the Littlewood-Paley projections will be later applied in situations which do involve the collisional geometry explicitly.

\subsection{Definitions and comparison to the non-isotropic norm}

Following the procedures and notation of \cite{gsNonCut1}, we first describe the lifting of vectors $v \in \threed$ to the paraboloid in $\R^\last$.  Specifically, for $v \in \threed$, let $\ext{v} \eqdef (v,\frac{1}{2} |v|^2) \in \R^\last$, and consider the mappings $\tau_v : \threed \rightarrow \threed$  and $\ext{\tau}_v : \threed \rightarrow \R^\last$ given by
\[ \tau_v u \eqdef u - (1- \ang{v}^{-1}) \ang{v,u} |v|^{-2} v,
\quad \mbox{ and } 
\quad \ext{\tau}_v u \eqdef (\tau_v u,  \ang{v}^{-1} \ang{v,u}). \]
These mappings should be thought of as sending the hyperplane $v_\last = 0$ 
to the hyperplane tangent to the paraboloid $(v, \frac{1}{2} |v|^2)$ at the point $\ext{v}$.  It's routine to check that $\ang{v, \tau_v u} = \ang{v}^{-1} \ang{v,u}$ and $|\tau_v u|^2 = |u|^2 - \ang{v}^{-2} \ang{v,u}^2$, which implies
$$
|\ext{\tau}_v u|^2
=
|\tau_v u|^2 + \ang{v, \tau_v u}^2 = |u|^2,
$$ 
meaning that $\ext{\tau}_v$ is an isometry from one hyperplane to the other. 
Moreover, it is not a hard calculation to check that
\[ \ext{v + \tau_v u} = \ext{v} + \ext{\tau}_v u + \frac{1}{2} |u|^2 e_\last, \]
where  $e_\last \eqdef (0,\ldots,0,1)$.  Next, fix any $C^\infty$ function supported on the unit ball of $\R^\last$ and consider the generalized Littlewood-Paley projections given by
\begin{align}
P_j f(v) & \eqdef \int_{\threed} dv' 2^{\dim j} \varphi(2^j(\ext{v} - \ext{v'})) \ang{v'} f(v') \label{lppj} \\
Q_j f(v) & \eqdef P_j f(v) - P_{j-1} f(v), \qquad j \geq 1 \nonumber \\
 & = \int_{\threed} dv' 2^{\dim j} \psi(2^j(\ext{v} - \ext{v'})) \ang{v'} f(v'), \label{lpqj}
\end{align}
where $\psi(w) \eqdef \varphi(w) - 2^{-\dim} \varphi(w/2)$. 
Here $P_j$ corresponds to the usual projection onto frequencies at most $2^{j}$ and $Q_j$ corresponds to the usual projection onto frequencies comparable to $2^{j}$ (recall that the frequency $2^{j}$ corresponds to the scale $2^{-j}$ in physical space).  We also define $Q_0 \eqdef P_0$ to simplify notation.  In order for these projections to be generally useful, the function $\varphi$ must be chosen to satisfy various cancellation conditions which we will discuss later.   For now, we note that, as long as the integral of $\varphi$ over any $\dim$-dimensional hyperplane through the origin equals $1$ (which will be the case for any suitably normalized radial function), we have that
$P_j f(v) \rightarrow f(v)$ as $j \rightarrow \infty$ for all sufficiently smooth $f$ and that
\begin{equation} \left( \int_{\threed} dv \left| P_j f(v) \right|^p \ang{v}^\rho \right)^{\frac{1}{p}}  
\lesssim 
\left( \int_{\threed} dv \left| f(v)\right|^p \ang{v}^{\rho} \right)^{\frac{1}{p}}, \label{boundedlp}
\end{equation}
uniformly in $j \geq 0$ for any fixed $\rho \in \mathbb{R}$ and any $p \in [1,\infty)$ (with suitable variants of this inequality also holding for $p = \infty$ as well as for the operators $Q_j$).  

The principal reason for defining our Littlewood-Paley projections in this way is that the particular choice of paraboloid geometry allows us to control the associated square functions by our non-isotropic norm.  This informal idea is made precise in the following proposition (the proof of which may be found in \cite{gsNonCut1}). 
\begin{proposition}
Suppose that $|Q_j (1)(v)| \lesssim 2^{-2j}$ \label{compareprop}
holds uniformly for all $v \in \threed$ and all $j \geq 0$. Then for any $s \in (0,1)$ and any real $\rho$, the following inequality holds:
\begin{equation}
\begin{split}
 \sum_{j=0}^\infty 2^{2sj} & \int_{\threed} dv ~ |Q_j f(v)|^2 \ang{v}^\rho \lesssim \\
& |f|^2_{L^2_\rho} + \int_{\threed} dv \int_{\threed} dv' 
\left( \ang{v}\ang{v'} \right)^{\frac{\rho+1}{2}}
\frac{(f(v) - f(v'))^2}{d(v,v')^{\dim+2s}} {\mathbf 1}_{d(v,v') \leq 1} . 
\end{split}
\label{compareton}
\end{equation}
This is true uniformly for all smooth  $f$.
\end{proposition}

\subsection{Littlewood-Paley commutator estimates}
For convenience, let us abbreviate 
$2^{\dim j} \varphi(2^j w) \eqdef \varphi_j(w)$ and likewise for $\psi_j$.  We seek at this point to relate the corresponding Littlewood-Paley square function to the
 norm $\nspace$ in arbitrary dimensions $n\ge 2$.  This was done in three dimensions for $N^{s,\gamma}$  in our first paper \cite{gsNonCut1}.  For this extension, it will also be necessary to make estimates involving the commutators of the Littlewood-Paley projections with isotropic, $n$-dimensional derivatives (i.e., derivatives which {\it do not} respect the intrinsic geometry).  The commutators themselves, $ \left[ \frac{\partial}{\partial v_i}, Q_j \right] f(v)
 \eqdef  \frac{\partial}{\partial v_i} \left[ Q_jf \right] (v)-  Q_j \left[ \frac{\partial}{\partial v_i}  f\right](v)$, are easy to calculate:
\begin{align*}
 \left[ \frac{\partial}{\partial v_i}, Q_j \right] f(v) = & \int_{\threed} dv' 2^j \left( {\tilde \nabla}_i \psi + v_i {\tilde \nabla}_\last \psi \right)_j(\ext{v} - \ext{v'}) \ang{v'} f(v') \\
& - \int_{\threed} dv' \psi_j (\ext{v} - \ext{v'}) \ang{v'} \frac{\partial f}{\partial v_i'} (v').
\end{align*}
Above ${\tilde \nabla}_i$ is the $i$-th component of ${\tilde \nabla}$.
After an integration by parts, 
\begin{align*}
 \left[ \frac{\partial}{\partial v_i}, Q_j \right] f(v) 
 = & \int_{\threed} dv' 2^j \left({\tilde \nabla}_i \psi + v_i {\tilde \nabla}_\last \psi \right)_j(\ext{v} - \ext{v'}) \ang{v'} f(v') \\
& + \int_{\threed} d v' \frac{\partial}{\partial v_i'} \left[ \psi_j (\ext{v} - \ext{v'}) \ang{v'} \right] f(v') \\
= & \int_{\threed} dv' 2^j (v_i - v_i') \left( {\tilde \nabla}_\last \psi \right)_j(\ext{v} - \ext{v'}) \ang{v'} f(v') \\
& +  \int_{\threed} d v' \psi_j(\ext{v} - \ext{v'}) v_i' \ang{v'}^{-1} f(v').
\end{align*}
In particular, this commutator may be written as
\[ \left[ \frac{\partial}{\partial v_i}, Q_j \right] f(v) = \tilde Q_j f(v) + Q_j \tilde f(v), \]
where $\tilde Q_j$ is given by replacing $\psi(w)$ with $w_i {\tilde \nabla}_\last \psi(w)$ in the integral \eqref{lpqj} and $\tilde f(v') \eqdef v'_i \ang{v'}^{-2} f(v')$.  The end result is that, after taking $\beta$ derivatives and studying these commutators, we may always  write $\partial_{\beta} 2^{-|\alpha|j} {\tilde \nabla}^{\alpha} Q_j$ as a finite sum
$$ 
\partial_{\beta} 2^{-|\alpha|j} {\tilde \nabla}^{\alpha} Q_j f(v) 
= 
\sum_{|\beta_1| \leq |\beta|}
\sum_{k} c_{k,\beta_1}^{\alpha,\beta} Q^{k}_j ( \omega_k \partial_{\beta_1} f)(v), 
$$
where each $Q^{k}_j$ has a form the same as \eqref{lpqj} for some $\psi^k$, the derivatives $\beta_1$ satisfy 
$|\beta_1| \leq |\beta|$, and the weights $\omega_k(v')$ are polynomials times powers of $\ang{v'}$ which together tend to zero at infinity.  Above ${\tilde \nabla}$ represents the $n+1$ dimensional gradient which is restricted to the paraboloid, $(v, \frac{1}{2}|v|^2)$,  post differentiation; since this is the way in which that particular operation is used in the sequel.

We are able to compare these weighted non-isotropic Littlewood-Paley projections to the non-isotropic norm
using the same method which we used to prove  Proposition \ref{compareprop} in \cite{gsNonCut1},  that is,
 by completing the square.  We bound above the expression
\[ 
\int_{\threed} \! dv \int_{\threed} \! dv' \int_{\threed} dz \ \psi^k_j(\ext{v} - \ext{z}) \psi^{k}_j(\ext{v'} - \ext{z}) (f(v) - f(v'))^2 \ang{z}^{\rho} \ang{v} \ang{v'} \omega_k(v) \omega_k(v'), 
\]
by integrating over $z$ and comparing this to the semi-norm piece of $\nspace$, then we show that this term is also equal to
\[ - 2 \int_{\threed} dv ~ |Q_j^{k} (\omega_k f)(v)|^2 \ang{v}^{\rho} + 2 \int_{\threed} dv ~
Q_j^{k}(\omega_k)(v) Q_j^{k}(\omega_k f^2)(v) \ang{v}^{\rho},
 \]
(in both expressions, we suppressed the $\partial_{\beta_1}$ acting on $f$).  As long as one has the following uniform inequality
 for all $j \geq 0$ and all $v$,
 \begin{equation}
 |Q_j^{k}(\omega_k)(v)| \leq 2^{-2j},
 \label{uniformQineq}
\end{equation}
 these estimates may be multiplied by $2^{2sj}$ 
 and summed over $j$ to obtain
\begin{multline*}
 \sum_{j=0}^\infty 2^{2sj} \int_{\threed}  dv ~ |Q_j^{k} (\omega_k f)(v)|^2 \ang{v}^{\rho} \lesssim 
  |f|_{L^2_{\rho}}^2
 \\
 + \int_{\threed} dv \int_{\threed} dv' \frac{(f'-f)^2}{d(v,v')^{\dim+2s}} \left(\ang{v}\ang{v'} \right)^{\frac{\rho+1}{2}} {\mathbf 1}_{d(v,v') \leq 1},
 \quad 
 s \in (0,1).
\end{multline*}
The derivation of this inequality is similar to that of Proposition \ref{compareprop} (which has the advantage of simpler notation), so we remind the reader of its proof now:
\begin{proof}[Proof of Proposition \ref{compareprop}]
For any $j \geq 1$, one has the equality
\begin{align*}
 \frac{1}{2} \int_{\threed}  \! \! \! dv & \! \int_{\threed}  \! \! \! dv' \! \int_{\threed} \! \! \! dz ~  (f(v) - f(v'))^2 \psi_j(\ext{z} - \ext{v}) \psi_j(\ext{z} - \ext{v'}) \ang{v} \ang{v'} \ang{z}^\rho  \\ 
 & =   - \int_{\threed} \! \! dv  ~
 ([Q_j f] (v))^2 \ang{v}^\rho + \int_{\threed}  \! \! \! dv \!  \int_{\threed} \! \! \! dz   (f(v))^2 \psi_j(\ext{z} - \ext{v}) Q_j(1)(z) \ang{v} \ang{z}^\rho,
\end{align*}
simply by expanding the square $(f(v) - f(v'))^2$ and exploiting the symmetry of the integral in $v$ and $v'$ (note also that the corresponding statement holds true for $Q_0$ when $\psi$ is replaced by $\varphi$).
By our assumption on the projections $Q_j$, namely $|Q_j(1)(z)| \lesssim 2^{-2j}$, we may control the second term on the right-hand side by
\begin{align*}
2^{-2j} \int_{\threed} dv ~ (f(v))^2 \ang{v}^{\rho}.
\end{align*}
This  bound follows from the change of variables 
$z \mapsto v + 2^{-j} \tau_v u$, which 
is a change the variable from $z$ to $u$ and has Jacobian $\ang{v}^{-1}$,
so that
\begin{align*}
 \int_{\threed} dz |\psi_j(\ext{z} - \ext{v})| \ang{z}^{\rho} & =
\ang{v}^{-1} \int_{\threed} du ~ |\psi( \ext{\tau}_v u + 2^{-j-1} |u|^2 e_\last)| \ang{v + 2^{-j} \tau_v u}^\rho \\
& \lesssim \ang{v}^{\rho-1}.
\end{align*}
We have used that $\ang{v} \approx \ang{z}$ on the support of $|\psi_j(\ext{z} - \ext{v})|$ since $j$ is non-negative.
Moreover, the same change of variables can be used to show that
\[ 
\int_{\threed} dz ~ |\psi_j(\ext{z} - \ext{v})| |\psi_j(\ext{z} - \ext{v'})| \ang{z}^\rho \lesssim 2^{\dim j} (\ang{v} \ang{v'})^{\frac{\rho+1}{2}}. 
\]
Here we use the inequality $|\psi_j(\ext{z} - \ext{v'})| \lesssim 2^{\dim j}$ and observe that $\ang{v} \approx \ang{z} \approx \ang{v'}$ on the support of the original integral.
Moreover, the triangle inequality guarantees that the integral is only nonzero when $d(v,v') \leq 2^{-j+1}$, so that we have
\[ 
2^{2sj} \int_{\threed} dz |\psi_j(\ext{z} - \ext{v})| |\psi_j(\ext{z} - \ext{v'})| \ang{z}^\rho 
\lesssim 
(\ang{v} \ang{v'})^{\frac{\rho+1}{2}} 
~ 2^{2sj} 2^{nj} {\mathbf 1}_{d(v,v') \leq 2^{-j+1}}. 
\]
These estimates may  be summed over $j \geq 1$ (when $v \ne v'$) because the sum terminates after some index $j_0$ with $2^{-j_0} < d(v,v') \leq 2^{-j_0+1}$, 
 yielding \eqref{compareton}, since 
\[ 
\sum_{j=1}^\infty 2^{2sj} 2^{nj} {\mathbf 1}_{d(v,v') \leq 2^{-j+1}} 
=
\sum_{j=1}^{j_0} 2^{2sj} 2^{nj} {\mathbf 1}_{d(v,v') \leq 2^{-j+1}} 
\lesssim
d(v,v')^{-\dim-2s} {\mathbf 1}_{d(v,v') \leq 1}. 
\]
The last inequality follows from $2^{(2s+n){j_0}} \lesssim d(v,v')^{-\dim-2s}$.  
 The remaining term $j=0$ is already bounded above by $|f|_{L^2_{\rho}}^2$.
\end{proof}

Thus, following the proof of Proposition \ref{compareprop},
subject only to the establishment of the 
 decay condition \eqref{uniformQineq} similar to the hypothesis of Proposition \ref{compareprop}, it follows that
\begin{equation}
 \sum_{|\beta| \leq K} \sum_{j=0}^\infty 2^{2(s-|\alpha|)j} \int_{\threed} dv |\partial_{\beta} {\tilde \nabla}^\alpha Q_j f (v)|^2 
 \ang{v}^{\gamma + 2s} w^{2\ell - 2 |\beta|}(v)
\lesssim  |f|_{\nspace}^2,
\label{lpsobolev0}
\end{equation}
also holds for any multi-index of derivatives $\alpha$ on $\R^\last$ and any fixed $K$ and $\ell$.  
The catch, of course, is that the functions $\psi^k$ become increasingly difficult to control when either $|\alpha|$ or $K$ becomes large.  This means that the uniform estimate is increasingly difficult to obtain.  In the next section, we will establish the desired inequality contingent on the following cancellation condition; that $\psi^k$ satisfies
\begin{equation*}
 \int_{\threed} du ~ p(u) ({\tilde \nabla}^\alpha \psi^k)(\ext{\tau}_v u) = 0,
\end{equation*}
for all polynomials $p$ of degree $1$
 and all multi-indices $|\alpha| \leq 1$.  Since $\psi^k$ is itself related to the original $\psi$ by taking a sequence of ${\tilde \nabla}$-derivatives and multiplying by a polynomial, it is sufficient to choose the original $\varphi$  so that $\psi$ satisfies
\begin{equation}
 \int_{\threed} du ~ p(u) ({\tilde \nabla}^\alpha \psi)(\ext{\tau}_v u) = 0, \label{lpcancel}
\end{equation}
for all polynomials of degree at most $M$ and all $|\alpha| \leq M$, where $M$ is any fixed but arbitrary natural number.

\subsection{Selection of $\varphi$ and inequalities for smooth functions}

Let $D$ be the dilation on functions in $\R^\last$ given by $D \varphi(w) \eqdef \varphi \left(\frac{w}{2} \right)$. 
We choose a radial function $\varphi_0 \in C^\infty(\R^\last)$  which is supported on the ball $|w| \leq R$, with $R>0$, and satisfies
\begin{equation} 
\int_{\threed} du \ \varphi_0(\ext{\tau}_v u) = 1, \qquad \forall v \in \threed. 
\label{lpnormalize}
\end{equation}
Since $\varphi_0$ is radial this equality will be true for all $v$ if it is true for any single $v$.  If $p$ is any homogeneous polynomial on $\threed$, a simple scaling argument shows that
\[ 
\int_{\threed} du ~ p(u) ({\tilde \nabla}^\alpha \circ D \varphi_0) (\ext{\tau}_v u) = 2^{\dim + \deg p - |\alpha|} \int_{\threed} du ~ p(u) ({\tilde \nabla}^\alpha \varphi_0)( \ext{\tau}_v u). 
\]
Next, fix some large integer $M$ and consider 
the $C^\infty$ function $\varphi$ which is supported on the ball of radius $2^{2M} R$ and is given by 
\[ \varphi \eqdef \left( \prod_{|k| \leq M, k \neq 0} \frac{I - 2^{-\dim + k} D}{1-2^k} \right) \varphi_0.  \]
By induction on $M$ and the scaling argument just mentioned, it follows that $\varphi$ satisfies exactly the same normalization condition as $\varphi_0$, namely \eqref{lpnormalize}.  Assuming that $M$ is fixed, the radius $R$ may be chosen so that $2^{2M} R \leq 1$.  For this fixed $M$, we take the corresponding function $\varphi$ to be the basic building block of our Littlewood-Paley projections, i.e., we use this $\varphi$ in the definition \eqref{lppj}, and use $\psi \eqdef \varphi - 2^{-\dim} D \varphi$ in \eqref{lpqj}. By our particular choice of $\varphi$, this $\psi$ satisfies the cancellation conditions \eqref{lpcancel} for all polynomials $p$ of degree at most $M$ and all multi-indices $\alpha$ of order at most $M$.
Consequently, we have the following integral estimates:

\begin{lemma}
Choose a large integer $M$.  Suppose that \eqref{lpcancel} holds for all polynomials of degree at most $M$ and all multi-indices $\alpha$ of order at most $M$.  Fix any multi-index $\alpha$ with $\last$ components.  If $f$ is a smooth function on $\R^\dim$,
and $g$ is some non-negative function satisfying for all $v \in \threed$ that
\[ 
\sup_{d(v,v') \leq 1} \sum_{|\beta| \leq k+1} | (\partial_{\beta} f) (v')| \leq C_f ~ g(v), 
\]
Suppose 
$k$ is
 some integer
 satisfying $-1 \leq k \leq \frac{1}{2}(M - |\alpha|)$.  Then the  inequality
\begin{equation}
 \left| 2^{-|\alpha| j} {\tilde \nabla}^{\alpha} Q_j f(v) \right| \lesssim 2^{-(k+1)j} C_f ~ g(v),\label{lpsmooth}
\end{equation}
holds uniformly in $f$, $v$, and $j \geq 0$.
\end{lemma}

\begin{proof}
We first consider the case $\alpha = 0$.
We study the integral $I_j(v)$ given by
\begin{align*} 
I_j(v) & \eqdef \int_{\threed} dv' \psi_j(\ext{v} - \ext{v'}) \omega(v') \\
 & = \ang{v}^{-1} \int_{\threed} du \ \psi( \ext{\tau}_v u + 2^{-j-1} |u|^2 e_ \last) \omega(v + 2^{-j} \tau_v u),
\end{align*}
where the equality follows after the change $v' \mapsto v + 2^{-j} \ext{\tau}_v u$.  If we expand the integrand in powers of $2^{-j}$ by means of Taylor's theorem, we have  an asymptotic expansion of this integral, with the coefficient of $2^{-kj}$ equaling
\[ 
\ang{v}^{-1} \frac{1}{k!}\int_{\threed} du \left(  \left. \frac{d^k}{d \epsilon^k} \right|_{\epsilon=0} \psi( \ext{\tau}_v u + \frac{\epsilon}{2} |u|^2 e_ \last) \omega(v + \epsilon \tau_v u) \right). 
\]
This can subsequently be expanded as a sum of terms, each of which is an integral of some derivative of $\psi$ times a polynomial in $u$.  The order of differentiation is at most $k$, and the degree of the polynomial is at most $2k$.  Consequently, if $2k \leq M$, the $k$-th term in the asymptotic series will vanish identically.  Using the integral form of the remainder in Taylor's theorem, it follows that $I_j(v)$ is exactly equal to 
 \[ \ang{v}^{-1} \frac{1}{k!} \int_{\threed} du \int_0^1 d \theta ~ (1- \theta)^{k} \frac{d^{k+1}}{d \theta^{k+1}} \left[ \psi( \ext{\tau}_v u + 2^{-j-1} \theta |u|^2 e_ \last) \omega(v + 2^{-j} \theta \tau_v u) \right], \]
for any $k \leq \frac{M}{2}$.  
From here, it is elementary to see that
\[ 
\left| I_j(v) \right| \lesssim 2^{-(k+1)j} \ang{v}^{-1} \sup_{d(v,v')  \leq 2^{-j} \leq 1} \sum_{|\beta| \leq k+1}  | (\partial_{\beta} \omega )(v')|, 
\]
uniformly for all $v$ and all $j \geq 0$.  Setting $\omega(v') = \ang{v'} f(v')$ establishes the result for $\alpha = 0$.  When $\alpha \neq 0$, the effect of including an additional $2^{-|\alpha| j} {\tilde \nabla}^\alpha$ is to replace $\psi$ with ${\tilde \nabla}^\alpha \psi$ in the definition of $I_j$.  The proof follows exactly as before, where now one only has \eqref{lpcancel} for derivatives up through order $M - |\alpha|$.
\end{proof}

This lemma 
may now be directly applied to obtain
 the uniform bounds \eqref{uniformQineq} on $|Q_j^k (\omega_l)(v)|$ as  required for Proposition \ref{compareprop} and \eqref{lpsobolev0} to be true.

\section{Upper bounds for the trilinear form}
\label{sec:upTRI}

 In this section, we finish off the proofs of Lemma \ref{sharpLINEAR}, Lemma \ref{estNORM3}, 
  Lemma \ref{NonLinEst}, and   Lemma \ref{DerCoerIneq}
  from Section \ref{mainESTsec}.  The main tools we use in each of these proofs are the 
  estimates from 
  Section \ref{physicalDECrel} with the decomposition of the singularity therein, and also a decomposition of the functions themselves via the non-isotropic Littlewood-Paley theory developed in 
  Section \ref{sec:aniLP}.
  We begin with a variant of Lemma \ref{NonLinEst}. 

\subsection{The main non-linear estimates}
We will now prove  non-linear estimates in the velocity norms in Lemma \ref{NonLinEstA}. 
They are somewhat technical to state because we would like to distribute negative index decaying velocity weights among the post-collisional velocities \eqref{sigma}.  This is explained above Proposition \ref{referLATERprop}.  We have

\begin{lemma}
\label{NonLinEstA}
(Non-linear estimate)
Consider the non-linear term \eqref{gamma0} and \eqref{DerivEst}.

For any multi-index $\beta$, any $\ell^+, \ell^-, \ell' \ge 0$ with $\ell = \ell^+ - \ell^-$ and $\ell' \le \ell^-$ we have
\begin{multline}
 |\ang{w^{2\ell}\Gamma_{\beta}( g,  h), f} | 
 \lesssim 
  \nsm w^{\ell^+ - \ell'} g\nsm_{L^2} 
  \nsm h\nsm_{N^{s,\gamma}_{\ell + \ell', \ksob} }  \nsm f\nsm_{N^{s,\gamma}_\ell}     
  \\
 + 
 \nsm w^{\ell^+ - \ell'} g\nsm_{L^2_{\gamma + 2s}} \nsm w^{\ell + \ell'} h\nsm_{L^2} 
\nsm f\nsm_{N^{s,\gamma}_\ell}.
 \label{nlineq1}
\end{multline}
Here $\ell + \ell' = \ell^+ - (\ell^- - \ell' )$.
We alternatively use the Sobolev embedding on $g$:
\begin{multline}
 |\ang{w^{2\ell}\Gamma_{\beta}( g,  h), f} | 
 \lesssim 
\nsm  g\nsm_{H^{\ksob}_{{\ell^+} - \ell'}}  
  \nsm h\nsm_{N^{s,\gamma}_{\ell + \ell'} }  \nsm f\nsm_{N^{s,\gamma}_\ell}
  \\
 + 
 \nsm w^{\ell^+ - \ell'} g\nsm_{L^2_{\gamma + 2s}} \nsm w^{\ell + \ell'} h\nsm_{L^2} 
\nsm f\nsm_{N^{s,\gamma}_\ell}.
\label{nlineq}
\end{multline}
\end{lemma}

Of course, these estimates immediately imply 
Lemma \ref{NonLinEst}, as we explain just now.
Consider derivatives of the non-linear term as in \eqref{DerivEst}; a typical term is
\[ \left|\ang{w^{2\ell - 2 | \beta|} \Gamma_{\beta_2} (\partial^{\alpha - \alpha_1}_{\beta - \beta_1}g,\partial^{\alpha_1}_{\beta_1}h), \partial^{\alpha}_{\beta} f} \right|. \]
Note that 
since 
$|\alpha| + |\beta|  \leq K$
we also have
$|\alpha_1| + |\beta_1| + |\alpha - \alpha_1| + |\beta - \beta_1| \leq K$.  To estimate this term, we will apply either \eqref{nlineq1} or \eqref{nlineq}; in the former case, the right-hand side applies an additional $\ksob =  \lfloor \frac{n}{2} +1 \rfloor$ velocity derivatives to $\partial^{\alpha_1}_{\beta_1} h$ (a consequence of Sobolev embedding), and in the latter, the same $\ksob$ velocity derivatives are applied instead to $ \partial^{\alpha-\alpha_1}_{\beta - \beta_1} g$.  We will make the choice of \eqref{nlineq1} versus \eqref{nlineq} so that the total number of derivatives on either $h$ or $g$ does not exceed $K$.

To that end, consider the situation when $|\alpha_1| + |\beta_1| \leq K - 2 \ksob$.
  Applying \eqref{nlineq1} 
with $\ell^+ = \ell \ge 0$, $\ell^- = |\beta|$  and $\ell' = |\beta - \beta_1|$, so $\ell^--\ell' = |\beta_1|$, gives the estimate
\begin{multline}
 \left|\ang{w^{2\ell - 2 | \beta|} \Gamma_{\beta_2} (\partial^{\alpha - \alpha_1}_{\beta - \beta_1}g,\partial^{\alpha_1}_{\beta_1}h), \partial^{\alpha}_{\beta} f} \right| 
 \\
 \lesssim 
  \nsm w^{\ell - |\beta - \beta_1|} \partial^{\alpha - \alpha_1}_{\beta - \beta_1}g\nsm_{L^2} 
  \nsm \partial^{\alpha_1}_{\beta_1}h\nsm_{N^{s,\gamma}_{\ell -|\beta_1|, \ksob} }  
  \nsm \partial^{\alpha}_{\beta} f\nsm_{N^{s,\gamma}_{\ell - |\beta|}}     
  \\
 + 
 \nsm w^{\ell - |\beta - \beta_1|} \partial^{\alpha - \alpha_1}_{\beta - \beta_1}g\nsm_{L^2_{\gamma + 2s}} 
 \nsm w^{\ell -|\beta_1|} \partial^{\alpha_1}_{\beta_1}h\nsm_{L^2} 
  \nsm \partial^{\alpha}_{\beta} f\nsm_{N^{s,\gamma}_{\ell - |\beta|}}.
\notag
\end{multline} 
Note that the $\ell$ right here is not the same $\ell$ as the one in Lemma \ref{NonLinEstA}; instead we use Lemma \ref{NonLinEstA} with $\ell$ replaced by $\ell - |\beta|$. 
Now after also integrating over $\mathbb{T}^n$ we can use the Sobolev embedding $H^{\ksob}(\mathbb{T}^n) \subset L^\infty (\mathbb{T}^n)$ applied to the norms of the $h$ terms and a simple Cauchy-Schwartz to establish Lemma \ref{NonLinEst}.  Here we have used this Sobolev embedding twice on the $h$ terms, so the maximum order of differentiation of $h$  on the upper bound will be $|\alpha_1| + |\beta_1| + 2 \ksob \leq K$.

By symmetry, if $|\alpha - \alpha_1| + |\beta - \beta_1| \leq K - 2 \ksob$, we may instead apply \eqref{nlineq} and estimate the $g$ terms on the right-hand side via Sobolev embedding applied to the $x$-variables as in the previous case.
In the remaining situation, we have
\begin{align*}
|\alpha_1| + |\beta_1| & \geq K - 2 \ksob + 1, \\
|\alpha - \alpha_1| + |\beta - \beta_1| & \geq K - 2 \ksob + 1. 
\end{align*}
However, since $|\alpha_1| + |\beta_1| + |\alpha - \alpha_1| + |\beta-\beta_1| \leq K$, it must also be the case that $|\alpha_1| + |\beta_1| \leq 2 \ksob - 1$ and $|\alpha - \alpha_1| + |\beta-\beta_1| \leq 2 \ksob - 1$.  In this situation we will use \eqref{nlineq1} (the Sobolev embedding for the velocity variables is applied to $h$) and estimate the $g$ terms on the right-hand side via Sobolev embedding in the $x$-variables.  For both $g$ and $h$, the total number of derivatives appearing on either term will then be at most $(2 \ksob - 1) + \ksob$.  Therefore, Lemma \ref{NonLinEst} will hold whenever $K \geq 3 \ksob -1$. 

Now we  set about to prove
the two non-linear estimates in
 Lemma \ref{NonLinEstA}.

\begin{proof}[Proof of  Lemma \ref{NonLinEstA}]
We'll write 
$$
f = P_0 f + \sum_{j=1}^\infty Q_j f \eqdef \sum_{j=0}^\infty f_j,
$$ 
and likewise for $h$.  Then we expand the non-linear term as in \eqref{DerivEst} as follows: 
\begin{align}
 & \ang{\Gamma_\beta (g,h),f}  = \sum_{j,j' = 0}^\infty \ang{\Gamma_{\beta}(g,h_{j'}),f_j} \nonumber \\
& \hspace{30pt} = \sum_{j=0}^\infty \ang{\Gamma_{\beta}(g,h_j),f_j} + \sum_{l=1}^\infty \sum_{j=0}^\infty  \left\{ \ang{\Gamma_{\beta}(g,h_{j+l}),f_j} + \ang{\Gamma_{\beta}(g,h_j),f_{j+l}} \right\}. 
\label{mainexpand}
\end{align}
Consider the sum over $l$ of the terms $\ang{\Gamma_{\beta}(g,h_{j+l}),f_j}$ for fixed $j$.  We expand $\Gamma_{\beta}$  by introducing the cutoff around the singularity of $b$ in terms of $T^{k,\ell}_{+}$ and $T^{k,\ell}_{-}$:
\begin{align}
 \sum_{l=1}^\infty \ang{\Gamma_{\beta}(g,h_{j+l}),f_j} 
 & = 
 \sum_{k=-\infty}^\infty  \sum_{l=1}^\infty 
\left\{ T^{k,\ell}_{+}(g,h_{j+l},f_j) - T^{k,\ell}_{-}(g,h_{j+l},f_j) \right\} \nonumber \\
& = 
\sum_{k=-\infty}^0 \left\{ T^{k,\ell}_{+}(g, h - P_j h, f_j) - T^{k,\ell}_{-}(g, h - P_j h, f_j) \right\}  \label{farsing} \\
& \hspace{30pt} 
+ 
\sum_{l=1}^\infty \sum_{k=1}^\infty \left\{ T^{k,\ell}_{+}(g,h_{j+l},f_j) - T^{k,\ell}_{-}(g,h_{j+l},f_j) \right\}. \label{nearsing}
\end{align}
Here we have used the basic telescoping property $h - P_j h =  \sum_{l=1}^\infty h_{j+l}$.
It is worth remarking that throughout the manipulation, the order of summation may be rearranged with impunity since the estimates we will employ below imply that the sum is absolutely convergent when $g,h,f$ are all Schwartz functions.

We will prove first \eqref{nlineq}, and later explain how to similarly obtain \eqref{nlineq1}.
Regarding the terms \eqref{farsing}, the inequalities \eqref{tminusg} and \eqref{tplussmall} dictate that
\begin{multline*}
\sum_{k=-\infty}^0 |T^{k,\ell}_{+}(g, h - P_j h, f_j) - T^{k,\ell}_{-}(g, h - P_j h, f_j)| 
\\
\lesssim 
\left(
 \nsm  g\nsm_{H^{\ksob}_{{\ell^+} - \ell'}} 
\nsm w^{\ell + \ell'} h\nsm_{L^2_{\gamma + 2s}} 
  +
\nsm w^{\ell^+ - \ell'} g\nsm_{L^2_{\gamma + 2s}} \nsm w^{\ell + \ell'} h\nsm_{L^2} 
\right)
  \nsm w^{\ell} f_j\nsm_{L^2_{\gamma + 2s}}.
\end{multline*}
Notice that the upper bound in  \eqref{tplussmall} is bigger than 
the upper bound from \eqref{tminusg}.
We have used $\nsm w^\rho P_j h\nsm_{L^2} \lesssim \nsm w^\rho h\nsm_{L^2}$ $\forall \rho\in\mathbb{R}$, a consequence of \eqref{boundedlp}.  Thus
\begin{multline*}
 \sum_{j=0}^{\infty} \sum_{k=-\infty}^0 |T^{k,\ell}_{+}(g, h - P_j h, f_j) - T^{k,\ell}_{-}(g, h - P_j h, f_j)| 
\\
\lesssim 
\left(
 \nsm  g\nsm_{H^{\ksob}_{{\ell^+} - \ell'}} 
\nsm w^{\ell + \ell'} h\nsm_{L^2_{\gamma + 2s}} 
  +
\nsm w^{\ell^+ - \ell'} g\nsm_{L^2_{\gamma + 2s}} \nsm w^{\ell + \ell'} h\nsm_{L^2} 
\right)
 \left| \sum_{j=0}^\infty 2^{2sj}   \nsm w^{\ell} f_j\nsm_{L^2_{\gamma + 2s}}^2 \right|^\frac{1}{2}
 \\
\lesssim 
\left(
 \nsm  g\nsm_{H^{\ksob}_{{\ell^+} - \ell'}} 
\nsm w^{\ell + \ell'} h\nsm_{L^2_{\gamma + 2s}} 
  +
\nsm w^{\ell^+ - \ell'} g\nsm_{L^2_{\gamma + 2s}} \nsm w^{\ell + \ell'} h\nsm_{L^2} 
\right)
\nsm f\nsm_{N^{s,\gamma}_\ell}.
\end{multline*}
This is just Cauchy-Schwartz.
The favorable comparison of the square-function norm of $f$ to the norm $\nsm  f \nsm_{N^{s,\gamma}_\ell}$ is provided by \eqref{lpsobolev0}.

As for the terms \eqref{nearsing}, when $k \leq j$ a similar approach holds; namely, \eqref{tminusg}  and \eqref{tplussmallK} guarantee that
\begin{multline*}
\sum_{k=1}^{j} \left| T^{k,\ell}_{+}(g,h_{j+l},f_j) - T^{k,\ell}_{-}(g,h_{j+l},f_j) \right| 
\\
\lesssim 2^{2sj} 
\nsm  g\nsm_{H^{\ksob}_{{\ell^+} - \ell'}}  
\nsm w^{\ell + \ell'} h_{j+l}\nsm_{L^2_{\gamma+2s}} 
\nsm w^{\ell}  f_j\nsm_{L^2_{\gamma+2s}}, 
\end{multline*}
(we have used the trivial facts that $2^{2sk} = 2^{2sj} 2^{2s(k-j)}$ and 
$\sum_{k=1}^{j} 2^{2s(k-j)}\lesssim 1$). 
In particular, this inequality may be summed over $j$; another application of Cauchy-Schwartz gives that
\begin{multline*}
 \sum_{j=0}^\infty \sum_{k=1}^{j}  \left| T^{k,\ell}_{+}(g,h_{j+l},f_j) - T^{k,\ell}_{-}(g,h_{j+l},f_j) \right|
 \\ 
  \lesssim 2^{-sl} 
\nsm  g\nsm_{H^{\ksob}_{{\ell^+} - \ell'}}  
 \left| \sum_{j=0}^\infty 2^{2s(j+l)} \nsm w^{\ell + \ell'} h_{j+l}\nsm_{L^2_{\gamma+2s}}^2  \right|^{\frac{1}{2}} 
 \left| \sum_{j=0}^\infty 2^{2sj} \nsm w^{\ell}  f_j\nsm_{L^2_{\gamma+2s}}^2 \right|^{\frac{1}{2}}
 \\
 \lesssim 2^{-2sl} \nsm  g\nsm_{H^{\ksob}_{{\ell^+} - \ell'}}  
  \nsm h\nsm_{N^{s,\gamma}_{\ell + \ell'} } \nsm f\nsm_{N^{s,\gamma}_\ell}.
\end{multline*}
This estimate may clearly also be summed over $l \geq 0$.

A completely analogous argument may be used to expand $\Gamma_{\beta}$ for the terms in \eqref{mainexpand} of the form $\ang{\Gamma_{\beta}(g,h_j),f_{j+l}}$ in terms of $T^{k,\ell}_{+} - T^{k,\ell}_{*}$; 
\begin{align}
\sum_{l=1}^\infty \ang{\Gamma_{\beta}(g,h_j),f_{j+l}} & = \sum_{k=-\infty}^\infty \sum_{l=1}^\infty (T^{k,\ell}_{+} - T^{k,\ell}_{*})(g,h_j,f_{j+l}) \nonumber \\
& = \sum_{k=-\infty}^0 (T^{k,\ell}_{+} - T^{k,\ell}_{*})(g,h_j,f - P_j f) \label{sum1} \\
& \hspace{30pt} + \sum_{l=1}^\infty \sum_{k=1}^\infty (T^{k,\ell}_{+} - T_*^k)(g,h_j,f_{j+l}) \label{sum2}.
\end{align}
In this case the estimates \eqref{tstarg}  and \eqref{tplussmall} are used to handle the terms \eqref{sum1} just as the corresponding terms \eqref{farsing} were handled.  Regarding the sum \eqref{sum2}, now \eqref{tstarg} and \eqref{tplussmallK} are used to estimate the sum analogously to the estimates of \eqref{nearsing} for $k\le j$.  The only difference is that the roles of $h$ and $f$ are now reversed.

Recalling the original expansion of $\ang{\Gamma_{\beta}(g,h),f}$ it is clear that the only terms that remain to be considered are the following:
\begin{align} 
\tilde{\Gamma}_{\beta}(g,h,f)  \eqdef 
&  \sum_{l=0}^\infty \sum_{j=0}^\infty \sum_{k=j+1}^\infty \left\{T^{k,\ell}_{+}(g,h_{j+l},f_j) - T^{k,\ell}_{-}(g,h_{j+l},f_j)\right\} 
\label{maincancelf} \\
 & + \sum_{l=1}^\infty \sum_{j=0}^\infty \sum_{k=j+1}^\infty \left\{ T^{k,\ell}_{+}(g,h_{j},f_{j+l}) - T^{k,\ell}_{*}(g,h_{j},f_{j+l}) \right\} . \label{maincancelh}
\end{align}
In other words, we have already established the inequality
\begin{align*}
 \left| \ang{\Gamma_{\beta}(g,h),f} - \tilde{\Gamma}_{\beta}(g,h,f) \right| \lesssim & \ 
\nsm  g\nsm_{H^{\ksob}_{{\ell^+} - \ell'}}  
  \nsm h\nsm_{N^{s,\gamma}_{\ell + \ell'} } \nsm f\nsm_{N^{s,\gamma}_\ell}
 \\
 & + 
 \nsm w^{\ell^+ - \ell'} g\nsm_{L^2_{\gamma + 2s}} \nsm w^{\ell + \ell'} h\nsm_{L^2} 
\nsm f\nsm_{N^{s,\gamma}_\ell}.
\end{align*}
Now the terms on the right-hand side of \eqref{maincancelf} and the terms \eqref{maincancelh} are both treated by the cancellation inequalities.  
The terms \eqref{maincancelf}, for example, are handled by  \eqref{cancelf2}.  For any fixed $l,j,k$, we have 
\begin{align*}
 \left| T^{k,\ell}_{+}(g,h_{j+l},f_j) \right. & \left. - T^{k,\ell}_{-}(g,h_{j+l},f_j) \right|
\\ &
\lesssim 2^{(2s-2)k}  
\nsm g\nsm_{H^{\ksob}_{-m}} 
  \nsm w^\ell h_{j+l}\nsm_{L^2_{\gamma+2s}}
 \nl w^\ell |\tilde{\nabla}|^2 f_{j} \nr_{L^2_{\gamma+2s}}. 
\end{align*}
There is decay of the norm as $k \rightarrow \infty$ since $2s - 2 < 0$.  This yields
\begin{align*}
 \sum_{k=j+1}^\infty  \left| T^{k,\ell}_{+}(g,h_{j+l},f_j) \right. & \left. - T^{k,\ell}_{-}(g,h_{j+l},f_j) \right| 
\\ & \lesssim 
2^{(2s-2)j}  
\nsm g\nsm_{H^{\ksob}_{-m}} 
  \nsm w^\ell h_{j+l}\nsm_{L^2_{\gamma+2s}}
 \nl w^\ell |\tilde{\nabla}|^2 f_{j} \nr_{L^2_{\gamma+2s}}.
\end{align*}
Again Cauchy-Schwartz is applied to the sum over $j$.  
In this case $2^{(2s-2)j}$ is written as 
$2^{(s-1)j} 2^{s(j+l)} 2^{-sl} 2^{-j} $; the first factor goes with $f_j$, the second with $h_{j+l}$, the third and fourth remain for the sums over $l$ and $j$ respectively.  Once again \eqref{lpsobolev0} is employed.  The factors of $2^{-sl}$ and $2^{-j}$ allow us to finish the sums over $l$ and $j$.

The desired bound for the non-linear term is completed by performing summation of the terms \eqref{maincancelh}.  The pattern of inequalities is exactly the same as the one just described, this time using \eqref{cancelh2g2}.  In particular, one has that
\begin{align*}
 \left| T^{k,\ell}_{+}(g,h_{j},f_{j+l}) \right. & \left. - T^{k,\ell}_{*}(g,h_{j},f_{j+l}) \right|
\\ &  \lesssim 2^{(2s - 2)k} \nsm g\nsm_{H^{\ksob}_{-m}}  \nsm w^\ell f_{j+l}\nsm_{L^2_{\gamma+2s}} 
\nl w^\ell |\tilde{\nabla}|^2 h_j \nr_{L^2_{\gamma+2s}}.
\end{align*}
Again, since $2s -2 < 0$,  this leads to the following inequality for the sum over $k$:
\begin{align*}
 \sum_{k=j+1}^\infty  \left| T^{k,\ell}_{+}(g,h_{j+l},f_j) \right. & \left. - T^{k,\ell}_{-}(g,h_{j+l},f_j) \right| 
\\ & 
\lesssim 2^{(2s - 2)j} \nsm g\nsm_{H^{\ksob}_{-m}}  \nsm w^\ell f_{j+l}\nsm_{L^2_{\gamma+2s}} 
\nl w^\ell |\tilde{\nabla}|^2 h_j \nr_{L^2_{\gamma+2s}}.
\end{align*}
The same Cauchy-Schwartz estimate is used for the sum over $j$; there  is exponential decay allowing the sum over $l$ to be controlled.  
In the estimates above $m\ge 0$ can be taken arbitrarily large, yet $-m=\ell$ is sufficient.
The end result is precisely \eqref{nlineq}.

We can obtain \eqref{nlineq1}  using the same techniques.   We expand all of the sums from \eqref{mainexpand} in the same way.  
Then we estimate 
\eqref{farsing} using the inequalities \eqref{tminush} and \eqref{tplussmall2}.  We control
\eqref{nearsing}, when $k \leq j$, with \eqref{tminush}  and \eqref{tplussmall2K}.  Then 
\eqref{tstarh}  and \eqref{tplussmall2} are used to handle the terms \eqref{sum1}.  
 Now \eqref{tstarh} and \eqref{tplussmall2K} are used to estimate the sum 
\eqref{sum2} for $k\le j$.  For the cancellations,
\eqref{maincancelf} are handled by  \eqref{cancelf21}.
Lastly,
\eqref{maincancelh} is controlled using \eqref{cancelh2g1}. 
Otherwise, the pattern of summing these inequalities is exactly the same as the one just described above.
\end{proof}

This concludes our main non-linear estimates. 

\subsection{The Compact Estimates}
Here we collect some estimates for the linearized collision operator.
The first one is the key to the estimate for operator $K$ in \eqref{compactupper}.

\begin{lemma}
\label{CompactEst}
(Compact Estimate)
For any $\ell \in \mathbb{R}$, we have the uniform estimate
\begin{equation}
\notag
\left| \langle  w^{2\ell}  Kg, h \rangle  \right|
 \lesssim 
 \nsm w^{\ell} g\nsm_{L^2_{\gamma+2s - \delta }} \nsm w^{\ell} h\nsm_{L^2_{\gamma+2s - \delta}},
 \quad
\delta = \min\{2s, (n-1)\}.
 \end{equation}
\end{lemma}

Since $\delta>0$ above Lemma \ref{CompactEst} easily implies \eqref{compactupper}  when $g = h$.  To see this, first apply Cauchy's inequality with $\frac{\eta}{2}$ to the upper bound in Lemma \ref{CompactEst}:
$$
\left| \langle  w^{2\ell}  Kg, g \rangle  \right|
 \le
 \frac{\eta}{2} \nsm w^{\ell} g\nsm_{L^2_{\gamma+2s - \delta }}^2 + C_\eta \nsm w^{\ell} g\nsm_{L^2_{\gamma+2s - \delta}}^2.
$$  
For the term  $C_\eta \nsm w^{\ell} g\nsm_{L^2_{\gamma+2s - \delta}}^2$ above, we split into $|v| \ge R$ and $|v|\le R$.  Choosing $R>0$ sufficiently large so that $C_\eta R^{-\delta} \le \frac{\eta}{2}$ proves \eqref{compactupper} subject only to Lemma \ref{CompactEst}.  Now Lemma \ref{CompactEst}  and other estimates will follow from \eqref{coerc1ineqPREP}  below.

\begin{proposition}
\label{upperBds}
For any function $\phi$ satisfying \eqref{rapidDECAYfcn}, we have the estimate
\begin{equation}
\label{coerc1ineqNORM}
\left| \ang{w^{2\ell}\Gamma_\beta (\phi,h),f} \right|
\lesssim
 \nsm  h \nsm_{N^{s,\gamma}_\ell}  \nsm  f \nsm_{N^{s,\gamma}_\ell}.
\end{equation}
For the next two estimates, we suppose that $\phi$ further satisfies \eqref{derivESTa}.  Then
\begin{equation}
\label{coerc1ineqPREP}
\left| \ang{w^{2\ell}\Gamma_\beta(g,\phi),f} \right|
 \lesssim 
 \nsm w^\ell g \nsm_{L^2_{\gamma + 2s-(n-1)}}  \nsm w^\ell f \nsm_{L^2_{\gamma + 2s-(n-1)}}.
\end{equation}
For any $m\ge 0$ we also have
\begin{equation}
\label{coerc1ineqPREP2}
\left| \ang{w^{2\ell}\Gamma_\beta (g,f),\phi} \right| \lesssim 
 \nsm  g \nsm_{L^2_{-m}}  \nsm  f \nsm_{L^2_{-m}}.
\end{equation}
Each of these estimates hold for any $\beta$, and any $\ell \in \mathbb{R}$.
\end{proposition}

Notice that these imply several other previously-stated estimates.  In particular, 
Lemma \ref{CompactEst} is a special case of \eqref{coerc1ineqPREP}  and
Pao's estimate of $\nu_K(v)$  in \eqref{compactpiece}.
Thus Proposition 
 \ref{upperBds} implies Lemma \ref{sharpLINEAR}, since also 
 \eqref{normupper} follows directly from \eqref{coerc1ineqNORM}.
 Other uses of Proposition 
 \ref{upperBds} will be seen below.

\begin{proof}[Proof of Proposition \ref{upperBds}]
To prove \eqref{coerc1ineqNORM}, we expand $\ang{w^{2\ell}\Gamma_\beta(\phi,h),f}$
as in \eqref{mainexpand}, and the proof follows the same lines as the proof of \eqref{nlineq}.
Following that proof,
 we estimate 
\eqref{farsing}
and 
\eqref{nearsing}, when $k \leq j$, using the inequalities \eqref{tminusg} and \eqref{posUPPer3norm}.
Then 
\eqref{tstarg} and \eqref{posUPPer3norm}
are 
 used to handle the terms \eqref{sum1} and \eqref{sum2} for $k\le j$.  
For the cancellations,
\eqref{maincancelf} are handled by  \eqref{cancelf2}  and
\eqref{maincancelh} is controlled  using \eqref{cancelh2g2}. 
The pattern of summing these inequalities is exactly the same as the proof of \eqref{nlineq}.

To prove the estimate in \eqref{coerc1ineqPREP} we will use the inequality 
\eqref{lpsmooth}.  In particular 
\begin{multline*}
\left| \ang{w^{2\ell}\Gamma_\beta(g,\phi),f} \right| 
=   
\left|
\sum_{j=0}^\infty \sum_{k=-\infty}^\infty  
\left\{ T^{k,\ell}_{+}(g,\phi_j,f) - T^{k,\ell}_{-}(g,\phi_j,f) \right\}
\right| 
\\
\lesssim
 \nsm w^\ell g \nsm_{L^2_{\gamma + 2s-(n-1)}}  \nsm w^\ell f \nsm_{L^2_{\gamma + 2s-(n-1)}}
  \sum_{j=0}^\infty \sum_{k=-\infty}^\infty  \min\{2^{(2s-2)k} , 2^{2sk}  \} 2^{-2j}.
\end{multline*}
We have used
\eqref{tminushRAP} and 
\eqref{tminushRAPplus}
with Proposition \ref{compactCANCELe}. Specifically, in each of those estimates 
$\phi_j$ satisfies 
\eqref{derivESTa} with $C_\phi \lesssim 2^{-2j}$, as in \eqref{lpsmooth}.

The estimate for 
$
\ang{w^{2\ell}\Gamma_\beta(g,f),\phi}
$
in \eqref{coerc1ineqPREP2}
is proved
in 
exactly the same way
using instead
\eqref{tminushRAP},
\eqref{tminushRAPplus2},
and Proposition \ref{compactCANCELe}.
 In particular, 
Proposition \ref{upperBds} follows.  
\end{proof}

This concludes our compact estimates.

\subsection{The coercive estimates}   In this sub-section we will prove the coercive interpolation inequalities in
 \eqref{coerc1ineq}
and
 \eqref{coerc2ineq} from Lemma \ref{DerCoerIneq}, and Lemma \ref{estNORM3}.  Actually,  \eqref{coerc2ineq} is a trivial consequence of Lemma \ref{estNORM3} and \eqref{compactupper}
 because 
$
\langle w^{2\ell} Lf, f \rangle = \langle w^{2\ell} Nf, f \rangle + \langle w^{2\ell} Kf, f \rangle.
$  
Thus we will restrict attention to \eqref{coerc1ineq}  and
 Lemma \ref{estNORM3}.

\begin{proof}[Proof of Lemma \ref{estNORM3}]
Firstly,  Lemma \ref{estNORM3} for $\ell = 0$ and $C=0$ is proven in 
\cite[Lemma 4]{gsNonCut1}; note that \cite[Lemma 4]{gsNonCut1} is only stated for $n=3$ dimensions and a more restrictive class of collision kernels than \eqref{kernelQ} and \eqref{kernelP}, however the proof  directly extends  to the case of $n\ge 2$ dimensions and all of the collision kernels $\gamma + 2s > -(n-2)$.
We focus here on estimating the norm piece when $\ell \ne 0$.  We expand
\begin{gather*}
\langle w^{2\ell} Ng, g \rangle 
=  
| g |_{B_{\ell}}^2 
+
\int_{\mathbb{R}^n} dv ~ w^{2\ell}(v)  \nu(v) ~ |g(v)|^2
+ J,
\end{gather*}
where $| g |_{B_{\ell}}$ is defined in \eqref{normexpr}.
 Furthermore, we have 
  the following equivalence 
  $$ 
  | g |_{B_{\ell}}^2+ | w^\ell g |_{L^2_{\gamma + 2s}}^2 \approx | g |_{N^{s,\gamma}_{\ell}}^2.
  $$
  The lower bound $\gtrsim$ of this equivalence follows directly from the proof in 
   \cite[Section 6]{gsNonCut1}.  The upper bound, $\lesssim$, follows from the estimates for $\ang{w^{2\ell}\Gamma (M,g),g}$ in \eqref{coerc1ineqNORM} when combined with the interpolation inequality that we prove just now.

The error term, $J$, above takes the form
\begin{gather*}
  J \eqdef
  \frac{1}{2} \int_{\mathbb{R}^n} dv  \int_{\mathbb{R}^n} dv_* \int_{\sph} d \sigma B (g'-g)g\left( w^{2\ell}(v') - w^{2\ell}(v) \right) M_*' M_* .
\end{gather*}
These expressions are derived exactly as in the computations preceding \eqref{normpiece}.

We will show that this error term $J$ is lower order via an expansion of the kernel.  
In particular we {\it claim} that there exists an $\epsilon>0$ such that
$$
\left| J \right| \lesssim | g |_{B_{\ell}}
|w^\ell g|_{L^2_{\gamma + 2s - \epsilon}}.
$$
This estimate implies that $J$ satisfies the same upper bound as \eqref{compactupper}, which easily implies Lemma \ref{estNORM3} by following the procedure which is explained below Lemma \ref{CompactEst}.

To prove this {\it claim}, notice that Cauchy-Schwartz gives us
$$
\left| J \right| \lesssim 
| g |_{B_{\ell}}
\left(
\int_{\mathbb{R}^n} dv  \int_{\mathbb{R}^n} dv_* \int_{\sph} d \sigma B |g|^2
\frac{\left( w^{2\ell}(v') - w^{2\ell}(v) \right)^2}{w^{2\ell}(v)} M_*' M_*
\right)^{1/2}.
$$
We will show in particular that there exists an $\epsilon>0$ such that
\begin{equation}
 \int_{\mathbb{R}^n} dv_* \int_{\sph} d \sigma ~ B(v - v_*, \sigma) 
\frac{\left( w^{2\ell}(v') - w^{2\ell}(v) \right)^2}{w^{2\ell}(v)} ~  M_*' M_*
\lesssim 
w^{2\ell}(v)
\ang{v}^{\gamma + 2s - \epsilon},
\label{claimEST1}
\end{equation}
and this will establish the {\it claim}.

To obtain \eqref{claimEST1}, first a simple 
Taylor expansion yields
$$
w^{2\ell}(v') - w^{2\ell}(v) 
=
(v' - v)\cdot (\nabla w^{2\ell})(\gamma(\tau)), 
\quad
\exists \tau \in [0,1],
$$
where $\gamma(\tau) = v + \tau(v' - v)$.  Since $|v' - v| = |v'_* - v_*|$
we have
$$
\ang{\gamma(\tau)} \lesssim \ang{v} \ang{v' - v} \lesssim \ang{v} \ang{v'_* - v_*},
$$
and similarly,
$
\ang{\gamma(\tau)}^{-1} 
 \lesssim \ang{v}^{-1}   \ang{v'_* - v_*}. 
$
Thus generally
$$
\frac{\left( w^{2\ell}(v') - w^{2\ell}(v) \right)^2}{w^{2\ell}(v)} ~  M_*' M_*
\lesssim
|v' - v|^2
w^{2\ell}(v)\ang{v}^{-2} \sqrt{M_*' M_*}.
$$
Now we split $|v' - v|^2 = |v' - v|^{2s + \delta} |v'_* - v_*|^{2 -2s - \delta}$ for any $\delta \in (0, 2-2s)$.  We can expand
$
|v' - v|^{2s + \delta} = |v - v_*|^{2s + \delta}\left( \sin \frac{\theta}{2} \right)^{2s + \delta},
$
and then we clearly have
$$
\int_{\sph} d \sigma ~ B(v - v_*, \sigma) ~ |v' - v|^{2s + \delta}\lesssim |v - v_*|^{\gamma+2s + \delta}.
$$
This follows directly from \eqref{kernelQ} and \eqref{kernelP}. Furthermore, 
$
|v'_* - v_*|^{2 -2s - \delta}
(M_*' M_*)^{1/4}
\lesssim 1.
$
Putting all of this together, we see that \eqref{claimEST1} holds with $\epsilon = 2- \delta >0$.
 \end{proof}
 
With the help of our coercive estimate \eqref{coerc2ineq} with no derivatives, in the following we will prove the main coercive estimate with high derivatives.

\begin{proof}[Proof of \eqref{coerc1ineq}]
We use the formula for $Lg$ from \eqref{LinGam}.  As in \eqref{DerivEst}, we expand
\begin{multline*}
\partial^{\alpha}_{\beta} Lg =
 L\left( \partial^{\alpha}_{\beta} g \right)
 - 
  \sum_{\beta_1 + \beta_2  = \beta, ~|\beta_1|< |\beta|}  C^{\beta}_{\beta_1, \beta_2} ~ 
\Gamma_{\beta_2} (\partial_{\beta-\beta_1} M,\partial^{\alpha}_{\beta_1} g) 
\\
 - 
  \sum_{\beta_1 + \beta_2  = \beta, ~|\beta_1|< |\beta|}  C^{\beta}_{\beta_1, \beta_2} ~ 
\Gamma_{\beta_2}(\partial^{\alpha}_{\beta_1} g, \partial_{\beta-\beta_1}M).
\end{multline*}
After multiplying by $w^{2\ell - 2|\beta|}\partial^{\alpha}_{\beta} g$, and integrating over $\mathbb{R}^n$ we can estimate the term
$
\langle w^{2\ell - 2|\beta|} L \left( \partial^{\alpha}_{\beta} g \right), \partial^{\alpha}_{\beta} g \rangle
$
 as in \eqref{coerc2ineq}.  For the error term that arises, i.e., $| \partial_\beta^\alpha  g |_{L^2(B_R)}$ for some $R>0$ we use the compact interpolation for any small $\delta >0$:
$$
| \partial_\beta^\alpha  g |_{L^2(B_R)} 
\le 
\eta  | \partial_\beta^\alpha  g |_{H^\delta(B_R)}
+
C   | \partial^\alpha g |_{L^2(B_{C})}.
$$
Here $C>0$ is some large constant, and $\eta>0$ is any small number.  Further  $| \partial_\beta^\alpha  g |_{H^\delta(B_R)} \lesssim \nsm \partial^{\alpha}_{\beta}g\nsm_{N^{s,\gamma}_{\ell - |\beta|}}$ when $\delta < s$; this holds because the non-isotropy of the norm $\nspace$ only comes into play near infinity. More precisely, if $v$ and $v'$ are confined to the Euclidean ball of radius $R$ at the origin, then $d(v,v') \approx |v-v'|$ (with constants depending on $R$), and so on this region the expression for \eqref{normdef} is comparable to the Gagliardo-type semi-norm for the space $H^s(B_R)$.
This gives the estimate for $L\left( \partial^{\alpha}_{\beta} g \right)$.

We estimate the rest of the terms using \eqref{coerc1ineqNORM} and \eqref{coerc1ineqPREP}.  We use the extra velocity decay which is left over from these estimates to split into a large unbounded region times a small constant and a bounded region with a large constant.  On the bounded region, we use the compact interpolation as just used in the last case to put all of the velocity derivatives into slightly larger Sobolev norm multiplied by an arbitrarily small constant.  This is all that is needed to finish the estimate.
 \end{proof}

\section{De-coupled space-time estimates, global existence and rapid decay} 
\label{sec:deBEest}

We will now proceed to show that the  estimates proven in all of the previous sections can be used to establish global existence as stated in Theorem \ref{mainGLOBAL}. The methodology that we employ, which goes back to Guo \cite{MR2000470},  essentially de-couples the  required space-time estimates that are needed from the new fractional and non-isotropic derivative estimates which are shown in the previous sections.   This works precisely because of the specific structure of the interactions between the velocity variables and the space-time variables.  Our key new contribution in this direction is to identify the unique weighted geometric fractional Sobolev norm \eqref{normdef}, and prove sharp 
non-isotropic geometric fractional derivative estimates such as those  stated in 
Lemma \ref{sharpLINEAR},
Lemma \ref{estNORM3},
Lemma \ref{NonLinEst},
and
Lemma \ref{DerCoerIneq}. 
These can then be included into the general de-coupled energy method in order to prove global existence and decay.
  We point the reader's attention also to the general abstract framework of  \cite{villani-2006}.

We furthermore remark that the whole space case of our main Theorem \ref{mainGLOBAL}, when $x\in \mathbb{R}^n$ replaces $x\in \mathbb{T}^n$, will follow  from our new velocity estimates in the previous sections when combined with other known whole space cut-off methods.  
This would only require re-writing this last section, using instead the cut-off energy methods in the whole space such as for instance \cite{MR2095473,MR2259206,MR1057534}. 

We will initially discuss the coercivity of the linearized collision operator, $L$.  
With the null space \eqref{null}, and the projection \eqref{hydro}, we 
 decompose  $f(t,x,v)$ as
\begin{equation}
f={\bf P}f+\{{\bf I-P}\}f.
\label{hyrdoSPLIT}
\end{equation}
We prove a sharp constructive lower bound for the linearized collision operator.  

\begin{theorem}
\label{lowerN}  
There is a constructive constant 
$\delta_0>0$ 
 such that
\begin{equation*}
\langle L h, h \rangle
\ge
\delta_0 
| \{ {\bf I - P } \} h |_{N^{s,\gamma}}^2.
\end{equation*}
\end{theorem}

This coercive lower bound is proven with our new constructive compact estimates from Lemma \ref{sharpLINEAR} and Lemma \ref{estNORM3} 
when used in conjunction with
the non-sharp but constructive bound from Mouhot \cite{MR2254617} for the non-derivative part of the norm.

\begin{proof}[Proof of Theorem \ref{lowerN}] 
Suppose $h = \{ {\bf I - P } \} h$.     From \eqref{coerc2ineq}, for some small $\eta >0$
$$
\langle L h, h \rangle
\ge 
\eta |  h |_{N^{s,\gamma}}^2
- C |  h |_{L^{2}(B_C)}^2, 
\quad \exists C \ge 0.
$$ 
Here $B_C$ is the ball of radius $C$ centered at the origin.  The positive constant $C$  is explicitly computable.  From \cite{MR2254617}, it is known that under our assumptions 
$$
\langle L h, h \rangle
\ge 
\delta_1  |  h |_{L^{2}_\gamma}^2.
$$ 
Here $\delta_1>0$ is an explicitly computable constant, and $\gamma$ is from \eqref{kernelQ}.  

Lastly, for any $\delta \in (0,1)$ we employ the splitting
$$
\langle L h, h \rangle
=
\delta \langle L h, h \rangle
+
(1-\delta)\langle L h, h \rangle
\ge 
\delta\eta |  h |_{N^{s,\gamma}}^2
- \delta C |  h |_{L^{2}(B_C)}^2
+
(1-\delta)\delta_1  |  h |_{L^{2}_\gamma}^2.
$$ 
Since $C$ is finite notice that  $|  h |_{L^{2}(B_C)}^2 \lesssim |  h |_{L^{2}_\gamma}^2$ for any $\gamma \in\mathbb{R}$.
Thus the lemma follows by choosing $\delta>0$ sufficiently small so that the last two terms are $\ge 0$.
\end{proof}

\subsection{Local Existence} 
In this section, we explain the approach from our first paper \cite{gsNonCut1} to prove the {\it a priori} energy estimates  which are the key step in the local existence theorem.   Our local existence proof for \eqref{Boltz} with small data is based on a uniform energy estimate for an iterated sequence of approximate solutions.

Our iteration starts at $f^0(t,x,v)\eqdef f_0(x,v)$.  We   solve for $f^{k+1}(t,x,v)$ such that 
\begin{equation}
\left( \partial _t+v\cdot \nabla _x+N \right) f^{k+1}+K f^k
=
\Gamma (f^k, f^{k+1}),
\quad
 f^{k+1}(0,x,v)=f_0(x,v).  \label{approximate}
\end{equation}
It is standard 
 to show the linear equation \eqref{approximate}
 admits smooth solutions with the same regularity  as a given  $H^K_\ell$ smooth small initial data, and also has a gain of 
 $L^2((0,T); \nspace )$.  This does not create difficulties and can be proven with our estimates.  We explain herein how to establish the {\it a priori} estimates necessary to find a local classical solution in the limit as $k\to\infty$.
The key point will be to obtain a uniform estimate for the iteration on 
a small time interval. The crucial energy estimate is as follows:

\begin{lemma}
\label{uniform}  The sequence \{$f^k(t,x,v)\}_{k\ge 0}$ is well-defined. There exists a short time
$
T^{*} >0,
$
such that for $\|f_0\|_{H}$
sufficiently small, there is a uniform constant $C_0>0$ for which 
\begin{equation}
\sup_{k\ge 0} ~
\sup_{0\le t \le T^{*}} ~ 
\left(
\| f^{k}(t)\|^2_{H}
+
\int_0^t 
d\tau ~ 
\| f^{k}(\tau)\|^2_{N}
\right)
\le
2C_0
\|f_0\|^2_{H}.  
\notag
\end{equation}
\end{lemma}

The proof of this estimate is exactly the same as the proof in \cite{gsNonCut1} using the estimates in Section \ref{mainESTsec}, see \cite{gsNonCut1} for the details.
With our uniform control over the iteration \eqref{approximate} granted by Lemma \ref{uniform}, we have the following local existence theorem.  

\begin{theorem}(Local Existence)
\label{local}
For any sufficiently small $M_0>0,$ there exists a time  $T^{*} = T^{*}(M_0)>0$ and an $M_1>0,$ such that if 
\[
\| f_0 \|_{H}^2\le M_1, 
\]
then there is a unique  solution $f(t,x,v)$ to \eqref{Boltz} on  
$[0,T^{*})\times \mathbb{T}^n\times \mathbb{ R}^n$ such that 
\[
\sup_{0\le t\le T^{*}}
\mathcal{ G}(f(t))
=
\sup_{0\le t\le T^{*}}
\left(
\| f(t)\|^2_{H}
+
\int_0^t 
d\tau ~ 
\| f(\tau)\|^2_{N}
\right)
\le M_0. 
\]
Furthermore $\mathcal{ G}(f(t))$ is continuous over $[0,T^{*}).$ Lastly, we have positivity in the sense that
 if
 $F_0(x,v)=\mu +\mu
^{1/2}f_0\ge 0,$ then 
$
F(t,x,v)=\mu +\mu ^{1/2}f(t,x,v)\ge 0. 
$
\end{theorem}

For the proof we refer to the argument in \cite{gsNonCut1}, the case here is  the same.

\subsection{Coercivity estimates for solutions to the non-linear equation}
The next step is to prove a general statement of the linearized H-Theorem, which manifests itself as a 
coercive inequality.  These types of coercive estimates for the linearized collision operator were originally proved by Guo \cite{MR2000470,MR2013332} in the hard-sphere and cut-off regime.  In \cite{gsNonCut1}, we proved this estimate for the Boltzmann equation without 
any angular cut-off, as in \eqref{kernelQ}, but with $p>3$ or $\gamma + 2s >0$.   The next theorem extends this estimate to the full range of $p>2$ and more generally $\gamma + 2s \ge -(n-2)$ in $n$ dimensions.

\begin{theorem}
\label{positive}  
Given the initial data $f_0 \in  \spaceh\left(\mathbb{T}^n \times \mathbb{R}^n \right)$ with derivatives of order $\NgE$ and weights of order $\ell \ge 0$.  Suppose that $f_0$ satisfies \eqref{conservation} initially and the assumptions of Theorem \ref{local}.  For the corresponding local solution, $f(t,x,v)$, to \eqref{Boltz} which  satisfies \eqref{conservation},
there exists a small constant $M_0 >0$ such that if
\begin{gather}
\| f(t) \|^2_{H} \le M_0,
\label{smallL}
\end{gather}
then, further, there are constructive constants $\delta>0$ and $C_2>0$ such that
$$
\sum_{|\alpha| \le K} \| \{ {\bf I - P } \} \partial^\alpha f (t)  \|_{_{N^{s,\gamma}}}^2
\ge 
\delta
\sum_{|\alpha| \le K} \| { \bf  P  } \partial^\alpha f (t)  \|_{_{N^{s,\gamma}}}^2- C_2\frac{d\mathcal{I}(t)}{dt},
$$
where $\mathcal{I}(t)$ is the ``interaction functional'' defined precisely in \eqref{mainINTERACTION} below.
\end{theorem}

This Theorem \ref{positive} is proven in exactly the same way as in \cite{gsNonCut1}, using the usual decomposition of
 ${\bf P}f$'s coefficients $a^f(t,x)$, $b^f_i(t,x)$ and 
$c^f(t,x)$, for a solution to \eqref{Boltz} using \eqref{hydro}.  We explain briefly how this works and refer to  \cite{gsNonCut1} for the full details.
We can decompose a solution  $f$ as \eqref{hyrdoSPLIT}   to
obtain the 
macroscopic equations 
\begin{eqnarray*}
\nabla _xc&=& -\partial_t r_c+ l_c+\Gamma_c  \label{c} \\
\partial_t c+\partial_i b_i &=& -\partial_t r_i+ l_i+\Gamma_i  \label{bi} \\
\partial_i b_j+\partial_j b_i &=& -\partial_t r_{ij}+ l_{ij}+\Gamma_{ij} \quad (i\neq j)  \label{bij} \\
\partial_t b_i+\partial _i a &=& -\partial_t r_{bi}+ l_{bi}+\Gamma_{bi}
\label{ai} \\
\partial_t a &=& -\partial_t r_a+ l_a+\Gamma_a.  \label{adot}
\end{eqnarray*}
Above  $\partial_i=\partial _{x_i}$. 
Consider the 
 basis, $\{ e_k \}_{k=1}^{3n+1+n(n-1)/2}$, which consists of
\begin{gather}
\left( v_i|v|^2\sqrt{\mu } \right)_{1\le i\le n}, ~ 
\left(v_i^2\sqrt{\mu } \right)_{1\le i\le n},  ~
\left(v_iv_j\sqrt{\mu } \right)_{1\le i<j\le n}, ~
\left(v_i\sqrt{\mu } \right)_{1\le i\le n}, ~
\sqrt{\mu }.  
\label{base}
\end{gather}
Further, for notational convenience we define the index set to be
$$
\mathcal{M} \eqdef \left\{c, ~i,~ \left( ij \right)_{i \ne j}, ~bi,~ a\left| ~ i, j = 1,\ldots,n\right. \right\}.
$$
This set $\mathcal{M}$ is just the collection of all indices in the macroscopic equations.  Then for $\mu \in  \mathcal{M}$ we have that 
the $l_\mu (t,x)$
are the coefficients of 
$l(\{{\bf I-P}\}f)$, given by
\begin{equation}
  l(\{{\bf I-P}\}f)
\eqdef
-v\cdot \nabla _x (\{{\bf I-P}\} f)
- L (\{{\bf I-P}\} f),
\label{l}
\end{equation}
with respect to the elements of \eqref{base}; similarly, the 
$\Gamma_\mu(t,x)$ 
and
$r_\mu(t,x)$
are the coefficients of $\Gamma(f,f)$ and $\{{\bf I-P}\}f$ respectively.  Each  $r_\mu$, for example, can be expressed as 
$$
r_\mu = \sum_{k=1}^{3n+1+n(n-1)/2} C_k^\mu \langle \{{\bf I-P}\}f, e_k \rangle.
$$
All of the constants $C_k^\mu$  can be computed explicitly.  Furthermore, each of the terms $l_\mu$ and $\Gamma_\mu$ can also be computed similarly.

Now,  we define the total interaction functional as
\begin{gather}
\mathcal{I}(t) 
\eqdef 
\sum_{|\alpha|\le K-1}
\left\{
\mathcal{I}_a^\alpha(t)
+
\mathcal{I}_b^\alpha(t)
+
\mathcal{I}_c^\alpha(t)\right\},
\label{mainINTERACTION}
\end{gather}
which is composed of the following elements:
\begin{gather*}
\mathcal{I}_a^\alpha(t) \eqdef 
\int_{\mathbb{T}^n}  ~ dx ~ \left( \nabla \cdot \partial ^\alpha b \right)\partial^\alpha a(t,x)
+
\sum_{i=1}^n
\int_{\mathbb{T}^n}  ~ dx ~ \partial_i\partial ^\alpha r_{bi} ~
\partial^\alpha a(t,x),
\\
\mathcal{I}_b^\alpha(t) \eqdef 
- 
 \sum_{j\neq i} \int_{\mathbb{T}^n} ~ dx ~ \partial_{j} \partial^\alpha r_{ij}\partial^\alpha b_i,
 \\
\mathcal{I}_c^\alpha(t) \eqdef 
- \int_{\mathbb{T}^n} ~ dx ~ \partial ^\alpha r_c(t,x) ~ \cdot \nabla_x \partial ^\alpha c(t,x).
\end{gather*}
This interaction functional will be comparable to the energy defined below.

With these macroscopic equations, the key to proving Theorem \ref{positive} will be the following lemma.  In fact, once we have this next Lemma \ref{linear}, the proof of Theorem \ref{positive} will follow exactly as in \cite{gsNonCut1} and we refer to that paper for the rest of the details.

\begin{lemma}
\label{linear}
For all of the microscopic terms, $l_\mu$, from the macroscopic equations
\[
\sum_{\mu \in \mathcal{M}} \|l_\mu\|_{H^{K-1}_x}
\lesssim
\sum_{|\alpha| \le K}\|\{{\bf I-P}\} \partial^\alpha f\|_{L^2_{\gamma+2s}(\mathbb{T}^n\times\mathbb{R}^n)}.
\]
Furthermore,
let (\ref{smallL}) be valid for some $M_0>0.$ Then  for $K \ge 2\ksob$ we have
\[
\sum_{\mu \in \mathcal{M}}
\|\Gamma_\mu\|_{H^{K}_x}
\lesssim
\sqrt{M_0}\sum_{|\alpha |\le K} \| \partial ^\alpha f\|_{L^2_{\gamma+2s}(\mathbb{T}^n\times\mathbb{R}^n)}.
\]
\end{lemma}

We are using the abbreviated notation
$L^2_v = L^2(\mathbb{R}^n)$
and
$H^K_x = H^K(\mathbb{T}^n)$.

\begin{proof} 
Recall 
$\{ e_k \}$, the basis in (\ref
{base}).
To estimate the $l_\mu$ terms,  it suffices to estimate  the $H^{K-1}_x$ norm of
$
\langle l(\{{\bf I-P}\}f), e_k \rangle.
$
We  use (\ref{l}) to expand out
$$
\langle \partial^\alpha l(\{{\bf I-P}\}f), e_k \rangle
=
-\langle v\cdot \nabla _x (\{{\bf I-P}\}\partial^\alpha f), e_k \rangle
-\langle L (\{{\bf I-P}\}\partial^\alpha f), e_k \rangle.
$$
Now for any $|\alpha |\le K-1$ by Cauchy-Schwartz we have
\begin{gather*}
\| \langle v\cdot \nabla _x (\{{\bf I-P}\}\partial^\alpha f), e_k \rangle\|_{L^2_x}^2 
\lesssim 
\int_{\mathbb{T}^n\times \mathbb{R}^n} ~ dx dv ~
|e_k(v)| ~ 
|v|^2 ~ |\{{\bf I-P}\}\nabla_x\partial ^\alpha f|^2
\\
\lesssim 
\|\{{\bf I-P}\}\nabla _x\partial^\alpha f\|^2_{L^2(\mathbb{T}^n\times \mathbb{R}^n)}.
\end{gather*}
Here we have used the exponential decay of $e_k(v)$.

It remains to estimate the linear operator $L$.  
Notice that the basis vectors $e_k$ satisfy \eqref{derivESTa}.
With the expression from \eqref{LinGam}
and the estimate \eqref{coerc1ineqPREP2}, we have
\begin{gather*}
\| \langle L (\{{\bf I-P}\}\partial^\alpha f), e_k \rangle\|_{L^2_x}^2 
\lesssim 
\left\| ~\nsm \{{\bf I-P}\}\partial^\alpha f \nsm_{L^2_{\gamma + 2s}}~  \nsm M \nsm_{L^2_{\gamma + 2s}} \right\|^2_{L^2_{x}}
\\
\lesssim 
\left\| \{{\bf I-P}\}\partial^\alpha f   \right\|^2_{L^2_{\gamma+2s}(\mathbb{T}^n\times\mathbb{R}^n)}.
\end{gather*} 
This completes the proof of our estimates for the $l_\mu$ terms.

As in the above, to prove the estimate for $\Gamma_\mu$,
it will suffice to estimate the $H^{K}_x$ norm of $\langle \Gamma(f,f), e_k\rangle$.  We apply 
\eqref{coerc1ineqPREP2}
from
Proposition \ref{upperBds}
 to see that for any $m\ge 0$
\begin{gather*}
\left\| \langle \Gamma(f,f), e_k\rangle\right\|_{H^{K}_x}
\lesssim
\sum_{|\alpha |\le K}\sum_{\alpha_1 \le \alpha}
\left\|
 \nsm \partial^{\alpha - \alpha_1} f \nsm_{L^2_{-m}} \nsm \partial^{\alpha_1} f \nsm_{L^2_{-m}} 
\right\|_{L^2_x}.
\end{gather*}
We  take the supremum 
over the term with fewer derivatives to obtain for $K \ge 2\ksob$
\begin{gather*}
\lesssim
 \| f\|_{L_v^2H^{K}_x}
\sum_{|\alpha |\le K} \| \partial ^\alpha f\|_{L^2_{\gamma+2s}}
\\
\lesssim
\sqrt{M_0}\sum_{|\alpha |\le K} \| \partial ^\alpha f\|_{L^2_{\gamma+2s}}.
\end{gather*}
The last inequalities follow from
 the Sobolev embedding $L^\infty(\mathbb{T}^n) \supset H^{\ksob}(\mathbb{T}^n)$.
\end{proof}

We are now ready to prove global in time solutions to  \eqref{Boltz} exist. 

\subsection{Global Existence and Rapid Decay}  With the coercivity estimate for non-linear local solutions from Theorem \ref{positive}, we show that these solutions must be global with the standard continuity argument which is used in \cite{gsNonCut1}.   Then we will prove rapid decay for the soft potentials, using the method from the cut-off case \cite{MR2209761}. 

A crucial step in this analysis is to prove the following energy inequalities:
\begin{equation}
\frac{d}{dt}\mathcal{E}_{\ell,m}(t)+\mathcal{D}_{\ell,m}(t) \leq C_{\ell, m} \sqrt{\mathcal{E}_{\ell}(t)}\mathcal{D}_{\ell}(t).
\label{MAINeINEQ}
\end{equation}
These  hold
for any $\ell\ge 0$ and $m = 0, 1, \ldots,K$.   We define the ``dissipation rate''   as
\begin{gather}
\notag
\mathcal{D}_{\ell,m} (t)\eqdef 
\sum_{|\beta|\le m}
\sum_{|\alpha | \le K-|\beta|}
\|\partial_\beta^\alpha f(t)\|_{N_{\ell - |\beta|}^{s,\gamma}}^2.
\end{gather}
We also write $\mathcal{D}_{\ell,K} = \mathcal{D}_{\ell} =   \| f(t) \|^2_{N}$.   Furthermore the ``instant energy functional'' 
$
\mathcal{E}_{\ell,m}(t)
$
for a solution
is a high-order norm which satisfies
$$
\mathcal{E}_{\ell,m}(t)
\approx
\sum_{|\beta|\le m}
\sum_{|\alpha |   \le K - |\beta|}
\|w^{\ell - |\beta|}\partial^{\alpha}_{\beta} f(t)\|^2_{L^2 (\mathbb{T}^n \times \mathbb{R}^n)},
\quad
\mathcal{E}_{\ell,K}(t) = \mathcal{E}_{\ell}(t).
$$
We prove this energy inequality for a local solution via a simultaneous induction on both the order of the weights $\ell$ and on the number of velocity derivatives $m$.

The first inductive step is to prove \eqref{MAINeINEQ} for arbitrary spatial derivatives with $\ell=0$ and $|\beta| =0$. 
 We first fix $M_0\le 1$ such that both
Theorems \ref{local}  and \ref{positive} are valid.    
We now  take the spatial derivatives of $\partial^{\alpha}$ 
of \eqref{Boltz} to obtain
\begin{gather}
\label{initial}
\frac{1}{2}\frac{d}{dt}\|f(t)\|^2_{L^2_vH^K_x} 
+
\sum_{|\alpha |\le K}\left( L\partial^{\alpha}f,\partial^{\alpha}f\right) 
=
\sum_{|\alpha |\le K}
\left( \partial^{\alpha} \Gamma (f,f), \partial^{\alpha} f \right) . 
\end{gather}
By Lemma \ref{NonLinEst} we have
$$
\sum_{|\alpha |\le K}
\left( \partial^{\alpha} \Gamma (f,f), \partial^{\alpha} f \right) 
\lesssim
\sqrt{\mathcal{E}_{0}(t)} ~ 
\mathcal{D}_{0}(t).
$$
Now with Theorem \ref{lowerN} and then Theorem \ref{positive} we have
\begin{multline*}
\sum_{|\alpha |\le K}\left( L\partial^{\alpha}f,\partial^{\alpha}f\right) 
\ge \delta_0 \sum_{|\alpha |\le K} \| \{ {\bf I - P } \}\partial^{\alpha} h \|_{N^{s,\gamma}}^2
\\
\ge
\frac{\delta_0}{2} \sum_{|\alpha |\le K} \| \{ {\bf I - P } \}\partial^{\alpha} h \|_{N^{s,\gamma}}^2
+
\frac{\delta_0 \delta}{2} \sum_{|\alpha |\le K} \| {\bf P } \partial^{\alpha} h \|_{N^{s,\gamma}}^2
-\frac{\delta_0C_2 }{2} \frac{d\mathcal{I}(t)}{dt}. 
\end{multline*}
With $\tilde{\delta} \eqdef \min\left\{ \frac{\delta_0}{2} , \frac{\delta_0 \delta}{2} \right\}>0$
and
$C' \eqdef \delta_0C_2 >0$,
we  conclude that
\begin{gather}
\frac{1}{2}\frac{d}{dt}\left\{\|f(t)\|^2_{L^2_vH^K_x} - C'\mathcal{I}(t)\right\} +\tilde{\delta}  \mathcal{D}_{0,0}(t)  
\lesssim
\sqrt{\mathcal{E}_0(t)} \mathcal{D}_0(t).
\nonumber
\end{gather}
Now, by \eqref{mainINTERACTION}, 
for any $C'>0$ we can choose a large constant $C_1>0$  such that
$$
\|f(t)\|^2_{L^2_vH^K_x} 
\le
\left( C_1 + 1 \right)\|f(t)\|^2_{L^2_vH^K_x} - C'\mathcal{I}(t)
\lesssim
 \|f(t)\|^2_{L^2_vH^K_x}. 
$$
Notice $C_1$  only depends upon the structure of the interaction functional and $C'$, but not on $f(t,x,v)$.
We then define the equivalent instant energy functional by
$$
\mathcal{E}_{0,0}(t)  \eqdef \left( C_1 + 1 \right)\|f(t)\|^2_{L^2_vH^K_x} - C' \mathcal{I}(t).
$$
We multiply 
\eqref{initial} 
by $C_1$ and add it to this differential inequality to conclude
\begin{gather}
\frac{d\mathcal{E}_{0,0}(t)}{dt} +\tilde{\delta}   \mathcal{D}_{0,0}(t)  
\le C_*
\sqrt{\mathcal{E}_0(t)} \mathcal{D}_0(t),
\quad 
C_*>0.
\nonumber
\end{gather}
In the last step we have used the positivity of $L \ge 0$.   We have thus established 
\eqref{MAINeINEQ} 
when $\ell = |\beta| =0$.

We turn to the case when $\ell >0$, but still $|\beta| =0$.  We only have pure spatial derivatives. 
With \eqref{coerc2ineq} in  Lemma \ref{DerCoerIneq}, we deduce 
that for a
$C>0$ and $\ell\geq0$
\begin{equation}
\left( w^{2\ell} L\partial^{\alpha}f, \partial^{\alpha}f\right)  
\gtrsim
\frac{1}{2}
\| \partial^{\alpha} f \|_{N^{s,\gamma}_{\ell}}^{2}
-
C\| \partial^{\alpha}f\|_{L^2(B_C)}^{2}.
\label{newlower}
\end{equation}
Take the $\partial^\alpha$ derivative of
\eqref{Boltz}, then take the inner product of both sides with $w^{2\ell}   \partial^{\alpha}f$ and integrate to obtain the following: 
\[
\sum_{|\alpha|\leq K}\left(  \frac{1}{2}\frac{d}{dt}
\|w^\ell \partial^{\alpha} f(t)\|_{L^2}^{2}+\left(  w^{2\ell}L\partial^{\alpha}f,\partial^{\alpha
}f\right)  \right)  
\lesssim
\sqrt{{\mathcal{E}}_{\ell}(t)} {\mathcal{D}}_{\ell}(t).
\]
We have used Lemma \ref{NonLinEst} to estimate the non-linear term.   We apply the coercive lower bound \eqref{newlower}.  Then we add \eqref{MAINeINEQ} for the case $\ell = |\beta| =0$ multiplied by a suitably large constant $C_2$ to the result.  This yields
\[
\frac{d}{dt}{\mathcal{E}}_{\ell,0}(t) 
+
{\mathcal{D}}_{\ell,0}(t)
\lesssim
\sqrt{{\mathcal{E}}_{\ell}(t)} {\mathcal{D}}_{\ell}(t),
\]
where
$
{\mathcal{E}}_{\ell,0}(t) 
\eqdef
\frac{1}{2}
\sum_{|\alpha|\leq K}  
\|w^\ell \partial^{\alpha} f(t)\|_{L^2}^{2}
+
C_2{\mathcal{E}}_{0,0}(t).
$
Since this is indeed an instant energy functional, we have
\eqref{MAINeINEQ} 
when $\ell >0$ and  $|\beta| =0$.

The final step is the case when $\ell \ge 0$, but also $|\beta| =m+1>0$. We suppose that 
\eqref{MAINeINEQ} 
holds for any
 $\ell \ge 0$ and any  $|\beta| \le m$.
We take $\partial_{\beta}^{\alpha}$ of \eqref{Boltz} to obtain
\begin{equation}
\left(\partial_{t}+v\cdot\nabla_{x}\right)\partial_{\beta}^{\alpha}f+\partial
_{\beta}L\partial^{\alpha}f=-
\sum_{|\beta_{1}|=1} C^{\beta}_{\beta_{1}} ~ 
\partial_{\beta_{1}}v\cdot\nabla_{x}\partial_{\beta-\beta_{1}}^{\alpha
}f+\partial_{\beta}^{\alpha}\Gamma(f,f). 
\label{bgeBIG}
\end{equation}
We use Cauchy's inequality for $\eta>0$, since $|\beta_{1}|=1$ we have
\begin{multline*}
\left\vert \left(  w^{2\ell - 2|\beta|} \partial_{\beta_{1}}v\cdot\nabla
_{x}\partial_{\beta-\beta_{1}}^{\alpha}f,\partial_{\beta}^{\alpha}f\right)
\right\vert 
\\
\leq
\| w^{\ell - |\beta| - 1/2}  \partial_{\beta}^{\alpha}f(t)\|_{L^2}
\|w^{\ell - |\beta| + 1/2} \nabla_{x}\partial_{\beta-\beta_{1}}^{\alpha}f(t)\|_{L^2}
\\
\le
\eta
||\partial_{\beta}^{\alpha}f(t)||_{N^{s,\gamma}_{\ell - |\beta|}}^{2}
+
C_{\eta}||\nabla_{x}\partial_{\beta-\beta_{1}}^{\alpha}f||_{N^{s,\gamma}_{\ell - |\beta-\beta_1|}}^{2}.
\end{multline*}
Estimates using this particular trick were already seen in \cite{MR2013332}.

Now we multiply \eqref{bgeBIG}
 with $w^{2\ell -2|\beta|}\partial^{\alpha}f$ and integrate.  We estimate the non-linear term of the result with Lemma 
 \ref{NonLinEst}.  
 With \eqref{coerc1ineq} we estimate from below the linear term $\left(  w^{2\ell - 2|\beta|} \partial_{\beta}\{L\partial^{\alpha}f\}, \partial_{\beta}^{\alpha}f\right)$.
With these inequalities above, we
conclude \eqref{MAINeINEQ} for $|\beta| = m+1$, but only after adding  to the inequality a suitably large constant times  \eqref{MAINeINEQ} for $|\beta| = m$ similar to the previous cases.  This establishes \eqref{MAINeINEQ} in general by induction.
From here we can conclude global existence  using the standard continuity argument that was described in Section 7.3 of \cite{gsNonCut1}.  
It remains to establish the time decay rates for the soft potentials,
using the  argument from \cite{MR2209761}.
Exponential time decay for the soft potentials, as in \cite{MR2366140}, may also be feasible.

If $\|f_0\|_{H}^{2}$ is sufficiently small, from 
\eqref{MAINeINEQ} for $m = K$ and $\ell \ge 0$,  we have 
\begin{equation}
\frac{d}{dt}\mathcal{E}_{\ell}(t)+\delta\mathcal{D}_{\ell}(t) \leq 0,
\quad
\exists \delta >0.
\notag
\end{equation}
For the hard potentials, $\gamma + 2s \ge 0$, exponential time decay follows directly from 
$\mathcal{D}_{\ell}(t)\gtrsim \mathcal{E}_{\ell}(t)$. But for the soft potentials $\gamma + 2s <0$,
the problem is that for fixed $\ell,$ the non-derivative part of the dissipation rate $\mathcal{D}_{\ell}(t)$ is clearly
weaker than the instant energy $\mathcal{E}_{\ell}(t)$.  In particular, we only have $\mathcal{D}_{\ell}(t)\gtrsim \mathcal{E}_{\ell-1}(t)$.

We interpolate with stronger
norms  to overcome this difficulty. Fix $\ell\ge 0$ and $m>0$.  Interpolation  between the weight functions $w^{2\ell-2}(v)$ and $w^{2\ell+2m}(v)$ yields
\[
\mathcal{E}_{\ell}(t)
\lesssim
 \mathcal{E}_{\ell-1}^{m/(m+1)}(t)\mathcal{E}_{\ell+m}^{1/(m+1)}(t)
\lesssim
\mathcal{D}_{\ell}^{m/(m+1)}(t)\mathcal{E}_{\ell+m}^{1/(m+1)}(0).
\]
The last inequality follows from 
$
\mathcal{E}_{\ell+m}(t)
\lesssim
\mathcal{E}_{\ell+m}(0).
$
Then for some $C_{\ell,m}>0,$
\[
\frac{d}{dt}\mathcal{E}_{\ell}(t)+C_{\ell,m}{\mathcal{E}}_{\ell+m}^{-1/m}(0)~ \mathcal{E}_{\ell}^{(m+1)/m}(t)\leq0.
\]
It follows that 
$
-m~ d(\mathcal{E}_{\ell}(t))^{-1/m}/dt
\leq
-C_{\ell,m}\left(  {\mathcal{E}}_{\ell+m}(0)\right)  ^{-1/m}.
$ 
Integrate over $[0,t]$:
\[
m\left( \mathcal{E}_{\ell}(0)\right)^{-1/m}-m\left( \mathcal{E}_{\ell}(t)\right)^{-1/m}\leq-\left(  {\mathcal{E}}%
_{\ell+m}(0)\right)  ^{-1/m} ~ C_{\ell,m} ~ t.
\]
Hence
\[
\left( \mathcal{E}_{\ell}(t)\right)^{-1/m}\geq t\frac{C_{\ell,m}}{m}\left(  {\mathcal{E}}_{\ell+m}(0)\right)  ^{-1/m}+\{\mathcal{E}_{\ell}(0)\}^{-1/m}.
\]
Since we can assume $\mathcal{E}_{\ell}(0)\lesssim {\mathcal{E}}_{\ell+m}(0)$, the rapid decay
thus follows. \hfill {\bf Q.E.D.}

\end{document}